\documentclass[oneside,11pt,reqno]{amsart}
\usepackage[normalem]{ulem}
\newcommand{\stkout}[1]{\ifmmode\text{\sout{\ensuremath{#1}}}\else\sout{#1}\fi}
\usepackage{stmaryrd}
\usepackage{bbm}
\usepackage{tikz}
\usepackage{mathrsfs}
\usepackage{enumerate}
\usepackage{latexsym,amsxtra}
\usepackage{graphicx}
\usepackage{color}
\usepackage{amscd}
\usepackage{amsmath,amsfonts,amsthm,amssymb,mathtools}
\usepackage{latexsym,amsmath}
\usepackage{graphicx,psfrag}
\usepackage{mathabx}
\usepackage[utf8]{inputenc}
\usepackage{hyperref}

\usepackage{xy}
\xyoption{all}
\usepackage{pinlabel}

\usepackage[marginparwidth=.8in,marginparsep=0.15in]{geometry}

\newtheorem{thm}{Theorem}[section]

\newtheorem{lem}[thm]{Lemma}
\newtheorem{cor}[thm]{Corollary}
\newtheorem{pro}[thm]{Proposition}

\newtheorem{introthm}{Theorem}
\newtheorem{introcor}[introthm]{Corollary}
\newtheorem{conj}[introthm]{Conjecture}

\theoremstyle{definition}
\newtheorem{defi}[thm]{Definition}

\newtheorem{rmk}[thm]{Remark}

\DeclareMathOperator{\spinc}{spin^{c}}
\DeclareMathOperator{\scspinc}{sc-spin^{c}}
\DeclareMathOperator{\torspinc}{tor-spin^{c}}
\DeclareMathOperator{\spin}{spin}
\DeclareMathOperator{\sgn}{sign}
\newcommand{\SQ}{Q}

\def\HF {\operatorname{\mathit{HF}}}
\def\HM {\operatorname{\mathit{HM}}}
\def\KH {\operatorname{\mathit{Kh\,}}}

\def\HF {{HF}}
\def\HM {{HM}}
\def\KH {{Kh}\,}

\newcommand{\cptwo}{{\mathbb C}\textup{P}^2}
\newcommand{\cptwobar}{\overline{{\mathbb C}\textup{P}\!}\,^2}
\newcommand{\inte}{\operatorname{int}}
\newcommand{\ab}{\operatorname{ab}}
\newcommand{\ad}{\operatorname{ad}}
\newcommand{\Ad}{\operatorname{Ad}}
\renewcommand{\sf}{\operatorname{sf}}
\newcommand{\sign}{\operatorname{sign}}
\newcommand{\Fix}{\operatorname{Fix}}

\newcommand{\su}{\mathfrak{su}}
\newcommand{\id}{\operatorname{id}}

\newcommand{\ind}{\operatorname{ind}}
\newcommand{\Lef}{\operatorname{Lef}}

\newcommand{\red}{\operatorname{red}}
\newcommand{\hmred}{\HM^{\red}}

\newcommand{\coker}{\operatorname{coker}}

\newcommand{\PD}{\operatorname{PD}\,}
\newcommand{\A}{\mathcal A}
\newcommand{\B}{\mathcal B}

\newcommand{\M}{\mathcal M}

\newcommand{\T}{\mathcal T}
\newcommand{\G}{\mathcal G}
\newcommand{\RR}{\mathcal R}
\newcommand{\F}{\mathbb F}
\newcommand{\Q}{\mathbb Q}
\newcommand{\R}{\mathbb R}
\newcommand{\Z}{\mathbb Z}
\newcommand{\s}{\mathfrak s}

\renewcommand{\H}{\mathcal H}
\newcommand{\lsw}{\lambda_{\rm{SW}}}
\newcommand{\lfo}{\lambda_{\rm{FO}}}

\newcommand{\tr}{\operatorname{tr}}
\renewcommand{\phi}{\varphi}
\newcommand{\irr}{\operatorname{irr}}
\renewcommand{\div}{\operatorname{div}}

\renewcommand{\Fix}{\operatorname{Fix}}

\newcommand{\nLef}{\operatorname{Lef}^{\circ}}
\renewcommand{\SS}{\mathcal S}
\newcommand{\Char}{\operatorname{Char}}

\title{On the Fr{\o}yshov invariant and monopole Lefschetz number}
\thanks{The first author was partially supported by the NSF Grant DMS-1707857, the second author was partially supported by NSF grant DMS-1506328 and a Simons Fellowship, and the third author was partially supported by a Collaboration Grant from the Simons Foundation}
\author[Jianfeng Lin]{Jianfeng Lin}
\address{Department of Mathematics\newline\indent Massachusetts Institute of
Technology \newline\indent Cambridge MA 02139}
\email{\rm{linjian5477@gmail.com}}
\author[Daniel Ruberman]{Daniel Ruberman}
\address{Department of Mathematics, MS 050\newline\indent Brandeis
University \newline\indent Waltham, MA 02454}
\email{\rm{ruberman@brandeis.edu}}
\author[Nikolai Saveliev]{Nikolai Saveliev}
\address{Department of Mathematics\newline\indent
University of Miami \newline\indent PO Box 249085
\newline\indent Coral Gables, FL 33124}
\email{\rm{saveliev@math.miami.edu}}

\begin{document}

\begin{abstract}
Given an involution on a rational homology 3-sphere $Y$ with quotient the $3$-sphere, we prove a formula for the Lefschetz number of the map induced by this involution in the reduced monopole Floer homology. This formula is motivated by a variant of Witten's conjecture relating the Donaldson and Seiberg--Witten invariants of 4-manifolds. A key ingredient is a skein-theoretic argument, making use of an exact triangle in monopole Floer homology, that computes the Lefschetz number in terms of the Murasugi signature of the branch set and the sum of Fr\o yshov invariants associated to spin structures on $Y$. We discuss various applications of our formula in gauge theory, knot theory, contact geometry, and 4-dimensional topology.   
\end{abstract}


\maketitle

\section{Introduction}
The monopole Floer homology defined by Kronheimer and Mrowka \cite{Kronheimer-Mrowka} using Seiberg--Witten gauge theory is a powerful invariant of 3--manifolds which has had many important applications in low-dimensional topology. Because of its functoriality \cite[Theorem 3.4.3]{Kronheimer-Mrowka}, the monopole Floer homology of a 3--manifold is acted upon by its mapping class group. However, the information contained in this action is not easy to extract due to the gauge theoretic nature of the theory. In this paper, we make some first steps towards understanding this action by calculating the Lefschetz numbers of certain involutions making the 3--manifold into a double branched cover over a link in the 3--sphere. Our study is motivated by the calculation of Lefschetz numbers in the instanton Floer homology \cite{RS2} and by a variant \cite[Conjecture B]{ruberman-saveliev:sw-casson} of Witten's conjecture~\cite{witten:monopole} relating the Donaldson and Seiberg--Witten invariants. The following theorem is the main result of the paper, and it comes with many interesting applications. 

\begin{introthm}\label{T:main}
Let $Y$ be an oriented rational homology 3--sphere with an involution $\tau: Y \to Y$ making it into a double branched cover of $S^3$ with branch set a non-empty link $L$. Denote by $\tau_*: \hmred(Y)\to \hmred(Y)$ the induced map in the reduced monopole Floer homology, and by $\Lef(\tau_*)$ its Lefschetz number. Then
\begin{equation}\label{E:main}
\Lef\, (\tau_*)\; =\; \frac{\, 2^{|L|}}{16}\;\xi(L)\, -\;\sum\limits_{\s}\; h(Y,\s),
\end{equation}
where $|L|$ is the number of components of the link $L$ and $\xi(L)$ is its Murasugi signature~\cite{murasugi}, and the last term is the sum over all the spin structures on $Y$ of the Fr{\o}yshov invariants $h(Y,\s)$ of the spin manifold $(Y,\s)$. In particular, if the link $L $ is a knot $K$ in $S^3$,
\begin{equation}\label{E:main-knot}
\Lef(\tau_{*})\; = \; \frac 1 8\, \sgn\,(K) -  h(Y),
\end{equation}
where $\sgn(K)$ is the classical knot signature and $h(Y)$ is the Fr{\o}yshov invariant for the unique spin structure on $Y$.
\end{introthm}

We are using here the rational numbers as the coefficient ring of the monopole Floer homology, and we will continue doing so throughout the paper unless otherwise noted. 
We expect that a formula similar to (\ref{E:main}) will hold for any rational homology sphere $Y$ and a diffeomorphism $\tau: Y \to Y$ of order $n \ge 2$. In the special case when the quotient $Y/\tau$ is an integral homology sphere $Y'$ and the branch set is a knot $K\subset Y'$, we make this expectation precise and conjecture that
\begin{equation}\label{Lefs for higher order}
\Lef\, (\tau_*)\; =\; n\cdot \lambda(Y')\;+\;\frac1{8}\;\sum\limits_{m=0}^{n-1}\;\sgn^{m/n}(K)\, -\, \sum_{\s}\; h(Y,\s),
\end{equation}

\smallskip\noindent
where $\lambda(Y')$ is the Casson invariant of $Y'$, $\sgn^{m/n}(K)$ is the Tristram--Levine signature of $K$ (see \cite[Section 6]{RS2}), and the last summation extends to the $\spin^c$ structures $\s$ on $Y$ such that $\tau_*(\s) = \s$.  The origins of this conjecture will be discussed in Section \ref{S:lfo}.

\begin{rmk}\label{R:HM=HF}
We will be working throughout with the Fr\o yshov invariants $h(Y,\s)$, which are defined via monopole homology. However, it is important to note that these are now known to be equivalent to the Heegaard Floer theory correction terms $d(Y,\s)$, for which many more calculations are available. In particular, the work of Kutluhan, Lee, and Taubes \cite{kutluhan-lee-taubes:HFSW-I,kutluhan-lee-taubes:HFSW-II,kutluhan-lee-taubes:HFSW-III,kutluhan-lee-taubes:HFSW-IV,kutluhan-lee-taubes:HFSW-V}, or alternatively, the work of Colin, Ghiggini, and Honda \cite{colin-ghiggini-honda:HFECH-I,colin-ghiggini-honda:HFECH-II,colin-ghiggini-honda:HFECH-III} and Taubes \cite{taubes:HMECH}, identifies the  monopole homology and the Heegaard Floer homology. Furthermore, by combining the main results of \cite{Gripp,Gripp-Huang,Gardiner}, the absolute $\Q$-gradings in the two theories coincide. Therefore, the relation $d(Y,\s) = -2h(Y,\s)$ between the Fr\o yshov invariant and the Heegaard Floer correction term holds for any rational homology spheres. This relation plays a role in our proof of Theorem \ref{T:main} (see Proposition \ref{t=1 case}) as well in several of the corollaries.\newline
\indent
We conjecture that a version of Theorem \ref{T:main} holds for Heegaard Floer homology.  This would be established by showing that  the isomorphisms cited above between the reduced Floer theories are natural with respect to cobordisms, so that the Lefschetz numbers computed in the two theories are the same.\end{rmk}


\subsection{An outline of the proof} 
Since $Y$ is the double branched cover of $S^{3}$ with branch set $L$, we will also use the notation $Y = \Sigma(L)$ and assume that the orientation on $Y$ is pulled back from the standard orientation of $S^3$. We need to show the vanishing of the link invariant 
\begin{equation}\label{E:chi}
\chi(L)\;=\;\frac1{2^{|L|-1}}\,\left(\Lef\, (\tau_*)+\,\sum\limits_{\s}\; h(\Sigma(L),\s)\right)\,-\,\frac 1 8\;\xi(L)
\end{equation}
for all links $L$ with non-zero determinant. This is done by an inductive argument involving a skein relation between $\chi(L)$, $\chi(L_0)$, and $\chi(L_1)$, where $L_0$, $L_1$ are resolutions of $L$ at a certain crossing. The skein relation for $\xi(L)$ can be proved directly, and the skein relations for the other two terms are a consequence of an exact triangle relating the monopole Floer homology of $\Sigma(L)$, $\Sigma(L_0)$, and $\Sigma(L_1)$.
 
While the idea is straightforward, there are several technical obstacles one needs to overcome. First of all, to understand the skein-theoretic behavior the monopole Lefschetz number (as a rational number), one needs an exact triangle with $\Q$--coefficients; however, the original exact triangle \cite{KMOS} has coefficients in $\Z/2$. While one may be able to adapt the proof there by putting suitable plus and minus signs before various terms appearing in the proof, keeping the signs straight is complicated and would require a significant amount of work. Here, we follow an alternative route: we show that, with some extra input from homological algebra, one can deduce a $\mathbb{Q}$--coefficient exact triangle from the corresponding $\Z/2$ exact triangle using the universal coefficient theorem. It is a delicate matter to define the signs involved in the maps of this new exact triangle so that they are compatible with the induced action of $\tau$; we need this compatibility to deduce a vanishing result for the total Lefschetz number.

The second difficulty comes from the fact that the version of monopole Floer homology that appears in the exact triangle is $\widecheck{\HM}(Y)$, and it is always infinite dimensional. To discuss the Lefschetz number, one needs to modify $\widecheck{\HM}(Y)$ to a finite dimensional vector space by ignoring all generators of sufficiently high degree. However, we lose the exactness of the triangle by such a truncating operation. As a consequence, the skein relation for $\chi(L)$ only holds up to a universal constant $C$ depending on certain combinatorial data including the surgery coefficients. To prove that $C$ always equals zero, we start from the example of two-bridge links. Since such links are known to have vanishing $\chi(L)$, we can use them to show that in some cases, the constant $C$ vanishes and the actual skein relation holds. With the help of this special skein relation, we can produce more examples of links $L$ with $\chi(L) = 0$ and prove the vanishing result for $C$ in a more general situation. Repeating this procedure several times, we eventually produce enough examples to prove that $C=0$ in all possible cases. 

After establishing the skein relation, one might hope to prove $\chi(L)=0$ by an inductive argument. However, such an argument would need to avoid links with zero determinant because double branched covers of such links, not being rational homology 3-spheres, may have more complicated monopole Floer homology. Unfortunately, it is not clear how to reduce a general non-zero determinant link to the unknot solely by resolving crossing without involving any zero determinant links. To overcome this obstacle, we make use of Mullins's \emph{skein theory for non-zero determinant links} \cite{mullins}. Following his idea, we extend the inductive statement by adding another skein relation relating $\chi(L)$ with $\chi(\bar{L})$, where $\bar{L}$ is obtained from $L$ by a crossing change. The relation is then established by comparing the two exact triangles arising from the triples $(L,L_0,L_1)$ and $(\bar{L},L_1,L_0)$. 


\subsection{Calculations and applications}
Theorem \ref{T:main} can be used in several different ways. In some cases (for instance, when $Y$ is an $L$-space), the monopole Lefschetz number automatically vanishes and we obtain a direct relation between the Fr{\o}yshov invariant and the Murasugi signature. In other cases, one can use formula \eqref{E:main} to compute the monopole Lefschetz number. This Lefschetz number contains important information about the action and can be used to  explicitly describe the action in some cases, leading to non-trivial conclusions. The applications we present in this paper fall into four different categories: gauge theory, knot theory, contact geometry, and 4-dimensional topology. 


\subsubsection{\rm{\textbf{An application to gauge theory}}}\label{S:lfo}
Let $X$ be a closed smooth oriented 4-manifold such that
\begin{equation}\label{E:homology condition}
H_{*}(X;\Z) = H_{*}(S^1\times S^{3};\Z)\quad\text{and}\quad H_{*}(\tilde{X};\mathbb{Q}) = H_{*}(S^{3};\mathbb{Q}),
\end{equation}
where $\tilde X$ is the universal abelian cover of $X$ associated with a choice of generator for $H^1 (X;\Z) = \Z$, called a homology orientation on $X$. Associated with $X$ are two gauge-theoretic invariants whose definition depends on a choice of Riemannian metric on $X$ but which end up being metric independent. The invariant $\lfo(X)$ is roughly one quarter times a signed count of anti-self-dual connections on the trivial $SU(2)$ bundle over $X$; and the invariant $\lsw(X)$ is roughly a signed count of the Seiberg--Witten monopoles over $X$ plus an index theoretic correction term. The invariant $\lfo(X)$ was originally defined by Furuta and Ohta \cite{FO} under the more restrictive hypothesis that $H_*(\tilde X;\Z) = H_* (S^3; \Z$), and the invariant $\lsw(X)$ was defined by Mrowka, Ruberman, and Saveliev \cite{MRS1} without any assumption on $\tilde X$.

\begin{conj}\label{C:lfo-lsw}
For any closed oriented homology oriented smooth 4-manifold $X$ satisfying condition (\ref{E:homology condition}), one has
\[
\lfo(X)=-\lsw(X).
\]
\end{conj}

\noindent
This conjecture is a slight generalization of \cite[Conjecture B]{MRS1}. It relates the Donaldson and Seiberg--Witten invariants of certain smooth 4-manifolds and therefore can be thought of as a variant of the Witten conjecture \cite{witten:monopole} for manifolds with vanishing second Betti number. The conjecture has been verified in a number of examples \cite{ruberman-saveliev:sw-casson}. The following theorem, which we prove in this paper, provides further evidence towards it. 

\begin{introthm}\label{T:lfo-lsw}
Let $\tau: Y \to Y$ be an involution on a rational homology sphere $Y$ making $Y$ into a double branched cover of $S^3$ with branch set a knot $K$, and let $X$ be the mapping torus of $\tau$.  Then 
\[
\lfo(X) = -\lsw(X) = \frac 1 8\, \sgn\,(K).
\]
\end{introthm}

Note that our conjectural formula (\ref{Lefs for higher order}) can be interpreted as a special case of Conjecture \ref{C:lfo-lsw} for the mapping torus of a diffeomorphism of order $n$, by using the splitting formula for $\lsw(X)$ proved in our earlier paper \cite[Theorem A]{LRS} and the calculation of $\lfo(X)$ for the finite order mapping tori \cite[Theorem 1.1]{RS2}.


\subsubsection{\rm{\textbf{Strongly non-extendable involutions and Akbulut corks}}}
In \cite{Akbulut:cork}, Akbulut constructed a smooth compact contractible 4-manifold $W_1$ with boundary an integral homology sphere $\partial W_1$ and an involution $\tau: \partial W_1 \to \partial W_1$ which can be extended to $W_1$ as a homeomorphism but not as a diffeomorphism; it was the first example of what is now known as the Akbulut cork. We improve upon this result by constructing the first known example  of what we call a `strongly non-extendable involution'. The precise statement  is as follows.

\begin{introthm}\label{T: absolutely nonextendable involution}
There exists a smooth involution $\tau: Y \to Y$ on an integral homology 3-sphere $Y$ which has the following two properties:
\begin{itemize}
\item[(1)] $Y$ bounds a smooth contractible 4-manifold, and
\item[(2)] $\tau$ can not be extended as a diffeomorphism to any $\Z/2$ homology 4-ball that $Y$ bounds.
\end{itemize}
\end{introthm} 

\noindent
One example of a strongly non-extendable involution claimed by Theorem \ref{T: absolutely nonextendable involution} is the aforementioned involution $\tau: \partial W_1 \to \partial W_1$ of the original Akbulut cork $(W_1,\tau)$: we show that $\tau$ does not extend to a self-diffeomorphism not just of $W_1$ but of \emph{any} homology ball that $\partial W_1$ may bound. We accomplish this by computing the induced action of $\tau$ on the monopole Floer homology $\widehat{\HM}(\partial W_1; \Z)$ with the help of Theorem \ref{T:main} and the calculation of Akbulut and Durusoy~\cite{akbulut-durusoy:involution}.

When the homology 4-ball bounded by $Y$ is contractible, the involution $\tau$ always extends to it as a homeomorphism. Using this idea, we give a general construction in Section \ref{S:cork} that results in a large family of new corks. It is worth mentioning that previous examples of corks were usually detected by embedding them in a closed manifold whose smooth structure is changed by the cork twist. (In the terminology of~\cite{auckly-kim-melvin-ruberman:isotopy}, they have an {\em effective} embedding.)  On the other hand, the corks we construct do not have an obvious effective embedding and they are detected by monopole Floer homology.


\subsubsection{\rm{\textbf{Knot concordance and Khovanov homology thin knots}}}\label{S:concordance}
Recall from \cite{khovanov I, khovanov II} that a link $L$ in the 3--sphere is called Khovanov homology thin (over $\F_2$) if its reduced Khovanov homology $\widetilde{\KH}(L;\F_2)$ is supported in a single $\delta$-grading. Such links are rather common: for instance, according to \cite{manolescu-ozsvath:quasi}, all quasi-alternating links, as well as 238 of the 250 prime knots with up to 10 crossings, are Khovanov homology thin. In follows from the spectral sequence of Bloom \cite{bloom: spectral sequence} that $\hmred(\Sigma(L)) = 0$ if $L$ is a Khovanov homology thin link. Combined with Theorem \ref{T:main}, this leads to the following series of corollaries, the first of which confirms the conjecture of Manolescu and Owens \cite[Conjecture 1.4]{manolescu-owens}.

\begin{introcor}\label{C: froyshov for thin knots}
For any Khovanov homology thin link $L$ with nonzero determinant, one has the relation
\[
\xi(L)\; =\; 8\sum\limits_{\s \in \operatorname{spin}(\Sigma(L))}
h(\Sigma(L),\s)
\]
between the Murasugi signature of $L$ and the Fr\o yshov invariants of the double branched cover $\Sigma(L)$.
\end{introcor}

\begin{introcor}\label{C: Lefschetz concordance invariant}
For a knot $K$ in $S^3$, denote by $L(K)$ the Lefschetz number of the map on $\HM^{\red} (\Sigma(K))$ induced by the covering translation. Then $L(K)$ is a non-trivial additive concordance invariant which vanishes on Khovanov homology thin knots.
\end{introcor} 

\begin{introcor}\label{C: direct summand}
Let $\mathcal{C}_{\operatorname{s}}$ be the smooth knot concordance group  and $\mathcal{C}_{\operatorname{thin}}$ its subgroup generated by the Khovanov homology thin knots. Then the quotient group $\mathcal{C}_{\operatorname{s}}/\mathcal{C}_{\operatorname{thin}}$ contains a $\Z$--summand.
\end{introcor}


\subsubsection{\rm{\textbf{The choice of sign in the monopole contact invariant}}}
For any compact contact 3-manifold $(Y,\xi)$, Kronheimer and Mrowka \cite{km:contact} (see also \cite{KMOS}) defined a contact invariant 
\[
\psi(Y,\xi)\in \widecheck{\HM}(-Y; \Z/2),
\]
as well as a pair of monopole homology classes
\[
\pm\, \tilde{\psi}(Y,\xi)\in \widecheck{\HM}(-Y;\Z).
\]
Note that this construction results in a pair of homology classes rather than a single class when working with integer coefficients. Technically, this happens because the Seiberg--Witten moduli space involved in the construction does not carry a natural orientation. One might hope that a canonical element can be picked in the pair $\pm\,\tilde{\psi}(Y,\xi)$ by further analysis. However, we show that that is not possible: with the help of Theorem \ref{T:main}, we construct an involution on the Brieskorn homology sphere $-\Sigma(2,3,7)$ which preserves a certain contact structure but changes the sign of the (non-torsion) contact invariant. 

\begin{introthm}\label{T: contact sign}
There exists no canonical choice of sign in the definition of $\pm\, \tilde{\psi}(Y,\xi)$.
\end{introthm}

It is worth mentioning that similar contact invariants in Heegaard Floer homology were defined by Ozsv\'ath and Szab\'o \cite{Ozsvath-Szabo:contact}. A version of Theorem \ref{T: contact sign} in context of Heegaard Floer homology has been proved by Honda, Kazez, and Mati\'c \cite{Honda-Kazez-Matic:contact} using an approach different from ours. As discussed in Remark~\ref{R:HM=HF}, there exists an isomorphism between Heegaard Floer homology and monopole Floer homology which preserves the contact invariant \cite{colin-ghiggini-honda:HFECH-I,colin-ghiggini-honda:HFECH-II,colin-ghiggini-honda:HFECH-III,taubes:HMECH}. However, since the naturality of this isomorphism has not been established, our result and that of Honda, Kazez, and Mati\'c do not imply each other.


\subsection{Organization of the paper}
Section \ref{S:proof} sets up the skein theory argument, reducing the proof of Theorem \ref{T:main} to the key Proposition \ref{skein relation for chi}. The proof of that proposition occupies Sections 3--6, which form the bulk of the paper. Section \ref{S:skein-signature} establishes a skein relation for the Murasugi signature $\xi(L)$, and Section \ref{S:triangle} sets up a surgery exact triangle with rational coefficients that will be crucial for the remainder of the argument.  In Section \ref{S:skein}, we show that Proposition \ref{skein relation for chi} holds up to certain universal constants $C$, and organize the rather complicated data necessary to track spin and $\spinc$ structures through the skein theory argument. Section \ref{S:universal} studies the skein exact sequence for a large number of examples, sufficient to show that the constants $C$ vanish, thereby establishing Proposition \ref{skein relation for chi} and Theorem \ref{T:main}.

The remainder of the paper is devoted to applications. In Section \ref{lfo} we extend the definition of the Furuta--Ohta invariant $\lfo$ and evaluate it for the mapping torus of an involution on a rational homology sphere with homology sphere quotient. This establishes Theorem~\ref{T:lfo-lsw}.  We calculate the effect of a particular involution on a cork boundary in Section \ref{S:cork} proving the non-extension result Theorem \ref{T: absolutely nonextendable involution}. Corollaries \ref{C: froyshov for thin knots}, \ref{C: Lefschetz concordance invariant}, and \ref{C: direct summand} of Theorem \ref{T:main} concerning the knot concordance group are established in Section \ref{S:thin}. Finally, Section \ref{S:contact} proves the non-canonical nature of the sign in the contact invariant (Theorem \ref{T: contact sign}).

\medskip\noindent
\textbf{Acknowledgments:}\; We thank Ken Baker, John Baldwin, Olga Plamenevskaya, Marco Golla and Youlin Li for generously sharing their expertise, Tom Mark for pointing out Tosun's paper~\cite{Tosun}, and Christine Lescop for an interesting discussion on Mullins's approach to skein theory for the double branched cover.


\section{Skein relations and the proof of Theorem \ref{T:main}}\label{S:proof}
Let $L$ be an unoriented link in $S^3$ and $\Sigma(L)$ its double branched cover. A quasi-orientation of $L$ is an orientation on each component of $L$ modulo an overall orientation reversal. The set of quasi-orientations of $L$ will be denoted by $Q(L)$. Turaev \cite{Turaev notes} established a natural bijective correspondence between $Q(L)$ and $\text{spin}(\Sigma(L))$, the set of spin structures on $\Sigma(L)$.

The link $L$ will be called \emph{ramifiable} if $\det(L)\neq 0$ or, equivalently, if $\Sigma(L)$ is a rational homology sphere. Note that all knots are ramifiable. Given a ramifiable link $L$ in $S^3$, consider the quantity 
\[
\chi(L) = \frac1{2^{|L|-1}}\left(\sum\limits_{\s\in \text{spin}(\Sigma(L))}h(\Sigma(L),\s)\; -\; \frac 1 8 \sum\limits_{\ell\in \SQ(L)} \sigma(\ell)\; +\; \Lef(\tau_{*}) \right),
\]
where $|L|$ is the number of components of $L$, $\sigma(\ell)$ is the signature of the link $L$ quasi-oriented by $\ell$, $h(\Sigma(L),\s)$ is the Fr\o yshov invariant of the spin manifold $(\Sigma(L),\s)$, and $\Lef(\tau_{*})$ is the Lefschetz number of the map 
\begin{equation}\label{E:lef}
\tau_{*}: \HM^{\red}(\Sigma(L))\rightarrow \HM^{\red}(\Sigma(L))
\end{equation}
on the reduced monopole Floer homology of $\Sigma(L)$ induced by the covering translation $\tau: \Sigma(L) \to \Sigma(L)$. That the above formula for $\chi(L)$ matches formula \eqref{E:chi} can be seen as follows. 

Write $L = K_1\,\cup \ldots \cup\, K_m$ as a link of $m = |L|$ components, and choose a quasi-orientation $\ell \in Q(L)$. Recall that the \emph{Murasugi signature} of $L$ is defined as
\[
\xi(L) = \sigma(\ell)+ \sum\limits_{1\leq i<j\leq m} \operatorname{lk}(K_{i},K_j).
\]
Murasugi \cite{murasugi} proved that $\xi(L)$ does not depend on the choice of quasi-orientation $\ell$, hence $\xi(L)$ can be defined alternatively as
\begin{equation}\label{Murasugi invariant}
\xi(L) = \frac1{2^{m-1}}\sum\limits_{\ell\in Q(L)}\sigma(\ell).
\end{equation}
The following theorem is then equivalent to Theorem \ref{T:main}. 

\begin{thm}\label{main theorem}
For any ramifiable link $L\subset S^{3}$, one has $\chi(L)=0$.
\end{thm}

Our proof of Theorem \ref{main theorem} will rely on skein theory. Given a link $L$ in the 3--sphere, fix its planar projection and consider two resolutions of $L$ at a crossing $c$ as shown in Figure \ref{F:resolutions}; we follow here the convention of \cite{oz: spectral sequence}.

\bigskip

\begin{figure}[!ht]
\centering
\psfrag{L}{$L$}
\psfrag{L0}{$L_0$}
\psfrag{L1}{$L_1$}
\includegraphics[width=4.2in]{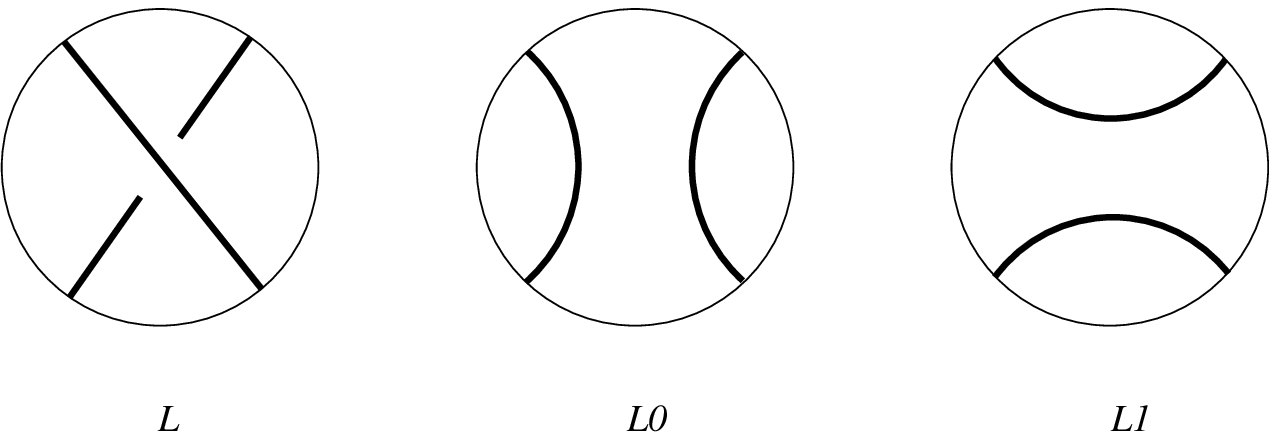}
\caption{}\label{F:resolutions}
\end{figure}

\noindent
The links $L_0$ and $L_1$ are called the $0$-resolution and the $1$-resolution of $L$, respectively, and the triple $(L, L_0, L_1)$ is called a \emph{skein triangle}. Note that a skein triangle possesses a cyclic symmetry: for any link in $(L,L_0, L_1)$, the other two taken in the prescribed cyclic ordering are its $0$- and $1$-resolutions. This symmetry is best seen when the links are drawn as in Figure \ref{F:tetrahedra}; see also \cite[Figure 6]{KM:khovanov}. Denote by $\bar L$ the link obtained by changing the crossing $c$ in the link $L$.

\begin{figure}[!ht]
\centering
\psfrag{L}{$L$}
\psfrag{L0}{$L_0$}
\psfrag{L1}{$L_1$}
\includegraphics{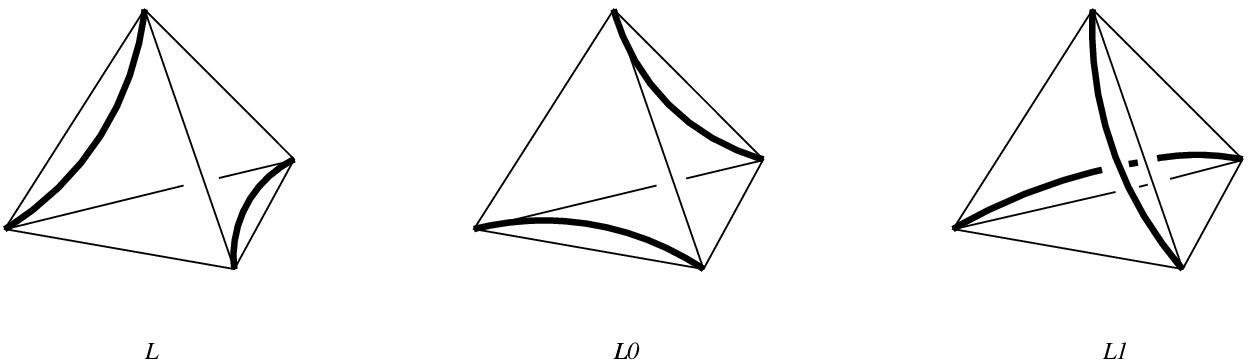}
\caption{}\label{F:tetrahedra}
\end{figure}

\begin{pro}\label{skein relation for chi}
Let $(L,L_0,L_1)$ be a skein triangle and assume that $L$ is ramifiable. Then at least two of the three links $\bar L$, $L_0$, $L_1$ are ramifiable and, in addition,
\begin{enumerate}[(i)]
\item if both $L_0$ and $L_1$ are ramifiable and $\chi(L_0) = \chi(L_1) = 0$ then $\chi(L) = 0$.
\item if one of the links $L_0$, $L_1$ is not ramifiable then $\chi(L) = \chi(\bar{L})$.
\end{enumerate}
\end{pro}

\noindent
We will now prove Theorem \ref{main theorem} assuming Proposition \ref{skein relation for chi}; the proof of the proposition will then occupy Section \ref{S:skein-signature} through Section \ref{S:universal}.

\begin{proof}[Proof of Theorem \ref{main theorem}]
The proof is a modification of the proof of \cite[Theorem 3.3]{mullins}. We will proceed by induction on the pair $(c(L),|L|)$, where $c(L)$ is the number of crossings in a diagram of $L$.

The case $(0,1)$ is trivial. The cases $(0,n)$ with $n \geq 2$ are vacuous because unlinks with more than one component are not ramifiable. Next, suppose that the statement has been proved for all links admitting a diagram with $k$ or fewer crossings. We want to prove it for the case $(k+1,n)$ with $n \ge 1$.

First let $n = 1$ then $L$ is a knot admitting a diagram with $k+1$ crossings. By changing $m < k$ crossings we can unknot $L$, thereby obtaining a sequence of knots
\[
L = L^1\rightarrow L^2 \rightarrow \cdots \rightarrow L^{m+1} = \text{unknot},
\]
where $L^{a+1}$ is obtained from $L^a$ by a single crossing change. Denote by $L^a_0$ and $L^a_1$ the two resolutions of $L^a$. We have $\chi(L^{m+1}) = 0$. To deduce that $\chi(L^a) = 0$ from $\chi (L^{a+1}) = 0$, we will consider the following two cases:

\begin{itemize}
\item[(i)] Both $L^a_0$ and $L^a_1$ are ramifiable. Since $c(L^a_0) \le k$ and $c(L^a_1) \leq k$, it follows from our induction hypothesis that $\chi(L^a_0) = \chi(L^a_1) =0$. Proposition \ref{skein relation for chi} (i) then implies that $\chi(L^a) = 0$.
\item[(ii)] One of the resolutions $L^a_0$, $L^a_1$ is not ramifiable. Then Proposition \ref{skein relation for chi} (ii) implies that $\chi(L^a) = \chi(L^{a+1}) = 0$, and we are finished.
\end{itemize}

Now let $n\geq 2$ so that $L$ is a multi-component link admitting a diagram with $k + 1$ crossings. Again, change $m < k$ crossings one by one to obtain a sequence of links
\[
L=L^1 \rightarrow L^2 \rightarrow \cdots \rightarrow L^{m+1} = \text{a split link},
\]
where $L^{a+1}$ is obtained from $L^a$ by a single crossing change. Since multi-component split links are not ramifiable, we can find $b \le m$ such that $L^1, \cdots, L^b$ are ramifiable and $L^{b+1}$ is not. Proposition \ref{skein relation for chi} then implies that both $L^b_0$ and $L^b_1$ are ramifiable. Since $c(L^b_0) \le k$ and $c(L^b_1) \le k$, we conclude that $\chi(L^b_0) = \chi(L^b_1) =0$ from our induction hypothesis. Proposition \ref{skein relation for chi} (i) then implies that $\chi(L^b) = 0$. The deduction that $\chi(L^{a+1}) = 0$ implies $\chi(L^a) = 0$ for all $a < b$ is exactly the same as in the $n=1$ case.
\end{proof}


\section{Skein relation for the Murasugi signature}\label{S:skein-signature}
Let $(L_0,L_1,L_2)$ be a skein triangle obtained by resolving a crossing $c$ inside a ball $B$. For any subscript $j$ viewed as an element of $\mathbb Z/3 = \{0,1,2\}$, denote by $S_j$ the standard cobordism surface from $L_j$ to $L_{j+1}$ inside the 4-manifold $[0,1] \times S^3$ obtained by adding a single $1$-handle to the product surface outside of $[0,1] \times B$. Denote by $W_j$ the double branched cover of $[0,1] \times S^3$ with branch set $S_j$. The manifold $W_j$ is an oriented cobordism from $\Sigma(L_j)$ to $\Sigma(L_{j+1})$; it will be described as a surgery cobordism in Section \ref{S:triangle}. The signature of $W_j$ can be either $0$, $1$, or $-1$.

\begin{lem}\label{skein for signature}
Let $(L_0,L_1,L_2)$ be a skein triangle such that $|L_2| = |L_0| + 1 = |L_1| + 1$, which is to say that the resolved crossing $c$ is between two different components of $L_2$. Then
\begin{equation}\label{skein relation}
2\xi(L_2)=\xi(L_0)+\xi(L_1) + \sgn(W_1) - \sgn(W_2).
\end{equation}
\end{lem}


\begin{proof} 
This can be derived from the Gordon--Litherland~\cite{gordon-litherland:signature} formula for the Murasugi signature but we will follow a more self-contained approach. It will rely on the disjoint decomposition 
\[
\SQ(L_2)=\SQ_0(L_2)\cup \SQ_1(L_2),
\]
where $\SQ_0(L_2)$ (resp. $\SQ_1(L_2)$) consists of the quasi-orientations of $L_2$ which make $L_0$ (resp. $L_1$) into an oriented resolution. For any $\ell \in \SQ_0 (L_2)$, the induced quasi-orientation on $L_0$ will be denoted by $\ell_0$; this establishes a bijective correspondence $Q_0 (L_2) = Q(L_0)$. Similarly, for any $\ell \in \SQ_1 (L_2)$, the induced quasi-orientation on $L_1$ will be denoted by $\ell_1\in \SQ(L_1)$; this establishes a bijective correspondence $Q_1(L_2) = Q(L_1)$. We claim that
\begin{align*}
& \sigma(\ell_0) = \sigma(\ell) + \sgn(W_2)\quad\text{for any}\quad\ell\in \SQ_0(L_2),\quad\text{and} \\
& \sigma(\ell_1)=\sigma(\ell)-\sgn(W_1)\quad\text{for any}\quad\ell\in \SQ_1(L_2).
\end{align*}
These identities, which are essentially due to Murasugi, can be verified as follows. Let $\ell \in \SQ_0(L_2)$. Since $\ell_0$ is an oriented resolution of $\ell$, the cobordism surface $S_2 \subset [0,1] \times S^3$ is naturally oriented. Choose an (oriented) Seifert surface for $\ell\subset \partial D^4$ and slightly push its interior into the interior of $D^4$ to obtain a properly embedded surface $F\subset D^4$.  Passing to double branched covers, we obtain
\[
\Sigma(D^4\cup ([0,1]\times S^3), F\cup S_2)\, =\, \Sigma(D^4, F)\,\cup\, \Sigma([0,1]\times S^3,S_2),
\]
where $\Sigma(A,B)$ stands for the double branched cover of $A$ with branch set $B$. Using additivity of the signature and the fact that $\Sigma([0,1]\times S^3,S_2) = W_2$, we obtain
\[
\sgn(\Sigma(D^4\cup ([0,1]\times S^3), F\cup S_2))\, =\, \sgn(\Sigma(D^4, F))+ \sgn(W_2).
\]
Observe that the surface $F\cup S_{2}$ in $D^{4}\cup([0,1]\times S_2)$ is an embedded surface with boundary $\ell_0$ in $\{1\}\times S_{2}$. It is a classical result (see for instance \cite{kauffman-taylor:links}) that
\[
\sgn(\Sigma(D^{4}, F))\, =\, \sigma(\ell)\quad\text{and}\quad \sgn(\Sigma(D^{4}\cup ([0,1]\times S^3), F\cup S_2))\,=\,\sigma(\ell_0).
\]
Therefore $\sigma (\ell_0) = \sigma(\ell) + \sgn(W_2)$ for any $\ell\in \SQ_0(L_2)$. The proof of the other identity is similar. With these identities in place, the proof of the lemma is completed as follows:
\begin{multline}\notag
2^{|L_2|-1}\cdot \xi(L_2)=\sum\limits_{\ell\in \SQ_0(L_2)}\sigma(\ell)+\sum\limits_{\ell\in \SQ_1(L_2)}\sigma(\ell)\\= \sum\limits_{\ell_0\in \SQ(L_0)}(\sigma(\ell_0)-\sgn(W_2))+\sum\limits_{\ell_1\in \SQ(L_1)}(\sigma(\ell_1)+\sgn(W_1))\\
=2^{|L_2|-2}\cdot (\xi(L_0)+\xi(L_1)+\sgn(W_1) - \sgn (W_2)).
\end{multline}
\end{proof}


\section{An exact triangle in monopole Floer homology}\label{S:triangle}
In this section, we will establish an exact triangle in the monopole Floer homology with rational coefficients. We will also show that this triangle possesses a certain conjugation symmetry, which will be instrumental in the proof of Proposition \ref{skein relation for chi} later in the paper.


\subsection{Statement of the exact triangle}
Let $Y$ be a compact connected oriented 3-manifold with boundary $\partial Y = T^2$, and let $\gamma_0$, $\gamma_1$, and $\gamma_2$ be oriented simple closed curves on $\partial Y$ such that
\[
\# (\gamma_0\cap \gamma_1) = \# (\gamma_1\cap \gamma_2) = \# (\gamma_2\cap \gamma_0) = -1,
\]
where the algebraic intersection numbers $\#$ are calculated with respect to the boundary orientation on $\partial Y$. Let $\mathbb F_2$ be the field of two elements. It follows from Poincar\'e duality that the kernel of the map $H_1 (\partial Y;\mathbb F_2) \to H_1 (Y; \mathbb F_2)$ is one-dimensional, therefore, we may assume without loss of generality that $\gamma_2$ is an $\mathbb F_2$ longitude (meaning that $[\gamma_2] = 0 \in H_1 (Y;\mathbb F_2)$), while $\gamma_0$ and $\gamma_1$ are not. 

For any $j$ viewed as an element of $\mathbb Z/3 = \{0, 1, 2\}$, denote by $Y_j$ the closed manifold obtained from $Y$ by attaching a solid torus to its boundary with meridian $\gamma_j$, and by $W_j$ the respective surgery cobordism from $Y_j$ to $Y_{j+1}$. The cobordism $W_j$ can be obtained by attaching $D^4$ to the component $S^3$ in the boundary of the 4-manifold
\begin{equation}\label{E:W0}
W^0_j = ([-1,0] \times Y_j)\, \cup_{\{0\}\times Y} ([0,1]\times Y)\, \cup_{\{1\}\times Y} ([1,2]\times Y_{j+1}).
\end{equation}

\begin{lem}\label{spin cobordism}
The manifolds $W_1$ and $W_2$ are spin, and the manifold $W_0$ is not.
\end{lem}

\begin{proof}
Notice that the inclusion $Y \to Y_2$ induces an isomorphism $H^1 (Y_2;\mathbb F_2) \to H^1 (Y;\mathbb F_2)$ hence any spin structure on $Y$ can be extended to a spin structure on $Y_2$. To show that $W_1$ is spin, start with any spin structure $\s$ on $Y_1$ and extend $\s|_{Y}$ to a spin structure on $Y_2$. This gives a spin structure on $W^0_1$, which extends over $D^4$ to a spin structure on $W_1$. A similar argument shows that $W_2$ is also spin.

To show that $W_0$ is not spin, we argue as follows. Suppose $W_0$ has a spin structure $\s$. By the argument above, $\s|_{Y_1}$ can be extended to a spin structure on $W_1$. Glue these two spin structures together to obtain a spin structure on the manifold $W_0\cup_{Y_1}W_1$. This leads to a contradiction because the manifold 
\[
W_0\,\cup_{Y_1} W_1 = (-W_2)\, \#\, \cptwobar
\]
contains an embedded sphere with self-intersection number $-1$.
\end{proof}

The space of $\spinc$ structures has a natural involution which carries a $\spinc$ structure $\s$ to its conjugate $\bar\s$. A $\spinc$ structure $\s$ is called self-conjugate if $c_1(\s) = \s - \bar\s$ vanishes. For a fixed self-conjugate $\spinc$ structure $\s_0$ on $Y$, we will come up with  an exact triangle involving $\spinc$ structures on the manifolds $Y_j$ restricting to the $\spinc$ structure $\s_0$ on $Y$. The usual exact triangle, involving all $\spinc$ structures, can be obtained by taking the direct sum of these restricted exact triangles over all possible $\mathfrak{s}_{0}$.

We will set up some notation first, for use in this and the following section. In the six lines that follow, $M$ can be any of the manifolds $Y$, $Y_j$ or $W_j$, $j \in \Z/3$, and we write:
\medskip
\begin{alignat*}1
\torspinc(M) &=\{\text{equivalence classes of torsion $\spinc$ structure on } M\}, \\
\scspinc(M) &=\{\text{equivalence classes of self-conjugate $\spinc$ structure on } M\}, \\
\torspinc(M,\s_0) &=\{\s\in \torspinc(M)\mid \s|_{Y}=\s_0\}, \\
\scspinc(M,\s_0) &=\{\s\in \scspinc(M)\mid \s|_{Y}=\s_0\}, \\
\spinc(M,\s_0) &= \{\s\in \spinc(M)\mid \s|_{Y}=\s_0\},\;\text{and} \\
\spin(M,\s_0) &=\{\s\in \spin(M)\mid (\s|_{Y})^{c}=\s_0\},\;\text{with the notation $(\s|_{Y})^{c}$ explained below}. \\
\end{alignat*}

\begin{rmk}\label{R:sc}
Recall that each spin structure $\s$ on $Y$ induces a self-conjugate $\spinc$ structure, which we denote by $\s^{c}$, and that each self-conjugate $\spinc$ structure on $Y$ is obtained in this fashion. Let $\s_1$ and $\s_2$ be two spin structures on $Y$ then $\s_1^c = \s_2^c$ if and only if $\s_1 = \s_2 + h$ for some $h$ in the image of the coefficient map $H^1 (Y; \Z)\rightarrow H^1 (Y; \Z/2)$. Therefore, each self-conjugate $\spinc$ structure on $Y$ corresponds to $2^{\,b_1(Y)}$ spin structures. A similar remark applies to the manifolds $Y_j$ and $W_j$, $j \in \Z/3$.
\end{rmk}

Our exact triangle will consist of the Floer homology groups\footnote{As we mentioned in the introduction, the monopole Floer homology will have the rational numbers as their coefficient ring unless otherwise noted.}
\[
\widecheck{\HM}(Y_j,[\s_0])\; =\mathop{\bigoplus}\limits_{\s\in \spinc(Y_j,\s_0)}\widecheck{\HM}(Y_j,\s)
\]

\noindent
and the maps between them induced by the cobordisms $W_j$. To ensure that the composition of any two adjacent maps is zero, we need to assign an appropriate plus or minus sign to each $\spinc$ structure on $W_j$. We will accomplish this by defining an auxiliary map 
\medskip
\begin{equation}\label{defi:mu}
\mu: \mathop{\bigcup}\limits_{j\in \Z/3} \spinc(W_j,\s_0)\longrightarrow \F_2
\end{equation}
as follows:
\begin{itemize}
\item $\mu$ is identically zero on $\spinc(W_2,\s_0)$;
\item Choose a base point $\s_1\in \spinc(W_0,\s_0)$ and let $\mu(\s_1) = 0$. Given an element of $\spinc(W_0,\s_0)$, write it in the form $\s_1+h$ with $h\in \ker (H^2(W_0;\Z)\rightarrow H^2(Y;\Z))$ and let
\[
\mu(\s_1+h) = h_{\F_2},
\] 
where 
\[
h_{\F_2} \in \ker (H^2(W_0;\F_2)\rightarrow H^2(Y;\F_2)) = \F_2
\]
is the mod $2$ reduction of $h$;

\item The case of $\spinc(W_1,\s_0)$ is similar: choose a base point $\s_2\in \spinc(W_1,\s_0)$ and let
\[
\mu (\s_2+h)=h_{\F_2}
\]
for any $h\in \ker (H^2(W_1;\Z)\rightarrow H^2(Y;\Z))$.
\end{itemize}

\begin{pro}\label{conjugation on mu}
$\mu(\s)=\mu(\bar{\s})$ for $\s\in \spinc(W_1,\s_0)$ and $\s \in \spinc(W_2,\s_0)$, and $\mu(\s)= \mu(\bar{\s})+1$ for $\s\in \spinc(W_0,\s_0)$.
\end{pro}

\begin{proof}
The lemma is trivial for $\s\in \spinc(W_2,\s_0)$. For $\s\in \spinc(W_1,\s_0)$, since that $\s=\bar{\s}+c_1(\s)$, the difference $\mu(\s)-\mu(\bar{\s})$ equals the mod 2 reduction of $c_1(\s)$, which is just the Stiefel--Whitney class $\omega_2(W_1)$. According to Lemma \ref{spin cobordism} the cobordism $W_1$ is spin, hence $\omega_2(W_1)=0$ and $\mu(\s)=\mu(\bar{\s})$. The proof for $\s\in \spinc(W_0,\s_0)$ is similar.
\end{proof}

By functoriality of monopole Floer homology, the cobordism $W_j$ equipped with a $\spinc$ structure $\s\in \spinc(W_j,\s_0)$ induces a map 
\[
\widecheck{\HM}(W_j,\s):\widecheck{\HM}(Y_j,\s|_{Y_j})\rightarrow \widecheck{\HM}(Y_j,\s|_{Y_{j+1}}),\quad j \in \Z/3.
\]
We will combine these maps into a single map
\[
F_{W_j}: \widecheck{\HM}(Y_j,[\s_0])\rightarrow \widecheck{\HM}(Y_{j+1},[\s_0])
\]
defined by the formula
\[
F_{W_j} = \sum\limits_{\s\in \spinc(W_j,\s_0)}(-1)^{\mu(\s)}\cdot \widecheck{\HM}(W_j,\s).
\]
Note that, up to an overall sign, the maps $F_{W_j}$ are independent of the arbitrary choices of base points in the definition \eqref{defi:mu}.

\begin{pro}\label{composition of maps equals zero}
We have $F_{W_{j+1}}\circ\, F_{W_j}=0$ for all $j\in \Z/3$.
\end{pro}

\begin{proof}
Using the composition law for the cobordism induced maps in monopole Floer homology, we obtain
\begin{equation}\label{composing cobordism map}
F_{W_{j+1}}\circ F_{W_j} = \sum\limits_{\s\in \spinc(W_j\cup W_{j+1},\s_0)}(-1)^{\mu(\s |_{W_j}) + \mu(\s |_{W_{j+1}})}\cdot \widecheck{\HM}(W_{j+1}\circ W_j,\s).
\end{equation}
The manifold 
\[
X_j = W_j\cup W_{j+1} = (-W_{j+2})\,\#\;\cptwobar
\]
has an embedded 2--sphere $E_j$ with self-intersection $-1$. Therefore, every $\s\in \spinc(X_j,\s_0)$ can be uniquely written as $\s_1\# \s_2$ with $\s_1\in \spinc(-W_{j+2},\s_0)$ and $\s_2\in \spinc(\cptwobar)$. Let us consider a diffeomorphism of $X_j$ which takes $[E_j]\in H_2 (X_j)$ to $-[E_j] \in H_2 (X_j)$ and restricts to the identity map on $-W_{j+2}$. Since this diffeomorphism does not change the homology orientation, and since the cobordism map in monopole Floer homology is natural, we obtain the identity
\[
\widecheck{\HM}(W_{j+1}\circ W_j,\s_1\# \s_2)=\widecheck{\HM}(W_{j+1}\circ W_j,\s_1\#\bar{\s}_2).
\]
Note that $\bar{\s}_2$ never equals $\s_2$ because $\cptwobar$ is not spin. As a result, the terms on the right hand side of (\ref{composing cobordism map}) come in pairs. The proof of the proposition will be complete after we prove the following lemma.
\end{proof}

\begin{lem}\label{conjugate different sign}
For any $j\in \Z/3$ and any $\spinc$ structures $\s_1\in \spinc(-W_{j+2},\s_0)$ and $\s_2\in \spinc(\cptwobar)$,
\[
\mu((\s_1 \# \s_2)|_{W_j})+\mu((\s_1\# \s_2)|_{W_{j+1}}) = 1 + \mu((\s_1\# \bar{\s}_2)|_{W_j}) + \mu((\s_1\# \bar{\s}_2)|_{W_{j+1}})
\in \F_2.
\]
\end{lem}

\begin{proof} 
Let PD stand for the Poincar\'e duality isomorphism. Then 
\begin{equation}\label{conjugate difference}
\s_1 \# \s_2 = \s_1\# \bar{\s}_2 + (2k+1)\cdot \PD[E_j]
\end{equation}
for some $k\in \Z$. To prove the lemma, we will compute the mod 2 reductions of $\PD[E_j]|_{W_j}$ and $\PD[E_j]|_{W_{j+1}}$, which we will denote by $(\PD[E_j]|_{W_j})_{\F_2}$ and $(\PD[E_j]|_{W_{j+1}})_{\F_2}$, respectively.

Recall that $W_j$ is obtained by attaching a 2--handle $H_j$ to $I \times Y_j$. Since $\gamma_2$ (treated as a knot in $\{1\}\times Y_j$) is an $\F_2$ longitude, we can find an immersed, possibly non-orientable surface $\Sigma_2\subset Y$ with boundary $\gamma_2$. Capping $\Sigma_2$ off with the surface $\Sigma_1\subset H_j$ bounded by $\gamma_2$,  we obtain a closed surface $\Sigma_1 \cup_{\gamma_2} \Sigma_2$ which generates the group
\[
\ker(H^2(W_j;\F_2)\rightarrow H^2(Y;\F_2))^* = \operatorname{coker}(H_2(Y;\F_2)\rightarrow H_2(W_j;\F_2)) = \F_2.
\]
As a result, we have 
\[
(\PD[E_j]|_{W_j})_{\F_2} = \#(E_j\cap (\Sigma_1\cup_{\gamma_2}\Sigma_2)) \pmod 2.
\]
Since $\Sigma_1\cup_{\gamma_2}\Sigma_2$ is contained in $W_j$, the 2--sphere $E_j$ in the above formula can be replaced by $E_j\,\cap\,W_j$, which is a 2--disk $D_j\subset H_j$. The boundary of $D_j$, denoted by $\ell_{j+1}$, is the core of the solid torus $Y_{j+1}\setminus \operatorname{int}(Y)$, therefore,
 \[
 \#(E_j\cap (\Sigma_1\cup_{\gamma_2}\Sigma_2))=\#(D_j\cap \Sigma_1),
 \]
which is the linking number of $\ell_{j+1}$ and $\gamma_2$ inside $\partial H_j$. After a  moment's thought we conclude that 
\[
\text{lk}(\ell_{j+1},\gamma_2) = \pm\, \#(\gamma_j\cap \gamma_2)
\]
and therefore 
\[
(\PD[E_j]|_{W_j})_{\F_2}=\#(\gamma_j\cap \gamma_2) \pmod 2.
\] 
A similar argument shows that 
\[
(\PD[E_j]|_{W_{j+1}})_{\F_2}=\#(\gamma_{j+2}\cap \gamma_2) \pmod 2.
\]
The rest of the proof is straightforward. We assume that $j = 0$; the other cases are similar. By (\ref{conjugate difference}) and the definition of $\mu$, we have
\begin{multline}\notag
\mu((\s_1\# \s_2)|_{W_0})+\mu((\s_1\# \s_2)|_{W_1})-\mu((\s_1\# \bar{\s}_2)|_{W_0})-\mu((\s_1\# \bar{\s}_2)|_{W_1}) = \\ 
(\PD[E_0]|_{W_0})_{\F_2}+(\PD[E_0]|_{W_1})_{\F_2} =
\#(\gamma_0\cap \gamma_2)+\#(\gamma_2\cap \gamma_2) = 1 \pmod 2.
\end{multline}
\end{proof}

We are now ready to state the main result of this section, the exact triangle in monopole Floer homology with rational coefficients. 

\begin{thm}\label{exact triangle}
The following sequence of monopole Floer homology groups is exact over the rationals
\smallskip
\[
\begin{CD}
\cdots @>F_{W_0} >> \widecheck{\HM}(Y_1,[\s_0]) @> F_{W_1} >> \widecheck{\HM}(Y_2,[\s_0]) @> F_{W_2} >> \widecheck{\HM}(Y_0,[\s_0]) @> F_{W_0} >> \cdots
\end{CD}
\]
\end{thm}

\medskip


\subsection{Proof of the exact triangle}\label{S:pet}
We already know from Proposition \ref{composition of maps equals zero} that the composite of any two adjacent maps is zero. To complete the proof of exactness, we will combine the universal coefficient theorem with the $\F_2$ coefficient exact triangle proved in \cite{KMOS}.

Before we go on with the proof, we need to review some basic constructions in monopole Floer homology; see Kronheimer--Mrowka \cite{Kronheimer-Mrowka} for details. For any $j \in \Z/3$, denote by $C^{o}(Y_j)$ (resp. $C^{s}(Y_j)$ and $C^{u}(Y_j)$) the free $\Z$-modules generated by the gauge equivalence classes of irreducible monopoles (resp. boundary stable and boundary unstable monopoles) whose associated spin$^{c}$ structure belongs to $\spinc(Y_j,\s_0)$. By counting points in the zero-dimensional moduli spaces of monopoles on $\mathbb{R}\times Y_j$, we obtain a linear map
\[
\partial^{o}_{o}(Y_j): C^{o}(Y_j)\rightarrow C^{o}(Y_j)
\]
and its companions $\partial^{o}_{s}(Y_j),\ \partial^{u}_{o}(Y_j),\ \partial^{u}_{s}(Y_j),\ \bar{\partial}^{s}_{u}(Y_j),\ \bar{\partial}^{u}_{s}(Y_j),\ \bar{\partial}^{s}_{s}(Y_j),\ \bar{\partial}^{u}_{u}(Y_j)$. Note that the last four maps count only reducible monopoles. Set $$\widecheck{C}(Y_j)=C^{o}(Y_j)\oplus C^{s}(Y_j)$$
and define the map
\[
\check{\partial}(Y_j):  \widecheck{C}(Y_j)\rightarrow \widecheck{C}(Y_j)
\]
by the matrix
\[
\begin{bmatrix}
   \partial^{o}_{o}(Y_j) & -\partial^{u}_{o}(Y_j)\bar{\partial}^{s}_{u}(Y_j)   \\
   \partial^{o}_{s}(Y_j) & \bar{\partial}^{s}_{s}(Y_j)-\partial^{u}_{s}(Y_j)\partial^{s}_{u}(Y_j)
\end{bmatrix}.
\]

\medskip\noindent
One can check that $\check{\partial}(Y_j)\circ \check{\partial}(Y_j) = 0$. The homology of the chain complex $(\widecheck{C}(Y_j),\check{\partial}(Y_j))$ is the monopole Floer homology $\widecheck{\HM}(Y_j,[\s_0];\Z)$. To obtain homology with rational coefficients, we set $\widecheck{C}(Y_j)_{\mathbb{Q}}=\widecheck{C}(Y_j)\otimes _{\Z}\mathbb{Q}$ and use the linear map $\check{\partial}(Y_j)_{\mathbb{Q}} = \check{\partial}(Y_j)\, \otimes\, \id$ as the boundary operator. (We will henceforth use similar notations without further explanation). Consider the manifold with cylindrical ends 
\[
W^{*}_j=((-\infty,0]\times Y_j)\cup_{Y_j} W_j\cup_{Y_{j+1}}([0,+\infty)\times Y_{j+1}).
\]
(In what follows, the superscript $*$ will indicate attaching a product end to the boundary of the manifold at hand). For any $\s\in \spinc(W_j,\s_0)$, the count of monopoles on $W^{*}_j$ defines the map
\[
m^{o}_{o}(W_j,\s): C^{o}(Y_j)\rightarrow C^{o}(Y_j)
\]
and its companion maps $m^{o}_{s}(W_j,\s)$, $m^{u}_{o}(W_j,\s)$, $m^{u}_{s}(W_j,\s)$, $\bar{m}^{s}_{u}(W_j,\s)$, $\bar{m}^{s}_{s}(W_j,\s)$, $\bar{m}^{u}_{s}(W_j,\s)$, and $\bar{m}^{u}_{u}(W_j,\s)$. Define the map
\[
\check{m}(W_j,\s): \widecheck{C}(Y_j)\rightarrow \widecheck{C}(Y_{j+1})
\]
by the matrix
\[
\begin{bmatrix}
   m^{o}_{o}(W_j,\s) & -m^{u}_{o}(W_j,\s) \bar{\partial}^{s}_{u}(Y_j)-\partial^{u}_{o}(Y_{j+1}) \bar{m}^{s}_{u}(W_j,\s) \\
   m^{o}_{s}(W_j,\s) &  \bar{m}^{s}_{s}(W_j,\s)-m^{u}_{s}(W_j,\s) \bar{\partial}^{s}_{u}(Y_j)- \partial^{u}_{s}(Y_{j+1})\bar{m}^{s}_{u}(W_j,\s)
\end{bmatrix}
\]

\medskip\noindent
and sum over all the spin$^{c}$ structures with proper signs to obtain the map
\smallskip
\[
\check{m}(W_j):=\sum\limits_{\s\in \spinc(W_j,\s_0)} (-1)^{\mu(\s)}\cdot \check{m}(W_j,\s): \widecheck{C}(Y_j)\rightarrow \widecheck{C}(Y_{j+1}).
\]
This is the chain map that induces the map $F_{W_j}$ in the exact triangle.

As our next step, we will construct an explicit null-homotopy of the composite $\check{m}(W_{j+1})\,\circ\, \check{m}(W_j)$. To this end, recall that the composite cobordism 
\[
X_j = W_j\cup_{Y_{j+1}}W_{j+1}
\]
from $Y_j$ to $Y_{j+2}$ contains an embedded 2--sphere $E_j$ with self-intersection number $-1$. Denote by $S_j$ the boundary of a normal neighborhood of $E_j$. Let 
\[
Q_j = \{\,g_T\; |\; T\in \R \,\}
\]
be the family of metrics on $X_j$ constructed as follows. Start with an arbitrary metric $g_0$ on $X_j$ which is a product metric near $\partial X_j$, $Y_{j+1}$, and $S_j$, and which has the property that the metric it induces on $S_j$ is close enough to the round metric to have positive scalar curvature. For any $T \in \R$, the metric $g_T$ is then obtained from $g_0$ by inserting the cylinder $[T,-T]\times S_j$ into a normal neighborhood of $S_j$ if $T < 0$, and by inserting the cylinder $[-T,T]\times Y_{j+1}$ into a normal neighborhood of $Y_{j+1}$ if $T > 0$.

Given a spin$^{c}$ structure $\s$ on $X_j$, we again count monopoles on the manifold $X_j^{*}$ with cylindrical ends over the whole family $Q_{j}$ to define the map
\[
H^{o}_{o}(X_j,\s): C^{o}(Y_j)\rightarrow C^{o}(Y_{j+2})
\]
as well as its companion maps $H^{o}_{s}(X_j,\s)$, $H^{u}_{o}(X_j,\s)$, $H^{u}_{s}(X_j,\s)$, $\bar{H}^{s}_{u}(X_j,\s)$, $\bar{m}^{s}_{s}(X_j,\s)$, $\bar{m}^{u}_{s}(X_j,\s)$, and $\bar{H}^{u}_{u}(X_j,\s)$. Using these maps, we define the map
\[
\widecheck{H}(X_j,\s):\widecheck{C}(Y_j)\rightarrow \widecheck{C}(Y_{j+2})
\]
by the matrix
\[
\left[
\begin{array}{c|c}
 H^{o}_{o}(X_j,\s) & 
  \begin{array}{l}
   -H^{u}_{o}(X_j,\s) \bar{\partial}^{s}_{u}(Y_j)\\
   -m^{u}_{o}(W_{j+1},\s)\bar{m}^{s}_{s}(W_j,\s|_{W_j})-\partial^{u}_{o}(Y_{j+2}) \bar{H}^{s}_{u}(X_j,\s|_{W_j}) 
  \end{array}
  \\
  \\[-3ex]
\hline
  \\
  \\[-7ex]
   H^{o}_{s}(X_j,\s)  & 
  \begin{array}{l}
   \bar{H}^{s}_{s}(X_j,\s)-H^{u}_{s}(X_j,\s) \bar{\partial}^{s}_{u}(Y_j)\\
   -m^{u}_{o}(W_{j+1},\s|_{W_{j+1}})\bar{m}^{s}_{s}(W_j,\s|_{W_j})- \partial^{u}_{s}(Y_{j+2})\bar{H}^{s}_{u}(X_j,\s)
  \end{array}
\end{array}
\right]
\]


\medskip\noindent
Note that an $\F_2$ version of this map can be found in \cite[page 491]{KMOS}, and the correct sign assignments for its integral version in \cite[(26.12)]{Kronheimer-Mrowka}. By summing up over the spin$^{c}$ structures, we obtain the map
\[
\widecheck{H}(X_j) = \sum\limits_{\s\in \spinc(X_j,\s_0)} (-1)^{\mu(\s|_{W_j})+\mu(\s|_{W_{j+1}})}\widecheck{H}(X_j,\s):\widecheck{C}(Y_j)\rightarrow \widecheck{C}(Y_{j+2}).
\]

\begin{pro} (1) One has the equality
\begin{equation}\label{chain homotopy}
\check{\partial}(Y_{j+2}) \circ \widecheck{H}(X_j) + \widecheck{H}(X_j) \circ \check{\partial}(Y_j)=\check{m}(W_{j+1})\circ \check{m}(W_j).
\end{equation}
\noindent
(2) The map $\Psi_j: \widecheck{C}(Y_j)\rightarrow \widecheck{C}(Y_j)$ defined as
\[
\widecheck{H}(X_{j+1})\circ \check{m}(W_j)-\check{m}(W_{j+2})\circ \widecheck{H}(X_j)
\]
is an anti-chain map. Moreover, the map 
\[
(\Psi_j)_{\mathbb{Q}}:\widecheck{C}(Y_j)_{\mathbb{Q}}\rightarrow \widecheck{C}(Y_j)_{\mathbb{Q}}
\]
induces an isomorphism in homology.
\end{pro}

\begin{proof}
(1) We can upgrade the proof of the mod 2 version \cite[Proposition 5.2]{KMOS} of this result to the integers as follows: Let $B_j$  be the (closed) normal neighborhood of $E_j$ and $Z_j$ the closure of $X_j\setminus B_j$. The family $Q_j$ of metrics on $X_j$ can be completed by adding the disjoint union $Z_j^{*}\cup B^{*}_j$ at $T=-\infty$ and the disjoint union $W^{*}_j\cup W^{*}_{j+1}$ at $T=+\infty$. Denote this new family by $\bar{Q}_j$. Given monopoles $\mathfrak{a}$ on $Y_j$ and $\mathfrak{b}$ on $Y_{j+2}$ and a $\spinc$ structure $\s$ on $X_j$, consider the parametrized moduli space
\[
\mathcal{M}(\mathfrak{a},(X^{*}_j,\s),\mathfrak{b})_{\bar{Q}}
\]
on the manifold $(X^{*}_j,\s)$ and construct its compactification $
\mathcal{M}^{+}(\mathfrak{a},(X^{*}_j,\s),\mathfrak{b})_{\bar{Q}}
$ by adding in broken trajectories. When this moduli space is one-dimensional, the number of its boundary points, counted with sign, must be zero. This gives us a boundary identity. By adding these boundary identities over all possible  $\s$ with sign $(-1)^{\mu(\s|_{W_j})+\mu(\s|_{W_{j+1}})}$, we obtain various summed up boundary identities for different $(\mathfrak{a},\mathfrak{b})$.

We claim that the points in $\mathcal{M}^{+}(\mathfrak{a},(X^{*}_j,\s),\mathfrak{b})_{g_{-\infty}}
$ do not contribute to these identities: As explained in the proof of \cite[Proposition 5.2]{KMOS}, these points always come in pairs of the form $(\gamma,\gamma')$ and $(\gamma,\gamma'')$, where $\gamma$ is a (possibly broken) solution over $Z^{*}_j$ and $\gamma'$ and $\gamma''$ are reducible solutions over $B^{*}_j$. Moreover, $\gamma'$ and $\gamma''$ correspond to conjugate spin$^{c}$ structures over $B^{*}_j$. Since $b^{+}_2(B_j)=b_1(B_j)=0$, all reducible monopoles over $B^*_j$ are positive. Therefore, $(\gamma,\gamma')$ and  $(\gamma,\gamma'')$ contribute to their respective boundary identities with the same sign. By Lemma \ref{conjugate different sign}, when we take the sum with the weights $(-1)^{\mu(\s|_{W_j})+\mu(\s|_{W_{j+1}})}$ these contributions cancel.

The rest of the proof proceed exactly as in \cite[Proposition 5.2]{KMOS}. In \cite[Lemma 26.2.3]{Kronheimer-Mrowka}, several similar boundary identities are obtained by considering one-dimensional moduli spaces of monopoles for a family of metrics parametrized by $[0,1]$. As a consequence of our claim, the summed up boundary identities we have here can be obtained from the identities there by removing terms corresponding to $T=0$. For example, we have
\begin{equation*}
\begin{split}
0= & \sum\limits_{\s\in \spinc(X_j,\s_0)} (-1)^{\mu(\s|_{W_{j+1}})+\mu(\s|_{W_j})} (-H^{o}_{o}(X_j,\s) \partial^{o}_{o}(Y_j)-\partial^{o}_{o}(Y_{j+2})H^{o}_{o}(X_j,\s)\\&+H^{u}_{o}(X_j,\s)\bar{\partial}^{u}_{s}(Y_j)\partial^{o}_{s}(Y_j)+\partial^{u}_{o}(Y_{j+2})\bar{H}^{s}_{u}(X_j,\s)\partial^{o}_{s}(Y_j)+\partial^{u}_{o}(Y_{j+2})\bar{\partial}^{s}_{u}(Y_{j+2})H^{o}_{s}(X_j,\s)\\&+m^{o}_{o}(W_{j+1},\s|_{W_{j+1}})m^{o}_{o}(W_j,\s|_{W_j})+m^{o}_{u}(W_{j+1},\s|_{W_{j+1}})\overline{m}^{s}_{u}(W_j,\s|_{W_j})\partial^{o}_{s}(Y_j)\\&-m^{u}_{o}(W_{j+1},\s|_{W_{j+1}})\overline{\partial}^{s}_{u}(Y_{j+1})m^{o}_{s}(W_j,\s|_{W_j})-\partial^{u}_{o}(Y_{j+2})\overline{m}^{s}_{u}(W_{j+1},\s|_{W_{j+1}})m^{o}_{s}(W_j,\s|_{W_j}))
\end{split}
\end{equation*}
Using these identities, formula (\ref{chain homotopy}) can be proved by an elementary (but cumbersome) calculation. 

(2) The fact that $\Psi_j$ is a chain map follows easily from (1). According to \cite[Lemma 5.11]{KMOS}, the map
\[
\Psi_j\,\otimes\,\id: \widecheck{C}(Y_j)\,\otimes\,\F_2 \rightarrow \widecheck{C}(Y_j)\,\otimes\,\F_2
\]
induces an isomorphism in homology. By the universal coefficient theorem, the map $\Psi_{\mathbb{Q}}$ also induces an isomorphism in homology. 
\end{proof}

The proof of Theorem \ref{exact triangle} is now completed by the following `triangle detection lemma'. The mod 2 version of this lemma appears as Lemma 4.2 in \cite{oz: spectral sequence}. The proof of the version at hand is essentially the same.\footnote{A version of this lemma over the integers can be found as Lemma 7.1 in \cite{KM:khovanov}. Our sign conventions here are slightly different.}

\begin{lem} 
For any $j\in \Z/3$, let $(C_j,\partial_j)$ be a chain complex over the rationals. Suppose that there are chain maps $f_j:C_j\rightarrow C_{j+1}$ satisfying the following two conditions:
\begin{itemize}
\item the composite $f_{j+1}\circ f_j$ is null-homotopic by a chain homotopy $H_j:C_j\rightarrow C_{j+2}$ with
\[
\partial H_j+H_j\partial=f_{j+1}\circ f_j, \quad\text{and} 
\]
\item the map 
\[
\psi_j=H_{j+1}\circ f_j-f_{j+2}\circ H_j: C_j\rightarrow C_j,
\]
which is an anti-chain map by the first condition, induces an isomorphism in homology.
\end{itemize}
Then the sequence
\[
\begin{CD}
\cdots @>>> H_{*}(C_j) @> {(f_j)_{*}} >> H_{*}(C_{j+1}) @> {(f_{j+1})_{*}} >> H_{*}(C_{j+2}) @>>> \cdots
\end{CD}
\]
is exact.
\end{lem}


\section{Skein relations up to constants}\label{S:skein}
Let $(L_0,L_1,L_2)$ be a skein triangle obtained by resolving a crossing $c$ of the link $L = L_2$ as shown in Figure \ref{F:resolutions}. 

 \begin{defi}\label{D:skein} The skein triangle $(L_0,L_1,L_2)$ will be called \emph{admissible} if 
\begin{enumerate}
\item $|L_2| = |L_0| + 1 = |L_1| + 1$, which is equivalent to saying that the resolved crossing $c$ is between two different components of $L_2$, and 
\item at least one of the links $L_0$, $L_1$, and $L_2$ is ramifiable.
\end{enumerate}
\end{defi}

\noindent
In Section \ref{S:proof} we defined a link $\bar L_2$ by changing the crossing $c$. Define $\bar L_0$ and $\bar L_1$ similarly using the cyclic symmetry as in Figure \ref{F:tetrahedra}.

\begin{lem}\label{at most one not ramifiable}
If $(L_0,L_1,L_2)$ is an admissible skein triangle then at most one of the six links $L_0$, $L_1$, $L_2$, $\bar L_0$, $\bar L_1$, $\bar L_2$ is not ramifiable. In particular, we have three more admissible skein triangles, $(\bar{L}_1,L_0,L_2)$, $(L_1,\bar{L}_0,L_2)$, and $(L_1,L_0,\bar{L}_2)$.
\end{lem}

\begin{proof}
By our definition of $\bar{L}$, all of the triples $(\bar{L}_1,L_0,L_2)$, $(L_1,\bar{L}_0,L_2)$, and $(L_1,L_0,\bar{L}_2)$ are skein triangles. Now suppose that two of the links $L_0$, $L_1$, $L_2$, $\bar L_0$, $\bar L_1$, and $\bar L_2$ are not ramifiable. By \cite[Claim 3.2]{mullins}, these two links have to be $\bar{L}_{j}$ and $\bar{L}_{j+1}$ for some $j\in \mathbb{Z}/3$. Recall that after putting suitable signs, the determinants of the three links in a skein triangle add up to zero. Therefore, from the skein triangles $(L_{j-1},L_{j+1},\bar{L}_{j})$ and $ (L_{j}, L_{j-1}, \bar{L}_{j+1})$ we deduce that $\det(L_{0})=\det(L_{1})=\det(L_{2})$. Since $(L_0,L_1,L_2)$ is a skein triangle, this implies that $\det(L_{0})=\det(L_{1})=\det(L_{2})=0$. This contradicts Condition (2) of Definition \ref{D:skein}.
\end{proof}

Let $(L_0,L_1,L_2)$ be an admissible skein triangle and $B\subset S^{3}$ a small ball containing the resolved crossing $c$. Denote by $Y$ the double branched cover of $S^{3}\setminus B$ with branch set $(S^{3}\setminus B)\cap L_2$ then $Y$ is a manifold with torus boundary $\partial Y$. 

\begin{defi}
A \emph{boundary framing} is a pair of oriented simple closed curves $(m,l)$ on $\partial Y$ such that
\begin{itemize}
\item[(1)] $\# (m\cap l)=-1$,
\item[(2)] $[l]=0\in H_1(Y;\mathbb{Q})$, and
\item[(3)] either $m$ or $l$ represents the zero element in $H_1(Y;\F_2)$.
\end{itemize}
\end{defi}

\noindent
One can easily check that a boundary framing always exists. Once a boundary framing $(m,l)$ is fixed, we will define the following numbers:
\begin{itemize}
\item The divisibility of the longitude
\[
t(Y) = \min\, \{\,a\in \Z\mid a > 0\;\text{ and }\; a\cdot [l]=0\in H_1 (Y;\Z)\,\};
\]
\item Set $s(Y)=0$ if $l$ represents the zero element in $H_1(Y;\F_2)$ and set $s(Y)=1$ otherwise;
\item The double branched cover $Y_j = \Sigma(L_j)$, $j \in \Z/3$, is obtained from $Y$ by attaching a solid torus along $\partial Y$, matching the meridian with a simple closed curve $\gamma_j$ on $\partial Y$. We will orient the curves $\gamma_j$ by the following two conditions:
\begin{itemize}
\item[(a)] $\# (\gamma_0 \cap \gamma_1) = \# (\gamma_1 \cap \gamma_2) = \# (\gamma_2 \cap \gamma_0) = -1$, where the algebraic intersection numbers $\#$ are calculated with respect to the boundary orientation on $\partial Y$ (see Ozsv{\'a}th--Szab{\'o} \cite[Section 2]{oz: spectral sequence}), and
\item[(b)] $\# (\gamma_2 \cap m)>0$ when $s(Y)=0$ and $\# (\gamma_2 \cap l)>0$ when $s(Y)=1$ (this makes sense because $\gamma_2$ represents zero in $H_1 (Y;\F_2)$ by Definition \ref{D:skein} (1)).
\end{itemize}
Having oriented the curves $\gamma_j$ this way, we define the integers $(p_j,q_j)$ by the equality $[\gamma_j] = p_j\cdot [m] + q_j\cdot [l]$, which holds in $H_1(\partial Y;\Z)$.
\end{itemize}

\begin{defi}\label{D:res}
Given a boundary framing $(m,l)$, we define the \emph{resolution data} for the admissible skein triangle $(L_0,L_1,L_2)$ as the six-tuple
\[
(t(Y),s(Y),(p_0,q_0),(p_1,q_1))
\]
(we should note that $(p_2,q_2)=(-p_0-p_1,-q_0-q_1)$ since $[\gamma_0]+[\gamma_1]+[\gamma_2]=0\in H_1(Y;\Z)$).
\end{defi}

The main goal of this section is to establish the following `skein relations up to universal constants'. We will show later in Section \ref{S:universal} that these universal constants actually vanish.

\begin{thm} \label{skein with error term}
Let $(L_0,L_1,L_2)$ be an admissible skein triangle, and fix a boundary framing $(m,l)$ on the boundary $\partial Y$ of the manifold $Y$ as above. Then
\begin{enumerate}
\item if all of the links $L_0,L_1,L_2$ are ramifiable,
\begin{equation}\label{skein with error term 1}
2\chi(L_2)=\chi(L_0)+\chi(L_1)+C(t(Y),s(Y),(p_0,q_0),(p_1,q_1));
\end{equation}
\item if $L_j$ is not ramifiable for some $j\in \Z/3$ then $L_{j\pm 1}$ and $\bar{L}_{j\pm 1}$ are all ramifiable and, in addition,
\begin{align}
& \chi(\bar{L}_{j-1})=\chi(L_{j-1})+C^{-}_j(t(Y),s(Y),(p_0,q_0),(p_1,q_1))\;\;\text{and}\label{skein with error term 2} \\
& \chi(\bar{L}_{j+1})=\chi(L_{j+1})+C^{+}_j(t(Y),s(Y),(p_0,q_0),(p_1,q_1)),\label{skein with error term 3}
\end{align}
\end{enumerate}
where $C(t(Y),s(Y),(p_0,q_0),(p_1,q_1))$ and $C^{\pm}_j(t(Y),s(Y),(p_0,q_0),(p_1,q_1))$ are certain universal constants depending only on the resolution data.
\end{thm}


\subsection{The action of covering translations}
We will use $\tau_M$ to denote covering translations on various double branched covers $M$ such as $M = Y_j$ or $M = W_j$.

\begin{lem}\label{homology of cobordism}
Let $(L_0,L_1,L_2)$ be an admissible skein triangle. Then exactly one of the following two options is realized:
\begin{enumerate}
\item if all $L_j$ are ramifiable, there exists a unique $n\in \Z/3$ such that $|p_{n}|>|p_{n\pm 1}|$ and, in addition,
\begin{itemize}
\item $b_1(Y_j)=b_1(W_j) = 0$ for all $j\in \Z/3$,
\item $b^+_2 (W_{n+1}) = 1$, $b^+_2 (W_{n-1}) = b^+_2 (W_n) = 0$, and
\item $b^-_2 (W_{n+1}) =0$, $b^-_2(W_{n-1}) = b^-_2(W_n)=1$.
\end{itemize}
\item if $L_{n}$ is not ramifiable for some $n\in \Z/3$ then
\begin{itemize}
\item $b_1(Y_{n})=1$ and $b_1(Y_{n-1}) = b_1(Y_{n+1}) = 0$,
\item $b_1(W_j)=b^{+}_2(W_j)=0$ for all $j\in \Z/3$, and
\item $b^{-}_2 (W_{n+1}) = 1$ and $b^{-}_2 (W_{n-1}) = b^{-}_2 (W_n) = 0$.
\end{itemize}
\end{enumerate}
\end{lem}

\begin{proof}
The claims about $b_1$ follow easily from the Mayer--Vietoris sequence. As for the $b^{\pm}_2$ claims, note that the inequality $|p_n| > |p_{n \pm 1}|$ is equivalent to $p_n$ having an opposite sign to both $p_{n-1}$ and $p_{n+1}$. The result then follows from the explicit calculation of the cup-product structure on $H^2(W_j)$ in Lemma \ref{l: cup product on W}. 
\end{proof}

\begin{lem}\label{conjugation action}
The covering translations act as follows:
\begin{enumerate}
\item $\tau^{*}_{Y}(a)=-a$ for any $a\in H^2(Y;\Z)$;
\item $\tau^{*}_{Y_j}(b)=-b$ for any $b\in H^2(Y_j;\Z)$;
\item $\tau^{*}_{W_j}(c)=-c$ for any $c\in H^2(W_j;\Z)$;
\item $\tau^{*}_{Y_j}(d)=-d$ for any $d\in H_1(Y_j;\Z)$;
\item $\tau^{*}_{W_j}(e)=-e$ for any $e\in H_1(W_j;\Z)$;
\item $\tau^{*}_{Y}(\s)=\bar{\s}$ for any $\s\in \spinc(Y)$;
\item $\tau^{*}_{Y_j}(\s)=\bar{\s}$ for any $\s\in \spinc
(Y_j)$;
\item $\tau^{*}_{X_j}(\s)=\bar{\s}$ for any $\s\in \spinc(X_j)$.
\end{enumerate}
\end{lem}
\begin{proof}
Since $H_2(Y;\Z)=0$, the universal coefficient theorem provides a natural identification of $H^2 (Y;\Z)$ with the torsion part of $H_1(Y;\Z)$. Therefore, to prove (1), we only need to show that $\tau$ acts on $H_1(Y;\Z)$ as $-1$. To see this, let $\tilde{\alpha}$ be a loop in $Y$ which does not intersect the branch locus. The image of $\tilde{\alpha}$ under the covering map $Y\rightarrow S^{3}\setminus B$ will be called $\alpha$. Since $S^{3}\setminus B$ is contractible, there is a continuous map $f: D^2\rightarrow S^{3}\setminus B$ with $f(\partial D^2)=\alpha$. Lift $f$ to a map $\tilde{f}: F\rightarrow Y$, where $F$ is a double branched cover of $D^2$. If $\tilde{f}(\partial F)$ equals $\tilde{\alpha}$ or $\tau(\tilde{\alpha})$, then $\tilde{\alpha}$ is null-homologous. Otherwise, we have $\tilde{f}(\partial F)=\tilde{\alpha}+\tau(\tilde{\alpha})$, which implies $[\tilde{\alpha}]=-[\tau(\tilde{\alpha})]$ and proves (1). Claim (4) can be proved similarly, while (2) is just the Poincar\'{e} dual of (4), and (5) follows from (4) and the fact that $H_1(\partial W_j;\Z)\rightarrow H_1(W_j;\Z)$ is onto.

We will next prove (3) under the assumption that $b_1(Y_j) = 0$ (otherwise, $b_1 (Y_{j+1}) = 0$, and the argument is similar). Using the Mayer--Vietoris sequence for the decomposition of $W_j$ into $I \times Y_j$ and the 2-handle, we obtain an exact sequence 
\[
\begin{CD}
0 @>>> H^1(S^1\times D^2;\Z) @> {\partial} >> H^2(W_j;\Z) @> {i^{*}} >> H^2(I \times Y_j;\Z) @>>> \cdots
\end{CD}
\]
The maps induced by $\tau$ on the cohomology groups in this sequence are compatible with $\partial$ and $i^{*}$. For any element $\alpha\in H^2(W_j;\Z)$, it follows from Claim (2) that $i^{*}(\tau^{*}_{W_j}(\alpha)+\alpha)=\tau^{*}_{Y_j}(i^{*}(\alpha))+i^{*}(\alpha)=0$. Therefore, there exists $\beta\in H^1(S^1 \times D^2;\Z)$ such that $\partial \beta= \tau^{*}_{W_j}(\alpha)+\alpha$. Notice that 
\[
\partial(-\beta)=\partial (\tau^{*}_{S^1\times D^2}(\beta))=\tau^{*}_{W_j}(\tau^{*}_{W_j}(\alpha)+\alpha)=\tau^{*}_{W_j}(\alpha)+\alpha=\partial \beta.
\]
Since $\partial$ is injective, $\beta$ must vanish and therefore $\tau^{*}_{W_j}(\alpha)+\alpha = 0$ as claimed in (3).

Let us now turn our attention to the action of $\tau^*$ on $\spinc$ structures, starting with (7). Turaev \cite[Section 2.2]{Turaev notes} established a natural one-to-one correspondence between spin structures on $Y_j$ and quasi-orientations on $L_j$. Using this correspondence, one can easily see that all the spin structures on $Y_j$ are invariant under $\tau_{Y_j}$. This proves (7) for self-conjugate $\spinc$ structures. Since any $\spinc$ structure can be written as $\s+h$, with $\s$  self-conjugate and $h\in H^2(Y;\Z)$, claim (7) follows from (2). To prove (6), consider $\s_0=\s_1|_{Y}$ with $\s_1$ a $\spinc$ structure on $Y_1$. By (7) we have $\tau^{*}_{Y}(\s_0)=\bar{\s}_0$. Then we express a general $\spinc$ structure as $\s_0+h$ with $h\in H^2(Y;\Z)$ and use (1).
We are left with (8). Take any $\s\in \spinc(W_j)$. It follows from (7) that
\[
\tau^{*}_{W_j}(\s)|_{Y_j}=\tau^{*}_{Y_j}(\s|_{Y_j})=\bar{\s}|_{Y_j}.
\]
Therefore, $\tau^{*}_{W_j}(\s)=\bar{\s}+h$ for some $h\in \ker(H^2(W_j;\Z)\rightarrow H^2(Y_j;\Z)).$ Using (3), we conclude that $c_1(\tau^{*}_{W_j}\s)=-c_1(\s)=c_1(\bar{\s})$ and therefore $2h=0$. However, it follows from the Mayer--Vietoris exact sequence that $\ker(H^2(W_j;\Z)\rightarrow H^2(Y_j;\Z))$ is torsion free. Therefore, $h=0$, and claim (8) is proved. 
\end{proof}

\subsection{$\text{Spin}^c$ systems and their equivalence}
In this section we introduce the concept of a $\spinc$ system on a skein triangle and relate the $\spinc$ systems corresponding to admissible skein triangles with the same resolution data.

\begin{defi}\label{spin-c system}
Let $(L_0,L_1,L_2)$ be a skein triangle and $\s_0$ a fixed self-conjugate $\spinc$ structure on $Y$. A $\spinc$ \emph{system} $\SS((L_0,L_1,L_2),\s_0)$ is the set
\[
\mathop{\bigcup}\limits_{j\in \Z/3}\; (\spinc(W_j,\s_0)\, \cup\, \spinc(Y_j,\s_0))
\]
endowed with the following additional structure:
\begin{enumerate}
\item the restriction map $r_{j,j+1}: \spinc(W_j,\s_0)\longrightarrow \spinc(Y_j,\s_0)\times \spinc(Y_{j+1},\s_0)$ for every $j \in \Z/3$,
\item for every $\s\in \spinc(W_j,\s_0)$ with $r_{j,j+1}(\s)$ torsion, the Chern number $c_1 (\s)^2 \in \Q$ (see \eqref{E:cup}), and 
\item the involution $\tau^{*}_{W_j}$ on $\spinc(W_j,\s_0)$ and the involution $\tau^{*}_{Y_j}$ on $\spinc(Y_j,\s_0)$. By Lemma \ref{conjugation action}, these act by conjugation on the set of $\spinc$ structures.
\end{enumerate}
\end{defi}

\begin{defi}\label{equivalent spin-c system}
Two $\spinc$ systems $\SS((L_0,L_1,L_2),\s_0)$ and $\SS((L'_0,L'_1,L'_2),\s'_0)$ are called \emph{equivalent} if $\sign (W_j) = \sign (W'_j)$ for all $j \in \Z/3$ and there exist bijections
\[
\hat{\theta}_j: \spinc(W_j,\s_0)\rightarrow \spinc(W'_j,\s'_0)\quad\text{and}\quad
\theta_j: \spinc(Y_j,\s_0)\rightarrow \spinc(Y'_j,\s'_0)
\]
which are compatible with the additional structures (1), (2), and (3) in the obvious way. We use $\sim$ to denote this equivalence relation.
\end{defi}

\begin{thm}\label{same resolution data give equivalent spin-c system}
Let $(L_0,L_1,L_2)$ and $(L'_0,L'_1,L'_2)$ be admissible skein triangles. Suppose that, for a suitable choice of  boundary framings, the resolution data of $(L_0,L_1,L_2)$ matches that of $(L'_0,L'_1,L'_2)$. Then there exist disjoint decompositions
\[
\scspinc(Y)=A_0\cup A_1\quad\text{and}\quad \scspinc(Y')=A'_0\cup A'_1
\]
with the following properties:
\begin{itemize}
\item $|A_0| = |A_1|$ and $|A'_0| = |A'_1|$ (the vertical bars stand for the cardinality of a set), and
\item for $i=0,1$ and any $\s_0\in A_i$ and $\s'_0\in A'_i$, we have
\[
\SS((L_0,L_1,L_2),\s_0)\sim \SS((L'_0,L'_1,L'_2),\s'_0).
\]
\end{itemize}
\end{thm}

The proof of Theorem \ref{same resolution data give equivalent spin-c system} will take up the rest of this subsection. The idea of the proof is straightforward: we give an explicit description of the topology and the $\spinc$ structures on $Y_j$ and $W_j$ in terms of the resolution data. 

We begin by studying the algebraic topology of cobordisms $W_j$. Let us consider the decompositions
\[
W_j = (I \times Y)\cup_{I \times \partial Y} D^{4}\quad \text{and}\quad Y_j=Y\cup_{\partial Y} (S^1\times D^2).
\]
The gluing map in the first decomposition (which matches the decomposition \eqref{E:W0}) identifies $\partial Y$ with the standard torus in $S^3 = \partial D^{4}$ by sending $\gamma_j$ and $\gamma_{j+1}$ to the standard meridian $m$ and longitude $l$, respectively. The corresponding Mayer--Vietoris exact sequences are of the form 
\begin{equation}\label{mv sequence1}
\begin{CD}
\ldots @>>> H^1(Y) @>>> H^1(\partial Y) @> {\partial_{W_j}} >> H^2(W_j) @> {i_{W_j}} >> H^2(Y) @>>> \ldots
\end{CD}
\end{equation}
and
\begin{equation}\label{mv sequence2}
\begin{CD}
\ldots @>>> H^1(Y)\oplus H^1(S^1\times D^2) @>>> H^1(\partial Y) @> {\partial_{Y_j}} >> H^2(Y_j) @> {i_{Y_j}} >> H^2(Y) @>>> \ldots
\end{CD}
\end{equation}

\smallskip\noindent
Let $\hat{{\gamma}}_j$, $\hat{m}$, and $\hat{l}\in H^1(\partial Y)$ be the Poincar\'e duals of $[\gamma_j]$, $[m]$, and $[l]\in H_1(\partial Y)$, respectively. The following lemma is a direct consequence of (\ref{mv sequence1}) and (\ref{mv sequence2}).
\begin{lem}\label{relative homology}
There are isomorphisms
\begin{gather*}
\ker i_{W_j}\;=\; \Z\oplus\Z/\langle(0,t(Y))\rangle\quad\text{with generators}\quad \partial_{W_j}(\hat{m}),\; \partial_{W_j}(\hat{l}),\quad\text{and} \\
\ker\, i_{Y_j}\;\,=\; \Z\oplus\Z/\langle(0,t(Y)),(p_j,q_j)\rangle\quad\text{with generators}\quad \partial_{Y_j}(\hat{m}),\; \partial_{Y_j}(\hat{l}),
\end{gather*}
under which the map $\ker i_{W_j}\to \ker i_{Y_j}\,\oplus\,\ker i_{Y_{j+1}}$ induced by the boundary inclusions is the natural projection map.
\end{lem}

Recall that, for a cohomology class $\alpha\in H^2(W_j)$ with $\alpha|_{\partial W_j}$ torsion, the number $\alpha^2 \in \mathbb{Q}$ is defined as follows. Let $u_j: H^2(W_j,\partial W_j)\rightarrow H^2(W_j)$ be the map induced by the inclusion. Then there exists a non-zero integer $k$ and a cohomology class $\beta\in H^2(W_j,\partial W_j)$ such that $k\alpha = u_j(\beta)$. We define
\begin{equation}\label{E:cup}
\alpha^2\; =\; \frac1{k^2}\; \langle\, \beta\smile \beta,\, [W_j,\partial W_j]\,\rangle.
\end{equation}

\begin{lem}\label{l: cup product on W}
(1) Suppose that $p_j \neq 0$ and $p_{j+1}\neq 0$. Then
\begin{equation}\label{cup product formula}
(\partial_{W_j}(a\hat{m}+b\hat{l}))^2\, =\, \frac{a^2}{p_j\cdot p_{j+1}}.
\end{equation}
(2) Suppose that either $p_j=0$ or $p_{j+1}=0$. Then $\alpha^2 = 0$ for any $\alpha\in H^2(W_j)$ with $\alpha|_{\partial W_j}$ torsion.
\end{lem}

\begin{proof}
Associated with the decomposition $W_j=(I \times Y)\cup_{I \times \partial Y} D^4$ is the Mayer--Vietoris exact sequence in homology, 
\[
\begin{CD}
\ldots @>>> H_2(Y) @> 0 >> H_2(W_j) @> {\partial} >> H_1(\partial Y) @>>> H_1(Y) @>>> \ldots
\end{CD}
\]
from which we conclude that $H_2(W_j)$ is a copy of $\Z$ generated by the homology class $[\Sigma]$ of a surface $\Sigma$ with $\partial[\Sigma]= t(Y)[l]$. The surface $\Sigma$ splits as $F_1\cup F_2$, where $F_1$ is an embedded surface in $I \times Y$ bounded by $t(Y)$ copies of $l$, and $F_2$ is the Seifert Surface for the right handed $(t(Y)\cdot p_j, t(Y)\cdot p_{j+1})$ torus link in $\partial D^{4}$. From this description, we see that the homological self-intersection number of $\Sigma$ equals the linking number between two parallel copies of the $(t(Y)\cdot p_j, t(Y)\cdot p_{j+1})$ torus link, which equals $t(Y)^2\cdot p_j\cdot p_{j+1}$. Therefore, 
\[
\langle\,\PD[\Sigma]\smile \PD[\Sigma],\,[W_j,\partial W_j]\,\rangle = \langle\,u_j (\PD[\Sigma]),\,[\Sigma]\,\rangle = t(Y)^2\cdot p_j\cdot p_{j+1}.
\]
Comparing this to 
\[
\langle\,\partial_{W_j}\hat{m},\, [\Sigma]\,\rangle = \langle \hat{m},\,\partial[\Sigma]\,\rangle =\langle \hat{m},\,t(Y)[l]\,\rangle = t(Y)
\]
we obtain 
\[
u(\PD[\Sigma])=\partial_{W_j}(\,t_{Y}\cdot p_j\cdot p_{j+1}\,\hat{m} + k \hat{l}\,)
\]
for some integer $k$, whose value is of no importance to us because $\partial_{W_j}(\hat{l})$ is torsion. From this we deduce that
\[
(\partial_{W_j}(a\hat{m}+b\hat{l}))^2 = \frac{a^2}{(t(Y)\cdot p_j\cdot p_{j+1})^2}\cdot \langle\, \PD[\Sigma]\smile \PD[\Sigma],\, [W_j,\partial W_j]\,\rangle=\frac{a^2}{p_j\cdot p_{j+1}}.
\]

\smallskip\noindent
This completes the proof of (1). To prove (2), observe that the map $H^2(W,\mathbb{Q}) \to H^2 (Y_{j+1},\mathbb{Q})$  is injective when $p_j=0$, and that the map $H^2(W,\mathbb{Q})\to H^2 (Y_j,\mathbb{Q})$ is injective when $p_{j+1}=0$. Therefore, any element $\alpha$ with $\alpha|_{\partial W}$ torsion is a torsion itself.
\end{proof}

Let $X_j = W_j\,\cup\, W_{j+1}$ be the composite cobordism from $Y_j$ to $Y_{j+2}$. As we mentioned in the proof of the exact triangle in Section \ref{S:pet}, there exists an embedded 2--sphere $E_j \subset X_j$ with homological self-intersection $-1$. It is obtained by gluing a disk $D_1\subset W_j$ to a disk $D_2 \subset W_{j+1}$ along $-\partial D_1=\partial D_2=l_{j+1}$, where $l_{j+1}$ is the core of the solid torus $Y_{j+1}\setminus \operatorname{int} Y$. Orient $E_j$ so that $l_{j+1}$ is homotopic to $\gamma _{j+2}$ in $Y_{j+1}\setminus \operatorname{int} Y$. Also recall a decomposition $X_j = (- W_{j+2})\,\#\,\cptwobar$, which induces an isomorphism
\begin{equation}\label{homology composition}
\rho_{j,j+1}: H^2(W_{j+2})\oplus H^2(\cptwobar)\longrightarrow H^2(X_j).
\end{equation}

\begin{lem}
For any integers $a$, $b$, $c$, denote by $\xi$ the image of $(\partial_{W_{j+2}}(a\hat{m}+b\hat{l}),c\cdot \PD[E_j])$ under the map $\rho_{j,j+1}$. Then
\[
\xi|_{W_j}=\partial_{W_j}(a\hat{m}+b\hat{l}+c\hat{\gamma_j})\quad\text{and}\quad
\xi|_{W_{j+1}}=\partial_{W_{j+1}}(a\hat{m}+b\hat{l}+c\hat\gamma_{j+2}).
\]
\end{lem}

\begin{proof}
It is a direct consequence of the naturality of the boundary map in the Mayer--Vietoris exact sequence that 
\[
\rho_{j,j+1}(\partial_{W_{j+2}}(a\hat{m}+b\hat{l}),0)|_{W_{n}}=\partial_{W_{n}}(a\hat{m}+b\hat{l})\quad\text{for}\quad n=j, j+1.
\]
We still need to show that 
\[
\rho_{j,j+1}(0,\PD[E_j])|_{W_j}=\partial_{W_j}(\hat\gamma_j)\quad\text{and}\quad
\rho_{j,j+1}(0,\PD[E_j])|_{W_{j+1}}=\partial_{W_{j+1}}(\hat\gamma_{j+2}).
\]
The Poincare dual of $\rho_{j,j+1}(0,\PD[E_j])$ is realized by the sphere $E_j$. Therefore, the restriction of $\rho_{j,j+1}(0,\partial_{W_{j+2}}(\PD[E_j])$ to $W_j$ equals the Poincare dual of $[E_j\cap W_j]\in H_2(W_j,\partial W_j)$. Using the fact that $E_j\cap W_j$ is a disk contained in the two-handle $D^4\subset W_j$ with the boundary  $l_{j+1}$, one can easily verify that $[E_j\cap W_j]$ equals the Poincare dual of $\partial_{W_j}(\hat{\gamma_j})$. This finishes the proof of the first formula. The proof of the second formula is similar.
\end{proof}

We will next study the spin and $\spinc$ structures on the manifolds $Y_j$ and $W_j$. First, define a map
\[
\lambda: \spin(Y)\rightarrow H^1(\partial Y; \F_2)
\]
as follows (compare with Turaev \cite{Turaev survey}). Fix a diffeomorphism $\phi: \partial Y \to \R^2/\Z^2$ and let $x_1$, $x_2$ be the standard coordinates on $\R^2$. Pull back the vector fields $\partial/\partial x_1$, $\partial/\partial x_2$ via $\phi$ to obtain vector fields $\vec{v}_1$, $\vec{v}_2$ on $\partial Y$. Any loop $\gamma$ in $\partial Y$ gives rise to the loop 
\[
\tilde{\gamma}(x) = (\vec{n}(x),\vec{v}_1(x),\vec{v}_2(x))
\]
in the frame bundle of $Y$, where $\vec{n}(x)$ is the outward normal vector at $x\in \partial Y$. We define $\lambda(\s)$ to be the unique cohomology class in $H^1(\partial Y,\F_2)$ with the property that $\langle \lambda(\s),[\gamma]\rangle=0$ if and only if $\tilde{\gamma}$ can be lifted to a loop in the spin bundle for $\s$. 

\begin{lem}\label{condition for spin structure extension}
The map $\lambda: \spin(Y)\rightarrow H^1(\partial Y; \F_2)$ has the following properties:
\begin{itemize}
\item[(1)] $\lambda$ does not dependent on the choice of diffeomorphism $\phi$,
\item[(2)] $\lambda(\s+\omega) = \s+\omega|_{\partial Y}$ for any $\s\in \spin(Y)$ and $\omega\in H^1(Y;\F_2)$, and 
\item[(3)] for any $j \in \Z/3$, a spin structure $\s$ can be extended to $Y_j$ if and only if $\langle \lambda(\s),\gamma_j\rangle=1\in \F_2$. The extension is unique if it exists.
\end{itemize}
\end{lem}

\begin{proof}
This is immediate from the definition of the map $\lambda$.
\end{proof}

\begin{lem}\label{spin structure extension}
Any $\s\in \spin(Y)$ extends to a spin structure on $Y_2$, and to a spin structure on one of the manifolds $Y_0$ and $Y_1$ but not the other. For any $\s_0\in \scspinc(Y)$, we have the following identity for the counts of self-conjugate $\spinc$ structures:
\begin{equation}\label{spin structure count}
2^{b_1(Y_0)}\cdot |\scspinc(Y_0,\s_0)| + 2^{b_1(Y_1)}\cdot |\scspinc(Y_1,\s_0)| = 2^{b_1(Y_2)}\cdot |\scspinc(Y_2,\s_0)|.
\end{equation}
\end{lem}

\begin{proof}
Since the map $H^1(Y_2;\F_2)\to H^1(Y;\F_2)$ is an isomorphism, any spin structure $\s$ on $Y$ can be extended to a spin structure on $Y_2$. This implies that 
\begin{equation}\label{exactly one can be extended}
\langle \lambda(\s),[\gamma_0]\rangle+\langle \lambda(\s),[\gamma_1]\rangle=\langle \lambda(\s),[\gamma_2]\rangle=1.
\end{equation}
It now follows from Lemma \ref{condition for spin structure extension}\,(3) that $\s$ can be extended a spin structure on exactly one of the manifolds $Y_0$ and $Y_1$. This finishes the proof of the first statement. 

According to Remark \ref{R:sc}, a self-conjugate $\spinc$ structure on $Y_j$ corresponds to $2^{b_1(Y_j)}$ spin structures. Therefore, 
\[
2^{b_1(Y_j)}\cdot |\scspinc(Y_j,\s_0)| = |\spin(Y_j,\s_0)|,
\]
and (\ref{spin structure count}) is equivalent to 
\[
|\spin(Y_2,\s_0)| = |\spin(Y_0,\s_0)| + |\spin(Y_1,\s_0)|,
\] 
which follows easily from the first statement.
\end{proof}

\begin{lem}\label{lemma self-conjugate spinc extension}
A self-conjugate $\spinc$ structure $\s_0\in\scspinc(Y)$ has the following extension properties to the cobordisms $W_j$:
\begin{enumerate}
\item $\scspinc(W_0,\s_0) = \emptyset$;
\item If $t(Y)$ is odd, then $\scspinc(W_1,\s_0) \neq \emptyset$ and\, $\scspinc(W_2,\s_0)\neq \emptyset$;
\item If $t(Y)$ is even, there is a disjoint decomposition
\begin{equation}\label{decomposition of self-conjugate spinc}
\scspinc(Y)=A_0\cup A_1
\end{equation}
such that $|A_0| = |A_1|$ and, in addition,
\begin{equation}\label{half of self-conjugate spinc can be extended}
\begin{split}
\s_0\in A_0\quad\text{if and only if}\quad \scspinc(W_1,\s_0) = \emptyset\;\;\text{and}\;\; \scspinc(W_2,\s_0)\neq \emptyset, \\
\s_0\in A_1\quad\text{if and only if}\quad \scspinc(W_1,\s_0)\neq \emptyset\;\;\text{and}\;\; \scspinc(W_2,\s_0)=\emptyset.
\end{split}
\end{equation}
\end{enumerate}
\end{lem}

\begin{proof}
Denote by $\s_0^1$ and $\s_0^2$ the two spin structures on $Y$ corresponding to the self-conjugate $\spinc$ structure $\s_0$ on $Y$. Then 
\begin{equation}\label{self-conjugate spinc extension}
\scspinc(W_j,\s_0)\neq \emptyset\quad \text{if and only if}\quad \spin(W_j,\s^1_0)\neq \emptyset\;\;\text{or}\;\,\spin(W_j,\s^2_0)\neq \emptyset.
\end{equation}
Since the cobordism $W_j$ is obtained by attaching $D^4$ to the manifold $W_j^0$ along $S^3$, see \eqref{E:W0}, we conclude that, for both $k = 1$ and $k = 2$,
\[
\begin{split}
\spin(W_j,\s^{k}_0)\neq \emptyset&\quad\text{if and only if}\quad \spin(Y_j,\s^{k}_0)\neq \emptyset\;\; \text{and}\;\, \spin(Y_{j+1},\s^{k}_0)\neq \emptyset\\
&\quad\text{if and only if}\quad \langle \lambda(\s^{k}_0),[\lambda_j]\rangle=\langle \lambda(\s^{k}_0),[\lambda_{j+1}]\rangle=1.
\end{split}
\]
With this understood, (1) follows from Lemma \ref{spin structure extension}. Since $\s_0^1$ and $\s_0^2$ correspond to the same $\spinc$ structure, we can write $\s^1_0 = \s^2_0 +(\omega_{\F_2})|_{\partial Y}$, where $\omega_{\F_2}$ is the mod 2 reduction of the generator $\omega \in H^1(Y;\Z)$. This implies that
\[
\lambda(\s^1_0)=\lambda(\s^2_0)+(\omega_{\F_2})|_{\partial Y}.
\]
It is not difficult to see that $(\omega_{\F_2})|_{\partial Y}\neq 0$ if and only if $t(Y)$ is odd. Therefore, if $t(Y)$ is odd, $\lambda(\s^1_0)\neq \lambda(\s^2_0)$. By Lemma \ref{spin structure extension}, one of $\s^{k}_0$ $(k=1,2)$ can be extended over $Y_0$ (and hence $W_2$), while the other one can be extended over $Y_1$ (and hence $W_1$). Claim (2) now follows from (\ref{self-conjugate spinc extension}). If $t(Y)$ is even, $\lambda(\s^1_0) = \lambda(\s^2_0)$. Define the sets
\[
A_j = \{ \s_0\mid \langle\lambda(\s^{k}_0),[\gamma_j]\rangle =0\;\;\text{for}\;\;k=1,2\},\quad j=0,1,
\]
then \eqref{decomposition of self-conjugate spinc} and (\ref{half of self-conjugate spinc can be extended}) follow directly from \eqref{exactly one can be extended}, and the equality $|A_0| = |A_1|$ can be verified as follows:
\[
|A_0| = \frac12\; |\spin(Y_0)| = \frac14\; |\spin(Y)| = \frac12\; |\spin(Y_1)| = |A_1|.
\]
\end{proof}

\begin{rmk}\label{R:odd t}
The disjoint decomposition $\scspinc(Y) = A_0\cup A_1$ of (\ref{decomposition of self-conjugate spinc}) with the additional properties \eqref{half of self-conjugate spinc can be extended} holds only for even $t(Y)$. We will extend it to the case of odd $t(Y)$ by choosing an arbitrary disjoint decomposition such that $|A_0| = |A_1|$.
\end{rmk}

In our next step toward the proof of Theorem \ref{same resolution data give equivalent spin-c system}, we will study the set of `relative characteristic vectors' defined as
\[
\Char(W_j,\s_0)=\{c_1(\s)\mid \s\in \spinc(W_j,\s_0)\}.
\]

\smallskip\noindent
To state the following set of results about this set, we need to recall the maps $i_{W_j}: H^2(W_j) \to H^2(Y)$ and $\partial_{W_j}: H^1(\partial Y)\to H^2(W_j)$ from the Mayer--Vietoris exact sequence (\ref{mv sequence1}).

\begin{lem} 
The set $\Char(W_j,\s_0)$ is a coset of\, $2\ker i_{W_j}$ inside of\, $\ker i_{W_j}$ which can be described precisely as follows:
\begin{enumerate}
\item Suppose $t(Y)$ is odd. Then, for any $\s_0\in \scspinc(Y)$, 
\begin{enumerate}
\item $\Char(W_j,\s_0)=2\ker i_{W_j} \text{ for }j=1,2$,
\item $\Char(W_0,\s_0)=\partial_{W_0}(\hat{m})+2\ker i_{W_0}$.
\end{enumerate}
\item Suppose $t(Y)$ is even and $\s_0\in A_0$. Then
\begin{enumerate}
\item $\Char(W_2,\s_0)=2\ker i_{W_2}$,
\item $\Char(W_1,\s_0)=\partial_{W_1}(\hat{\gamma}_2)+2\ker i_{W_1}$,
\item $\Char(W_0,\s_0)=\partial_{W_0}(\hat{\gamma}_0)+2\ker i_{W_0}$.
\end{enumerate}
\item Suppose $t(Y)$ is even and $\s_0\in A_1$. Then
\begin{enumerate}
\item $\Char(W_1,\s_0)=2\ker i_{W_2}$,
\item $\Char(W_2,\s_0)=\partial_{W_2}(\hat{\gamma}_2)+2\ker i_{W_2}$,
\item $\Char(W_0,\s_0)=\partial_{W_0}(\hat{\gamma}_1)+2\ker i_{W_0}$.
\end{enumerate}
\end{enumerate}
\end{lem}

\begin{proof}
We will only prove case (2) because cases (1) and (3) are similar. In case (2), we have even $t(Y)$ and $\s_0\in A_0$. Since $\scspinc(W_2, \s_0) \neq \emptyset$, the coset $\Char(W_2,\s_0)$ must contain zero. This proves (a). To prove (b), note that the image of $\Char(W_1,\s_0)$ under the restriction map $\ker i_{W_1} \to \ker i_{Y_1}$ does not contain zero because $\scspinc(Y_1,\s_0)=\emptyset$, while the image of  $\Char(W_1,\s_0)$ under the map $\ker i_{W_1}\rightarrow \ker i_{Y_2}$ contains zero because  $\scspinc(Y_2,\s_0)\neq \emptyset$. It is now not difficult to check that $\partial_{W_2}(\hat{\gamma}_2) + 2\ker i_{W_2}$ is the only one coset (of the four) satisfying these requirements. This proves (b). Case (c) is similar.
\end{proof}

Let $(L_0,L_1,L_2)$ and $(L'_0,L'_1,L'_2)$ be two admissible skein triangles with the same resolution data, and let us fix decompositions 
\[
\scspinc(Y) = A_0\cup A_1\quad\text{and}\quad \scspinc(Y')=A'_0\cup A'_1
\] 
as in Lemma \ref{lemma self-conjugate spinc extension} and Remark \ref{R:odd t}. Combining all of the above lemmas, we obtain the following result.

\begin{pro} \label{same resolution data gives same homology}
Let $\s_0\in A_n$ and $\s'_0\in A'_n$ with $n = 0$ or $n =1$. Then there exist isomorphisms
$\hat\xi_{j}:\ker i_{W_j}\rightarrow  \ker i_{W'_j}$ and $\xi_{j}:\ker i_{Y_j}\rightarrow  \ker i_{Y'_j}$ 
with the following properties:
\begin{enumerate}
\item $\hat\xi_j(\alpha)|_{Y'_j}=\xi_j(\alpha|_{Y_j})$ and $\hat\xi_{j+1}(\alpha)|_{Y'_{j+1}}=\xi_{j+1}(\alpha|_{Y_{j+1}})$;
\item For any $\alpha\in \ker i_{W_j}$ with $\alpha|_{\partial W_j}$ torsion, we have $\alpha^2 = (\hat\xi_j(\alpha))^2$;
\item $\hat\xi_j(\Char(W_j,\s_0))=\Char(W'_j,\s'_0)$,\;
\item Let $\rho_{j,j+1}$ be the isomorphism (\ref{homology composition}) and $\rho'_{j,j+1}$ the corresponding isomorphism for the skein triangle $(L'_0,L'_1,L'_2)$. Then, for any $\beta\in \ker i_{W_{j+2}}$ and any integer $k$ we have
\begin{gather*}
\hat\xi_j(\rho_{j,j+1}(\beta,k\cdot \PD[E_j])|_{W_j})=\rho'_{j,j+1}(\hat\xi_{j+2}(\beta),k\cdot \PD[E'_j])|_{W'_j}\quad\text{and} \\
\hat\xi_{j+1}(\rho_{j,j+1}(\beta,k\cdot\PD[E_j])|_{W_{j+1}})=\rho'_{j,j+1}(\hat\xi_{j+2}(\beta),k\cdot \PD[E'_j])|_{W'_{j+1}}.
\end{gather*}
\end{enumerate}
 \end{pro}

With all the necessary preparations now in place, we are finally ready to prove the main result of this subsection, Theorem \ref{same resolution data give equivalent spin-c system}.

\begin{proof}[Proof of Theorem \ref{same resolution data give equivalent spin-c system}] 
Let $A_0$, $A_1$, $A'_0$, and $A'_1$ be as above, and $\s_0\in A_{n}$ and $\s'_0\in A'_{n}$ for $n = 0$ or $n = 1$. We first pick any $\s_{W_0}\in\spinc(W_0,\s_0)$. It follows from Proposition \ref{same resolution data gives same homology}\,(3) that $\eta_0(c_1(\s_{W_0}))\in \Char(W'_0,\s'_0)$ and there exists $\s_{W'_0}$ such that $c_1(\s_{W'_0})=\eta_0 (c_1(\s_{W_0}))$. Denote by $\tilde\s$ the $\spinc$ structure on $\cptwobar$ with $c_1(\tilde{\s}) = \PD[E_j] = \PD[E'_j]$ and let
\[
\s_{W_j}=(\s_{W_0}\#\,\tilde{\s})|_{W_j}\quad\text{and}\quad \s_{W'_j}=(\s_{W'_0}\#\,\tilde{\s})|_{W'_j}\quad\text{for $j = 0, 1$},\quad\text{and}
\]
\[
\s_{Y_j}=\s_{W_j}|_{Y_j}\quad\text{and}\quad \s_{Y'_j}=\s_{W'_j}|_{Y'_j}\quad\text{for $j =0, 1, 2$}.
\]
Using Proposition \ref{same resolution data gives same homology}\,(4) one can show that
\[
\hat\xi_j(c_1(\s_{W_j}))=c_1(\s_{W'_j})\quad\text{and}\quad
\xi_j(c_1(\s_{Y_j}))=c_1(\s_{Y'_j}).
\]
We now define the map $\hat{\theta}_j: \spinc(W_j,\s_0)\rightarrow \spinc(W'_j,\s'_0)$ by the formula
\[
\hat{\theta}_j(\s_{W_j}+h)=\s_{W'_j}+\hat\xi_j(h)\quad \text{for}\;\; h\in \ker i_{W_j}
\]
and the map $\theta_j:\spinc(Y_j,\s_0)\rightarrow \spinc(Y'_j,\s'_0)$ by the formula
\[
\theta_j(\s_{Y_j}+h)=\s_{Y'_j}+\xi_j(h)\quad\text{for}\;\; h\in \ker i_{Y_j}.
\]
It is not difficult to verify that $\hat{\theta}_j$ and $\theta_j$ are compatible with the additional structures in Definition \ref{spin-c system} and that they provide the desired equivalence $\SS((L_0,L_1,L_2),\s_0)\sim \SS((L'_0,L'_1,L'_2),\s'_0)$.
\end{proof}


\subsection{Truncated Floer homology}
Recall that, according to Theorem \ref{exact triangle}, we have the following Floer exact triangle over the rationals
\begin{equation}\label{equation exact triangle}
\begin{CD}
\ldots @> {F_{W_0}} >> \widecheck{\HM}(Y_1,[\s_0]) @> {F_{W_1}} >>\widecheck{\HM}(Y_2,[\s_0]) @> {F_{W_2}} >> \widecheck{\HM}(Y_0,[\s_0]) @> {F_{W_0}} >> \ldots
\end{CD}
\end{equation}
Let us introduce the constants 
\[
c_0= b^{+}_2(W_2)+b_1(Y_0),\quad c_1=b^{+}_2(W_1)+b_1(Y_2),\quad\text{and}\quad c_2=0,  
\]
and use them to define `twisted' versions of the maps $\check{\tau}^j_{*}$ induced by the covering translations on $\widecheck{\HM}(Y_j,[\s_0])$ by the formula
\[
f_j=(-1)^{c_j}\cdot \check{\tau}^j_{*}: \widecheck{\HM}(Y_j,[\s_0])\rightarrow \widecheck{\HM}(Y_j,[\s_0]).
\]

\begin{lem}\label{exact triangle compatible with covering transformation}
The maps $f_j$ are compatible with the maps $F_{W_j}$ in the sense that
\begin{equation}\label{involution compatible with cobordism map}
f_{j+1}\circ F_{W_j}=F_{W_j}\circ f_j\quad\text{for}\quad j\in \Z/3.
\end{equation}
\end{lem}

\begin{proof}
Since the covering translation $\tau_{W_j}$ on $W_j$ extends the covering translations $\tau_{Y_j}$ and $\tau_{Y_{j+1}}$ on its boundary components, the functoriality of the monopole Floer homology and Lemma \ref{conjugation action} imply that
\begin{equation}\label{cobordism map for single spin-c}
\begin{split}
\check{\tau}^{j+1}_{*}\circ \widecheck{\HM}(W_j,\s) =(-1)^{b^{+}_2(W_j) +b_1(Y_{j+1})} \cdot \widecheck{\HM} & (W_j,\tau_{W_j}^{*}\s)  \circ \check{\tau}^j_{*} \\ &=(-1)^{b^{+}_2(W_j)+b_1(Y_{j+1})}\cdot \widecheck{\HM}(W_j,\bar{\s})\circ \check{\tau}^j_{*},
\end{split}
\end{equation}
for any $\s\in \spin^{c}(W_j,[\s_0])$. To explain the extra factor $(-1)^{b^{+}_2(W_j)+b_1(Y_{j+1})}$, we recall that the homology orientation, that is, an orientation of the vector space
\[
\bigwedge\nolimits^{\rm max}\left(H^1(W_j;\R)\oplus I^{+}(W_j;\R)\oplus H^1(Y_{j+1};\R)\right),
\]
is involved in the definition of the map $\widecheck{\HM}(W)$; see \cite[Definition 3.4.1]{Kronheimer-Mrowka}. Here, $I^{+}(W_j)$ stands for a maximum positive subspace for the intersection form on $\operatorname{im}(H^2(W_j,\partial W_j)\rightarrow H^2(W_j))$. By Lemma \ref{conjugation action}, the covering translation $\tau_{W_j}$ acts as the negative identity on the space 
\[
H^1(W_j;\R)\oplus I^{+}(W_j;\R)\oplus H^1(Y_{j+1};\R),
\]
thereby changing the homology orientation by the factor of 
\[
(-1)^{b^{+}_2(W_j)+b_1(W_j)+b_1(Y_{j+1})}
\]
Recall that the map $F_{W_j}$ was defined in Section \ref{S:triangle} by the formula
\[
F_{W_j}\; =\; \sum\limits_{\s\in \spinc(W_j,\s_0)}(-1)^{\mu(\s)}\cdot\widecheck{\HM}(W_j,\s),
\]
where $\mu(\s)$ is the $\F_2$-valued function defined in \eqref{defi:mu}. Therefore, in order to deduce \eqref{involution compatible with cobordism map} from \eqref{cobordism map for single spin-c}, we just need to check the relation
\[
b^{+}_2(W_j)+b_1(Y_{j+1}) +c_{j+1}+c_j\,=\, \mu(\s)-\mu(\bar{\s})\pmod 2
\]
for any $\s\in \spinc (W_j,\s_0)$. For $j=1$ and $j = 2$, this is immediate from Proposition \ref{conjugation on mu}. For $j=0$, this follows from Proposition \ref{conjugation on mu} and the identity
\begin{equation}\label{betti sum}
b_1(Y_0)+b_1(Y_1)+b_1(Y_2)+b^{+}_2(W_0)+b^{+}_2(W_1)+b^{+}_2(W_2)=1,
\end{equation}
which is a  consequence of Lemma \ref{homology of cobordism}. 
\end{proof}

Since $\widecheck{\HM}(Y_j,[\s_0])$ usually has infinite rank as a $\Z$--module, we will truncate it before discussing Lefschetz numbers.

\begin{defi}
For any rational number $q$ and $j \in \Z/3$, define the \emph{truncated monopole Floer homology} as
\[
\widecheck{\HM}_{\leq q}(Y_j,[\s_0]) = \Big(\mathop{\bigoplus}_{\substack{\s|_Y=\s_0\\
            c_1(\s)\;\text{torsion}}}\mathop{\bigoplus}_{a\leq q}\quad \widecheck{\HM}_{a}(Y_j,\s)\Big)
            \;\bigoplus\; \Big(\mathop{\bigoplus}_{\substack{\s|_{Y}=\s_0\\
            c_1(\s)\;\text{non-torsion}}}\widecheck{\HM}(Y_j,\s)\Big).
\]
Also define
\[
\widecheck{\HM}_{>q}(Y_j,[\s_0])\; =\; \widecheck{\HM}(Y_j,[\s_0])\Big/\,\widecheck{\HM}_{\leq q}(Y_j,[\s_0]).
\]
\end{defi}

We wish to find truncations of $\widecheck{\HM}(Y_j,[\s_0])$ for all $j \in \Z/3$ which are preserved by the maps $F_{W_j}$. To this end, recall the map
\[
\rho: \text{tor-spin}^{c}(Y_j)\rightarrow [0,2)
\]
defined by the formula
\[
\rho(\s) \equiv \frac 1 4\,(c_1(\hat\s)^2-\sgn(X)) \pmod 2
\]
for any choice of smooth compact $\spinc$ manifold $(X,\hat{\s})$ with the $\spinc$ boundary $(Y_j,\s)$ (see \cite{aps:II}).

\begin{defi}\label{truncated triangle}
Let $\tilde{\s} \in \text{sc-spin}^{c}(Y_2,\s_0)$ and choose an even integer $N > 0$ large enough so that, for each $j\in \Z/3$, the following two conditions are satisfied: 
\begin{enumerate}
\item\label{item truncation include reduced part} the natural map $\widecheck{\HM}_{\leq q}(Y_j,[\s_0])\rightarrow \HM^{\red}(Y_j,[\s_0])$ is surjective, and
\item\label{item truncation include image of reduced part} for any $\s\in \spin^{c}(Y_j,\s_0)$, there exists a finite set
\[
\{a_1,a_2,...,a_{n}\} \subset\widecheck{\HM}(Y_j,\s)
\]
representing a set of generators for $\HM_{\red}(Y_j,\s)$ as a quotient $\mathbb{Q}[U]$-module, such that 
\[
F_{W_j}(a_{i})\subset \widecheck{\HM}_{\leq N}(Y_{j+1},[\s_0]).
\]
(That this can be achieved follows from \cite[Lemma 25.3.1]{Kronheimer-Mrowka}).
\end{enumerate}
\emph{The truncated triangle with parameter} $(N,\tilde{\s})$ is then defined as the $3$-periodic chain complex
\begin{multline}\notag
\ldots\xrightarrow{F_{W_0}(N,\tilde{\s})} \widecheck{\HM}_{\leq N+\rho(\tilde{\s})+o(1)}(Y_1,[\s_0])\xrightarrow{F_{W_1}(N,\tilde{\s})}\widecheck{\HM}_{\leq N+\rho(\tilde{\s})+o(2)}(Y_2,[\s_0])\\ \xrightarrow{F_{W_2}(N,\tilde{\s})}\widecheck{\HM}_{\leq N+\rho(\tilde{\s})+o(0)}(Y_0,[\s_0])\xrightarrow{F_{W_0}(N,\tilde{\s})}\ldots
\end{multline}
where $F_{W_j}(N,\tilde{\s})$ is the restriction of $F_{W_j}$ and $o(j)$ is defined as follows:
\begin{enumerate}
\item if $b_1(Y_{n})=1$ for some $n\in \Z/3$ then $o(n)=o(n-1)=o(n+1)=0$;
\item if $b^{+}_2(W_{n})=1$ for some $n\in \Z/3$ then $o(n+1)=0$, $o(n)=1/2$, and $o(n-1) = 1/4$.
\end{enumerate}
(Note that by Lemma \ref{homology of cobordism}, exactly one of these two cases occurs). We denote this truncated triangle by $\mathfrak{C}^{\leq}(N,\tilde{\s})$. It is a subcomplex of the exact triangle (\ref{equation exact triangle}), and we denote by $\mathfrak{C}^{>}(N,\tilde{\s})$ the quotient complex.
\end{defi}

\begin{lem}
The image $F_{W_j}(\widecheck{\HM}_{\leq N+\rho(\tilde{\s})+o(j)}(Y_j,[\s_0]))$ is contained in 
\[
\widecheck{\HM}_{\leq N+\rho(\tilde{\s})+o(j+1)}(Y_{j+1},[\s_0]),
\]
and therefore the map $F_{W_j}(N,\tilde{\s})$ is well defined.
\end{lem}

\begin{proof}
We will only give the proof in the case when $b_1(Y_{n})=1$ for some $n\in \Z/3$ since the other case is similar. 

We will start with the map $F_{W_{n+1}}$. Since $b_1(Y_{n})=1$, we have $\gamma_n = l$, which implies that $|p_{n\pm 1}|=1$ and $p_{n-1}+p_{n+1}=-p_{n}=0$. Using (\ref{cup product formula}) and Lemma \ref{homology of cobordism}, we obtain
 \[
 \frac 1 4\,(c_1^2(\s)-2\chi(W_{n+1})-3\sigma(W_{n+1}))\; \leq\; \frac 1 4\,(-1-2+3) = 0
 \]
for any $\s\in \spinc(W_{n+1},[\s_0])$. Therefore, $F_{W_{n+1}}$ decreases the absolute grading, and the statement follows from the fact that $o(n+1)=o(n-1)$.

Let us now consider the map $F_{W_{n-1}}$. For a given $\s\in\spin^{c}(W_{n-1},\s_0)$, there are two possibilities:
\begin{itemize}
\item $\s|_{Y_{n}}$ is non-torsion. There is nothing to prove in this case because no truncation is done on $\widecheck{\HM}(Y_{n},\s|_{Y_{n}})$.
\item $\s|_{Y_{n}}$ is torsion. Since $b_1(Y_{n-1})=0$, the restriction $\s|_{\partial W_{n}}$ is torsion and the map $\widecheck{\HM}(W_{n},\s)$ has $\mathbb{Q}$--degree
\[
\frac 1 4\, (c_1(\s)^2-2\chi(W_{n})-3\sigma(W_{n})) = \frac 1 4\,(0-2-0) = -\frac12.
\]
The statement follows from the fact that $o(n)>o(n-1)- 1/2$.
\end{itemize}

Finally, consider the map $F_{W_{n}}$. For a given $\s\in\spin^{c}(W_{n},\s_0)$, there are again two possibilities:
\begin{itemize}
\item $\s|_{Y_{n}}$ is non-torsion. Then $\widecheck{\HM}(Y_{n},\s|_{Y_{n}})=\HM^{\red}(Y_{n},\s|_{Y_{n}})$ and the statement follows from Part (\ref{item truncation include image of reduced part}) of Definition \ref{truncated triangle}.
\item $\s|_{Y_{n}}$ is torsion. As in the corresponding case for $F_{W_{n-1}}$, the map $\widecheck{\HM}(W_{n},\s)$ has $\mathbb{Q}$--degree $-1/2$, and the statement follows from the fact that $o(n+1)>o(n)-1/2$.
\end{itemize}
\end{proof}

\noindent
Note that, in general, neither $\mathfrak{C}^{\leq}(N,\tilde{\s})$ nor $\mathfrak{C}^{>}(N,\tilde{\s})$ is exact. We denote their homology groups by $\{H_j^{\leq}(N,\tilde{\s})\}$ and $\{H_j^{>}(N,\tilde{\s})\}$, respectively. The absolute $\Z/2$ grading on $\widecheck{\HM}(Y_j,[\s_0])$ induces an absolute $ \Z/2$ grading on these homology groups. The maps $f_j$ give rise to involutions on both $\mathfrak{C}^{\leq}(N,\tilde{\s})$ and $\mathfrak{C}^{>}(N,\tilde{\s})$. We denote the corresponding chain maps by $f^{\leq}_j(N,\tilde{\s})$ and  $f^{>}_j(N,\tilde{\s})$. We also denote the induced maps on $H_j^{\leq}(N,\tilde{\s})$ and $H_j^{>}(N,\tilde{\s})$ by, respectively, $f^{\leq}_j(N,\tilde{\s})_{*}$ and $f^{>}_j(N,\tilde{\s})_{*}$. With a slight abuse of language, we will call all of these involutions \emph{covering involutions}.

\begin{lem}\label{truncated homology isomorphic to quotient homology}
For any $j\in \Z/3$, there is an isomorphism 
\[
\xi_j:H_j^{\leq }(N,\tilde{\s})\longrightarrow H_{j+1}^{>}(N,\tilde{\s})
\]
compatible with the covering involution. The map $\xi_j$ shifts the absolute $\Z/2$ grading by the same amount as the map $F_{W_j}$.
\end{lem}

\begin{proof}
Treat the exact triangle (\ref{equation exact triangle}) as a chain complex with trivial homology. Call this chain complex $\mathfrak{C}$. Then we have a short exact sequence
\[
\begin{CD}
0 @>>> \mathfrak{C}_{*}^{\leq}(N,\tilde{\s}) @>>> \mathfrak{C} @>>> \mathfrak{C}_{*}^{>}(N,\tilde{\s}) @>>> 0,
\end{CD}
\]
and the statement follows from the long exact sequence it generates in homology.
\end{proof}

Next, we will show that the chain complex $\mathfrak{C}^{>}(N,\tilde{\s})$ only depends on the equivalence class of the $\spinc$ system $\SS((L_0,L_1,L_2),\s_0)$. To make this statement precise, consider another admissible skein triple $(L'_0,L'_1,L'_2)$, and let $\s'_0$ be a self-conjugate $\spinc$ structure on $Y' = \Sigma(S^{3}\setminus B')$, where $B'$ is a small ball containing the resolved crossing. We suppose that there exists an equivalence
\[
\SS((L_0,L_1,L_2),\s_0)\sim \SS((L'_0,L'_1,L'_2),\s'_0)
\]

\smallskip\noindent
provided by the maps $\{\theta_j\}$ and $\{\hat{\theta}_j\}$ as in Definition \ref{equivalent spin-c system}. We write $\tilde\s' = \theta_j(\tilde{\s})$ and choose $N'$ large enough as to satisfy the conditions of Definition \ref{truncated triangle}. All of the above constructions can be repeated with $(L'_0,L'_1,L'_2)$ in place of $(L_0,L_1,L_2)$ and $(N',\tilde\s')$ in place of $(N,\tilde\s)$.


\begin{lem}\label{quotient chain complex isomorphic}
There is an isomorphism between the chain complexes $\mathfrak{C}^{>}(N,\tilde{\s})$ and $\mathfrak{C}^{>}(N',\tilde{\s}')$. This isomorphism preserves the covering involution, absolute $\Z/2$ grading and the relative $\mathbb{Q}$-grading.
\end{lem}

\begin{proof}
The chain complex $\mathfrak{C}^{>}(N,\tilde{\s})$ can be explicitly described in terms of the $\spinc$ structures and their Chern classes. For example, when $b_1(Y_j)=0$, each $\s\in \torspinc(Y_j,\s_0)$ contributes a summand $\mathcal{T}(\s) = \Z[U,U^{-1}]/\Z[U]$ to $\widecheck{\HM}_j^{>}(N,\tilde{s})$, supported in even $\Z/2$ grading. The smallest $\Q$--degree in this summand is given by 
\[
\min\{a \mid a\in 2\Z+\rho(s),\ a>N+\rho(\tilde{s})+o(j)\}.
\]
The covering involution interchanges $\mathcal{T}(\s)$ and $\mathcal{T}(\bar{\s})$. We have a similar description for the chain maps in $\mathfrak{C}^{>}(N,\tilde\s)$. The lemma can now be checked using this description and the corresponding description for $\mathfrak{C}^{>}(N',\tilde\s')$.
\end{proof}

\begin{cor}\label{equivalent spin-c systems have same homology}
For each $j\in \Z/3$, there exists an isomorphism $H_j^{\leq}(N',\tilde{\s}')\to H_j^{\leq }(N,\tilde{\s})$ that is compatible with the covering involution and preserves the absolute $\Z/2$ grading.
\end{cor}

\begin{proof}
This follows by combining Lemma \ref{quotient chain complex isomorphic} with Lemma \ref{truncated homology isomorphic to quotient homology}.
\end{proof}

\begin{lem}\label{lefs on chain map equals lefs on homology}
The following identity holds for the Lefschetz numbers on the truncated monopole Floer homology:
\begin{align*}
\Lef&(\check{\tau}^{\leq}_{Y_0,[\s_0]}  (N,\tilde{\s})) +\Lef(\check{\tau}^{\leq}_{Y_1,[\s_0]}(N,\tilde{\s}))-\Lef(\check{\tau}^{\leq}_{Y_2,[\s_0]}(N,\tilde{\s}))\\
&=(-1)^{b^{+}_2(W_2)+b_1(Y_0)}\Lef(f^{\leq}_0(N,\tilde{\s}))+
(-1)^{b^{+}_2(W_1)+b_1(Y_2)}\Lef(f^{\leq}_1(N,\tilde{\s})) -\Lef(f^{\leq}_2(N,\tilde{\s}))\\
&=(-1)^{b^{+}_2(W_2)+b_1(Y_0)}\Lef(f^{\leq}_0(N,\tilde{\s})_{*})+
(-1)^{b^{+}_1(W_1)+b_1(Y_2)}\Lef(f^{\leq}_1(N,\tilde{\s})_{*}) -\Lef(f^{\leq}_2(N,\tilde{\s})_{*}).
\end{align*}
A similar equality holds for $(L'_0,L'_1,L'_2)$.
\end{lem}
\begin{proof}
The first equality should be clear from the definition of $f_j$. The second equality is based on the following observation: by \cite[Proposition 2.5]{KMOS}, the map $F_{W_j}^{\leq}(N,\tilde{\s})$ preserves the absolute $\Z/2$ grading if and only if
\[
\frac12(\chi(W_j)+\sigma(W_j)-b_1(Y_j)+b_1(Y_{j+1})) = 0\pmod 2.
\]
Using Lemma \ref{homology of cobordism}, it is not difficult to check that this is equivalent to the condition
\[
b^{+}_2(W_j)+b_1(Y_{j+1}) = 0\pmod 2.
\]
This is exactly when the sign before $\Lef(f^{\leq}_j(N,\tilde{\s}))$ differs from the sign before $\Lef(f^{\leq}_{j+1}(N,\tilde{\s}))$ (see (\ref{betti sum}). As a result, this kind of alternating sum of the Lefschetz numbers for the chain map equals the corresponding sum for the induced map on homology.
\end{proof}

\begin{cor}\label{equivalent spin-c systems have same lefs number}
 We have the following equality of Lefschetz numbers
\begin{multline}\notag
\Lef(\check{\tau}^{\leq}_{Y_0,[\s_0]}(N,\tilde{\s}))+\Lef(\check{\tau}^{\leq}_{Y_1,[\s_0]}(N,\tilde{\s}))-\Lef(\check{\tau}^{\leq}_{Y_2,[\s_0]}(N,\tilde{\s}))\\
=\Lef(\check{\tau}^{\leq}_{Y'_0,[\s'_0]}(N',\tilde{\s}'))+\Lef(\check{\tau}^{\leq}_{Y'_1,[\s'_0]}(N',\tilde{\s}'))-\Lef(\check{\tau}^{\leq}_{Y'_2,[\s'_0]}(N',\tilde{\s}')).
\end{multline}
\end{cor}

\begin{proof}
This follows from Corollary \ref{equivalent spin-c systems have same homology} and Lemma \ref{lefs on chain map equals lefs on homology}.
\end{proof}

As our next step, we will study relations between the Lefschetz numbers $\Lef(\check{\tau}^{\leq}_{Y_j,[\s_0]}(N,\tilde{\s}))$ and the corresponding Lefschetz numbers on the reduced Floer homology.

\begin{defi}\label{normalized lefs}
For any $\s\in \scspinc(Y_j)$, define the \emph{normalized Lefschetz number} $\nLef (Y_j,\s)$ of the map
$$
\tau^{\red}_{Y_j,\s}:\HM^{\red}(Y_j,\s)\rightarrow \HM^{\red}(Y_j,\s)
$$
as follows:
\begin{itemize}
\item if $b_1(Y_j)=0$, we let 
\[
\nLef(Y_j,\s)=\Lef(\tau^{\red}_{Y_j,\s})+h(Y_j,\s).
\]
\item if $b_1(Y_j)=1$, recall that (as in Heegaard Floer theory~\cite[\S4.2]{oz:boundary}) there are two Fr\o yshov invariants $h_0(Y_j,\s)$ and $h_1(Y_j,\s)$ (see the proof of Lemma \ref{normalized lefs and truncated lefs} below). We let
\[
\nLef(Y_j,\s)=\Lef(\tau^{\red}_{Y_j,\s})+h_0(Y_j,\s)+h_1(Y_j,\s).
\]
\end{itemize}
\end{defi}

\begin{lem}\label{normalized lefs and truncated lefs}
For any $\s\in \scspinc(Y_j)$ and any rational number $q$, consider the map
\[
\check{\tau}^{\leq q}_{Y_j,\s}:\; \mathop{\bigoplus}_{a\leq q}\; \widecheck{\HM}_{a}(Y_j,\s)\longrightarrow \mathop{\bigoplus}_{a\leq q}\; \widecheck{\HM}_{a}(Y_j,\s).
\]
For all sufficiently large $q$, its Lefschetz number satisfies the equality 
\begin{equation}\label{equation normalized lefs and truncated lefs}
\Lef(\check{\tau}^{\leq q}_{Y_j,\s})-2^{b_1(Y_j)-1}q\;=\;\nLef(Y_j,\s)+C(b_1(Y_j),q-\rho(\s)),
\end{equation}
where $C(b_1(Y_j),q-\rho(\s))$ is constant depending only on $b_1(Y_j)$ and the mod 2 reduction of $q-\rho(\s)$ in $\mathbb{Q}/2\Z$.
\end{lem}

\begin{proof}
Let us assume that $b_1(Y_j)=1$; the case of $b_1(Y_j)=0$ is similar (and easier). We have the following (non-canonical) decomposition for  $\widecheck{\HM}(Y,\s)$:
\[
(\mathbb{Q}[U,U^{-1}]/\mathbb{Q}[U])_{-2h_0(Y,\s)}\,\oplus\, (\mathbb{Q}[U,U^{-1}]/\mathbb{Q}[U])_{-2h_1(Y,\s)}\,\oplus\, \HM^{\red}(Y_j,\s),
\]
with the lower indices indicating the absolute grading of the bottom of the $U$--tail. Regarding  the absolute $\Z/2$ grading, the first summand is supported in the even grading while the second summand is supported in the odd grading. With respect to this decomposition, the map $\check{\tau}_{Y_j,\s}$ is given by the matrix
\[
\begin{bmatrix}
   1 & \phantom{-}0  &  * \\
   0 & -1 &  * \\
   0 & \phantom{-}0  &  \tau^{\red}_{Y_j,\s}
\end{bmatrix}
\]

\smallskip\noindent
(the action on the second summand is $-1$ because $\tau^{*}_{Y_j}$ acts as negative identity
 on $H^1(Y_j;\mathbb{R})$; see \cite[Theorem 35.1.1]{Kronheimer-Mrowka}). 
 Therefore, if $q$ is large enough so that $\HM^{\red}(Y_j,\s)$ is supported in degree less than $q$, we obtain
\begin{multline}\notag
\Lef(\check{\tau}^{\leq q}_{Y_j,\s})-q-\nLef(Y_j,\s) = \left|\{k\in \Z^{\geq 0}\mid -2h_0(Y,\s)+2k\leq q\}\right| - q/2\; + \\ \left|\{k\in \Z^{\geq 0}\mid -2h_1(Y,\s)+2k\leq q\}\right|-q/2.
\end{multline}
Clearly, this number only depends on the mod 2 reduction of $q+2h_0(Y,\s)$ and $q+2h_1(Y,\s)$. To complete the proof, we observe that $\rho(\s)+2h_0(Y,\s)\in 2\Z$ and $\rho(\s)+2h_1(Y,\s)\in 2\Z+1$, which follows directly from the definition of the absolute grading in monopole Floer homology.
\end{proof}

By setting $q=N+o(j)+\rho(\tilde{\s})$ and taking the sum of the equalities (\ref{equation normalized lefs and truncated lefs}) over all  $\spinc$-structures $\s\in \text{sc-spin}^{c}(Y_j,\s_0)$, we obtain the equality
\medskip
\begin{equation}\label{truncated lefs and normalized lefs 2}
\begin{split}
\Lef(\check{\tau}^{\leq}_{Y_j,[\s_0]}(N,\tilde{\s}))- \sum\limits_{\s\in \text{sc-spin}^{c}(Y_j,\s_0)} &\nLef(Y_j,\s) \\
&= 2^{b_1(Y_j)-1}\cdot \left|\text{sc-spin}^{c}(Y_j,\s_0)\right|\cdot (N+\rho(\tilde{\s}))+C,
\end{split}
\end{equation}
where $C$ is a constant depending on $b_1(Y_j)$, $o(j)$, $\left| \scspinc (Y_j, \s_0)\right|$, and the mod 2 reduction of $\rho(\s)-\rho(\tilde{\s})$ for $\s\in \scspinc (Y_j,\s_0)$.

\begin{cor} If $\SS((L_0,L_1,L_2),\s_0)\sim \SS((L'_0,L'_1,L'_2),\s'_0)$ then
\medskip
\begin{equation*}
\begin{split}
&\sum\limits_{\s\in \scspinc(Y_0,\s_0)}\nLef(Y_0,\s)+\sum\limits_{\s\in \scspinc(Y_1,\s_0)}\nLef(Y_1,\s)-\sum\limits_{\s\in \scspinc(Y_2,\s_0)}\nLef(Y_2,\s) = \\
&\sum\limits_{\s\in \scspinc(Y'_0,\s'_0)}\nLef(Y'_0,\s)+\sum\limits_{\s\in \scspinc(Y'_1,\s'_0)}\nLef(Y'_1,\s)-\sum\limits_{\s\in \scspinc(Y'_2,\s'_0)}\nLef(Y'_2,\s)
\end{split}
\end{equation*}

\noindent In addition,
\begin{align}
\frac1{2^{|L_2|-2}}\,\Big(\sum\limits_{\s\in \scspinc(Y_0)}\nLef(Y_0,\s)+ \sum\limits_{\s\in \scspinc (Y_1)} & \nLef(Y_1,\s)-\sum\limits_{\s\in \scspinc(Y_2)}\nLef(Y_2,\s)\Big) \label{skein for normalized lefs} \\ 
&\qquad = C(t(Y),s(Y),(p_0,q_0),(p_1,q_1)).\notag
\end{align}
\end{cor}

\begin{proof}
Since $b_1(Y_j)=b_1(Y'_j)$, $o(j)=o'(j)$, $|\scspinc (Y_j,\s_0)|=|\scspinc (Y'_j,\s'_0)|$, and 
\[
\rho(\s)-\rho(\tilde{\s})\,=\, \rho(\theta(\s))-\rho(\tilde{\s}') \pmod 2,
\]
the constant $C$ in (\ref{truncated lefs and normalized lefs 2}) equals the corresponding constant for $Y'$. Now we add the equalities (\ref{truncated lefs and normalized lefs 2}) for $Y_0$, $Y_1$ and subtract the one for $Y_2$. By comparing the result with the corresponding result for $Y'_j$ and applying (\ref{spin structure count}) and Corollary \ref{equivalent spin-c systems have same lefs number}, we finish the proof of the first claim. The second claim follows easily from (1) and Theorem \ref{same resolution data give equivalent spin-c system}.
\end{proof}

We are now ready to prove the main theorem of this section.

\begin{proof}[Proof of Theorem \ref{skein with error term}] 
For any ramifiable link $L_j$, it follows from formula \eqref{Murasugi invariant} for the Murasugi signature that
\begin{equation}\label{chi decomposition}
\frac1{2^{|L_j|-1}}\sum\limits_{\s\in \scspinc(Y_j)}\nLef(Y_j,\s)\,-\,\frac 1 8\,\xi(L_j)\, =\, \chi(L_j).
\end{equation}
Therefore, \eqref{skein with error term 1} follows from \eqref{skein for normalized lefs} and Lemma \ref{skein for signature}. This proves statement (1) of Theorem \ref{skein with error term}.

The first assertion of statement (2) follows from Lemma \ref{at most one not ramifiable}. To prove the second assertion, suppose that $L_2$ is not ramifiable. Then $\gamma_2$ represents zero elements in both $H_1 (Y;\mathbb F_2)$ and  $H_1 (Y;\mathbb Z)$. Let $(t(Y),s(Y),(p_0,q_0),(p_1,q_1))$ be the resolution data for $(L_0,L_1,L_2)$. Then $s(Y)=0$ and $(p_2,q_2)=(0,1)$, which implies that $(p_1,q_1)=(-p_0,-q_0-1)$. It follows from Lemma \ref{at most one not ramifiable} that $(\bar{L}_1,L_0,L_2)$ forms an admissible skein triangle, and one can check that its resolution data is $(t(Y),s(Y),(p_0,q_0-1),(-p_0,-q_0))$. The equality (\ref{skein for normalized lefs}) now reads
\begin{align}
\frac1{2^{|L_2|-2}}\,\Big(\sum\limits_{\s\in \scspinc(Y_0)}\nLef(Y_0,\s) + \sum\limits_{\s\in \scspinc(\bar{Y}_1)} & \nLef(\overline{Y}_1,\s) - \sum\limits_{\s\in \scspinc(Y_2)}\nLef(Y_2,\s)\Big) \label{skein for normalized lefs 2} \\
& \quad = C(t(Y),s(Y),(p_0,q_0-1),(-p_0,-q_0)), \notag
\end{align}
where $\overline{Y}_1$ stands for the double branched cover of $S^{3}$ with branch set $\bar{L}_1$. Subtracting \eqref{skein for normalized lefs} from \eqref{skein for normalized lefs 2}, we obtain
\begin{align*}
 \frac1{2^{|L_1|-1}}\, \Big(\sum \limits_{\s\in \scspinc(\bar{Y}_1)} & \nLef(\bar{Y}_1,\s)-\sum\limits_{\s\in \scspinc(Y_1)}\nLef(Y_1,\s)\Big) \\
= C ( & t(Y), s(Y),(p_0,q_0-1),(-p_0,-q_0))-C(t(Y),s(Y),(p_0,q_0),(p_1,q_1)).
\end{align*}
Combining this with Lemma \ref{skein for signature} and \eqref{chi decomposition}, we finish the proof of \eqref{skein with error term 2} in the case of $n = 2$. The proofs of \eqref{skein with error term 3} for $n = 2$, and of both \eqref{skein with error term 2} and \eqref{skein with error term 3} for $n=0$ and $n = 1$ are similar.
\end{proof}


\section{Vanishing of the universal constants}\label{S:universal}
In this section, we will prove that the constants $C(t(Y),s(Y),(p_0,q_0),(p_1,q_1))$ and $C^{\pm}_j (t(Y),\allowbreak s(Y),(p_0,q_0),(p_1,q_1))$ in Theorem \ref{skein with error term} vanish for all resolution data. The cyclic symmetry of skein triangles will then imply Proposition \ref{skein relation for chi} and therefore finish the proof of Theorem \ref{main theorem}.

\begin{defi}\label{admissible tuple}
A six-tuple $(t,s,(p_0,q_0),(p_1,q_1))$, where $t$ is a positive integer, $s\in \Z/2$, and $p_j, q_j\in \Z$, $j = 0,1$, is called \emph{admissible} is the following four conditions are satisfied:
\begin{enumerate}
\item $p_0 q_1 - p_1 q_0 = 1$,
\item $s=0$ if $t$ is odd,
\item $p_2 = - p_0 - p_1$ is even when $s=0$, and $q_2 = - q_0 - q_1$ is even when $s=1$.
\end{enumerate}
\end{defi}

\begin{lem}
The resolution data $(t(Y),s(Y),(p_0,q_0),(p_1,q_1))$ associated with an admissible skein triangle $(L_0,L_1,L_2)$ as in Definition \ref{D:res} is an admissible six-tuple.
\end{lem}

\begin{proof} This can be easily verified: (1) corresponds to the requirement that $\#(\gamma_0\cap \gamma_1) = -1$, (2) follows from the fact that $l$ is an $\mathbb F_2$ longitude when $t(Y)$ is odd, and (3) is true because $\gamma_2$ is an $\F_2$ longitude by part (1) of Definition \ref{D:skein}.
\end{proof}

\begin{thm}\label{vanishing of constant}
Every admissible six-tuple $(t,s,(p_0,q_0),(p_1,q_1))$ is the resolution data of an admissible skein triangle. Furthermore,
\begin{enumerate}
\item if $(p_j,q_j)\neq (0,1)$ for all $j\in \Z/3$ then $C(t,s,(p_0,q_0),(p_1,q_1)) = 0$, and 
\item if $(p_j,q_j)=(0,1)$ for some $j\in \Z/3$\, then\, $C^{\pm}_j(t,s,(p_0,q_0),(p_1,q_1)) = 0$.
\end{enumerate}
\end{thm}

Our proof of Theorem \ref{vanishing of constant} is inspired by the proof of \cite[Theorem 7.5]{theta divisor}. The idea is roughly as follows: starting with two-bridge links, which are alternating and hence have vanishing $\chi$, we will generate sufficiently many examples of Montesinos links with vanishing $\chi$ to cover all possible admissible six-tuples. We will then apply Theorem \ref{skein with error term} to conclude that the constants in question are all zero. 

\begin{pro}\label{t=1 case}
Theorem \ref{vanishing of constant} holds for all admissible six-tuples with $t=1$.
\end{pro}

\begin{proof}
Observe that since $t = 1$ is odd, we automatically have $s = 0$ by Definition \ref{admissible tuple}, so the admissible six-tuples at hand are of the form $(1,0,(p_0,q_0),(p_1,q_1))$ with $p_0 q_1 - p_1 q_0 = 1$ and even $p_2 = - p_0 - p_1$. Let us consider an unknot in $S^3$ with $Y = S^1 \times D^2$ and the standard boundary framing $(m,l)$ on $\partial Y$, which has $t = 1$ and $s = 0$. For every $j \in \Z/3$, the manifold $Y_j$ obtained by the $p_j/q_j$ surgery on the unknot is a lens space of the form $Y_j = \Sigma(L_j)$, where $L_j$ is a two-bridge link. The links $\bar L_j$ are also two-bridge, and hence alternating. It then follows from \cite[Theorem 1.2]{manolescu-owens} and \cite[Lemma 3.4]{Donald-Owens}, combined with the relation cited in Remark \ref{R:HM=HF} between the Heegaard Floer and monopole correction terms
that $\chi(L_j) = \chi (\bar L_j) = 0$. 
The result will now follow from Theorem \ref{skein with error term} as soon as we show that $(L_0,L_1,L_2)$ is an admissible skein triangle. But the latter is a special case of the more general result proved below in Lemma \ref{L:res}.
\end{proof}

To continue, we will introduce some notation. Choose three distinct circle fibers in $S^2 \times S^1$ and remove their disjoint open tubular neighborhoods. The resulting manifold will be called $N$. The tori $T_j$, $j = 1, 2, 3$, on the boundary of $N$ have natural framings $(x_j,h)$, where $h$ is the circle fiber and the curves $x_1$, $x_2$, and $x_3$ co-bound a section of the product circle bundle.

Denote by $\Q^+ = \Q\, \cup\, \{\infty\}$ the extended set of rational numbers, with the convention that $\infty = 1/0$. Given three numbers $a_j/b_j \in \Q^+$ with co-prime $(a_j,b_j)$, $j = 1, 2, 3$, denote by $Y(a_1/b_1,a_2/b_2, a_3/b_3)$ the closed manifold obtained by attaching to $N$ three solid tori along their boundaries so that their meridians match the curves $a_j x_j + b_j h$. A direct calculation shows that the first homology group of $Y(a_1/b_1,a_2/b_2, a_3/b_3)$ is finite if and only if $b_1 a_2 a_3 + a_1 b_2 a_3 + a_1 a_2 b_3\; \neq\; 0$, in which case
\begin{equation}\label{E:order}
|H_1(Y(a_1/b_1,a_2/b_2, a_3/b_3);\Z)|\; = \;|b_1 a_2 a_3 + a_1 b_2 a_3 + a_1 a_2 b_3|.
\end{equation}
A surgery description of  $Y(a_1/b_1,a_2/b_2, a_3/b_3)$ is shown in Figure \ref{F:fig15}. Denote by $Y(a_1/b_1, a_2/b_2,\bullet)$ the manifold with a single boundary component obtained by attaching to $N$ just the first two solid tori.

\begin{figure}[!ht]
\centering
\psfrag1{$a_1/b_1$}
\psfrag2{$a_2/b_2$}
\psfrag{n}{$a_3/b_3$}
\psfrag0{$0$}
\includegraphics{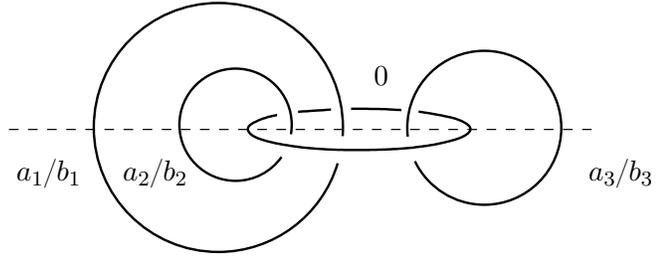}
\caption{The manifold $Y(a_1/b_1,a_2/b_2,a_3/b_3)$}
\label{F:fig15}
\end{figure}

The $180^{\circ}$ rotation with respect to the dotted line in Figure \ref{F:fig15} makes $Y(a_1/b_1,a_2/b_2,a_3/b_3)$ into a double branched cover over $S^3$ with branch set the Montesinos link $K(a_1/b_1, a_2/b_2, a_3/b_3)$ pictured in Figure \ref{F:fig11}.

\begin{figure}[!ht]
\centering
\psfrag1{$a_1/b_1$}
\psfrag2{$a_2/b_2$}
\psfrag{3}{$a_3/b_3$}
\includegraphics{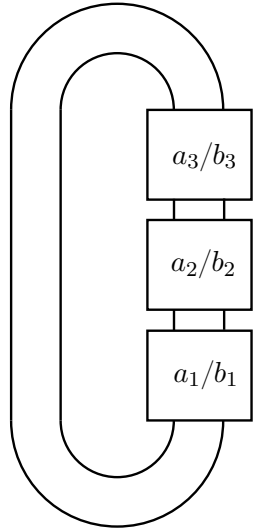}
\caption{The Montesinos link $K(a_1/b_1,a_2/b_2,a_3/b_3)$}
\label{F:fig11}
\end{figure}

\noindent 
Each of the boxes marked $a_j/b_j$ in the figure stands for the rational tangle $T(a_j/b_j)$ obtained from a continued fraction decomposition 
\begin{equation}\label{E:c-frac}
a_j/b_j = [t_1,\ldots,t_{k_j}] =   t_1 - \cfrac[l]1
                     { t_2 - \cfrac[l]1 
                     { \cdots - \cfrac[l]1 
                     { t_{k_j}}}}
\end{equation}
by applying consecutive twists to neighboring endpoints starting from two unknotted and unlinked arcs. Our conventions for rational tangles should be clear from the examples in Figure \ref{F:fig14}.

\medskip

\begin{figure}[!ht]
\centering
\psfrag1{$3 = [3]$}
\psfrag2{$\hspace{-0.2in} -7/3 = [-2,3]$}
\psfrag{3}{$31/7 = [4,-2,3]$}
\psfrag0{$0$}
\includegraphics{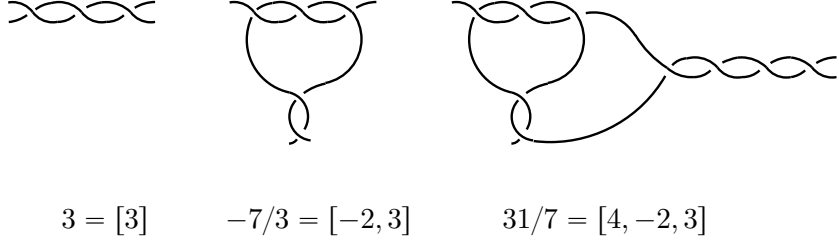}
\caption{Examples of rational tangles}
\label{F:fig14}
\end{figure}

We will study skein triangles formed by these Montesinos links. Given $p_0/q_0$ and $p_1/q_1 \in \mathbb{Q}^{+}$ with co-prime $(p_0,q_0)$ and $(p_1,q_1)$, define the \emph{distance} between them by the formula
\begin{equation}\label{E:distance}
\Delta ({p_0}/{q_0},{p_1}/{q_1}) = |p_0q_1-p_1q_0|.
\end{equation}
We will say that three points in $\mathbb{Q}^{+}$ form a \emph{triangle} if the distance between any two of them is equal to $1$. Two triangles $T_1$ and $T_2$ are called \emph{adjacent} if the intersection $T_1\cap T_2$ consists of exactly two points.

\begin{lem}\label{triangle induction}
For any $r,s,t\in \mathbb{Q}^{+}$ with $\Delta(r,s)=1$, there exists a chain of triangles $S_0$, $S_1,\ldots, S_n$ such that $r,s \in S_0$, $t\in S_n$, and $S_i$ is adjacent to $S_{i+1}$ for all $i = 0,\ldots,n$. 
\end{lem}

\begin{proof}
The modular group $PSL(2,\Z)$ acts on the set $\mathbb{Q}^{+}$ by linear fractional transformations 
\[
\begin{pmatrix}a&b\\c&d\end{pmatrix}\cdot\left(\frac{p}{q}\right)\; =\; \frac{ap+bq}{cp+dq}.
\]
This action preserves the distance \eqref{E:distance} and hence sends triangles to triangles. It follows from the general properties of the modular group (and can also be checked directly) that, for any pair $r, s \in \Q^+$, there exists $A \in PSL(2,\Z)$ such that $A\cdot r= 0$ and $A\cdot s = \infty$. Therefore, we may assume without loss of generality that  $r = 0$ and $s = \infty$.

Observe that there are exactly two choices for the triangle $S_0$ with vertices $0$ and $\infty$: in one of these triangles, the third vertex is $1$, and in the other $-1$. Therefore, we will find a chain of triangles $S_0$, $S_1, \ldots, S_n$ connecting $0$ and $\infty$ to $t = p/q$ as soon as we find a chain of triangles $S_1, \ldots, S_n$ connecting $0$ and $\pm 1$ to $t = p/q$. First suppose that the triangle $S_0$ has vertices $0$, $1$, and $\infty$. The matrix 
\[
\begin{pmatrix*}[r]
1 & -1\\ 1 & 0
\end{pmatrix*}
\]
sends $0$, $1$, and $p/q$ into $\infty$, $0$, and $(p-q)/q$, respectively, and turns the problem at hand into the problem of finding a chain of triangles $S_1$,\ldots, $S_n$ connecting $0$ and $\infty$ to $(p-q)/p$. This is, of course, the original problem with $t = p/q$ replaced by $t = (p-q)/p$. Similarly, the matrix
\[
\begin{pmatrix*}[r]
1 & 0 \\ -1&1
\end{pmatrix*}
\]
sends $0$, $1$, and $p/q$ into $0$, $\infty$, and $p/(q-p)$, respectively, thereby replacing $t = p/q$ by $t = p/(q-p)$. If the triangle $S_0$ has vertices $0$, $-1$ , and $\infty$, the matrices 
\[
\begin{pmatrix*}[r]
1 & 1\\ 1 & 0
\end{pmatrix*}
\quad \text{and} \quad
\begin{pmatrix*}[r]
1 & 0 \\ 1 & 1
\end{pmatrix*}
\]
can be used to replace $t = p/q$ with $t = (p+q)/p$ and $t = p/(p+q)$, respectively. In summary, $t = p/q$ can be replaced with any one of the four fractions $(p \pm q)/q$ and $p/(q\pm p)$. One can find a sequence of such replacements making any $t = p/q$ into $t = 1$, for which there is an obvious solution. 
\end{proof}

\begin{lem}\label{L:res}
For any $p$, $q\in \mathbb{Q}^{+}$ and any adjacent triangles $\{r,s,t\}$ and $\{r,s,t'\}$, one can find a planar projection of the link $K(p,q,t)$ and a crossing $c$ such that
\begin{itemize}
\item the two resolutions of $K(p,q,t)$ at the crossing $c$ are $K(p,q,r)$ and $K(p,q,s)$, and
\item the link $K(p,q,t)$ with the crossing $c$ changed is $K(p,q,t')$.
\end{itemize}
In particular, each of the sets 
\[
\{K(p,q,r),K(p,q,s),K(p,q,t)\}\quad\text{and}\quad \{K(p,q,r),K(p,q,s),K(p,q,t')\}
\]
forms an admissible skein triangle, possibly after a permutation. For both skein triangles, the manifold $Y$ with torus boundary is just $Y(p,q,\bullet)$. 
\end{lem}

\begin{proof}
Let $B^3$ be a 3-ball in $S^{3}$ which contains the third rational tangle in all of the Montesinos links at hand. Identify its boundary $\partial B^3$ with the quotient $(\R^2/\Z^2)/\pm 1$ of the torus $\R^2/\Z^2$ by the hyperelliptic involution. The standard action of $SL(2,\Z)$ on the plane $\R^2$ induces an action of $PSL(2,\Z)$ on $\partial B^3$ which permutes the points $(0,0)$, $(0,1/2)$, $(1/2,0)$, and $(1/2,1/2)$. Every $A \in PSL(2,\Z)$ gives a homeomorphism $A: \partial B^3 \to \partial B^3$ which extends to a homeomorphism $f_A: B^3 \to B^3$ by the coning construction, $f_A\, (t\cdot x) = t\cdot A(x)$. The homeomorphism $f_A$ sends a rational tangle $T(\ell)$ to a rational tangle $T(A\cdot \ell\,)$, which can be seen by factoring $A$ into a product of Dehn twists 
\medskip
\[
A\; =\;\begin{pmatrix*}[r] q & s \\ p & r \end{pmatrix*}\;=\;
\begin{pmatrix*}[l] 1 & 0 \\ t_1 & 1 \end{pmatrix*}
\begin{pmatrix*}[r] 0 & 1 \\   -1 & 0 \end{pmatrix*}
\begin{pmatrix*}[l] 1 & 0 \\ t_2 & 1 \end{pmatrix*}
\begin{pmatrix*}[r] 0 & 1 \\   -1 & 0 \end{pmatrix*} \cdots
\begin{pmatrix*}[r] 0 & 1 \\   -1 & 0 \end{pmatrix*}
\begin{pmatrix*}[l] 1 & 0 \\ t_k & 0 \end{pmatrix*}
\]

\medskip\noindent
using a continued fraction $p/q = [t_1,\ldots,t_k]$ as in \eqref{E:c-frac}. Now, given $r,s\in \Q^{+}$, there exists $A\in PSL(2,\Z)$ such that $A\cdot r = 0$, $A\cdot s = \infty$, and then necessarily $\{A\cdot t,A\cdot t'\} = \{\pm 1\}$. By the coning construction, $f_A$ extends to the exterior of $B^3$, resulting in a homeomorphism of $S^3$. This homeomorphism turns the original tangle decompositions into tangle decompositions of the form  
\[
K(p,q,r) = T'\cup\,T(0),\quad K(p,q,s) = T'\cup\,T(\infty),\quad\text{and}
\]
\[
\{K(p,q,t),\,K(p,q,t')\} = \{\,T'\cup\,T(1),\,T'\cup\,T(-1)\},
\]
where $T'$ is a certain tangle in the exterior of $B^3$. The conclusion of the lemma is now clear.
\end{proof}

\begin{pro}\label{one coefficient is integer}
Suppose  the link $K(r_1,r_2,r_{3})$ is ramifiable and $1/r_j$ is an integer or infinity for some $j$. Then $\chi(K(r_1,r_2,r_{3}))=0.$
\end{pro}

\begin{proof}
In this case, $K(r_1,r_2,r_3)$ is a two-bridge link and, in particular, it is alternating. The result now follows from \cite{manolescu-owens} and \cite{Donald-Owens} as in the proof of Proposition \ref{t=1 case}.
\end{proof}

\begin{pro}\label{switching singular fiber}
For any $p,q,r\in \mathbb{Q}^{+}$, suppose that $K(p,q,r)$ is ramifiable and Theorem \ref{vanishing of constant} holds for all admissible six-tuples with $t = t(Y(p,q,\bullet))$. Then $\chi(K(p,q,r))=0.$
\end{pro}

\begin{proof}
Use formula \eqref{E:order} to find a positive integer $k$ such that both $K(p,q,1/k)$ and $K(p,q,\allowbreak 1/(k+1))$ are ramifiable. It follows from Proposition \ref{one coefficient is integer} that
\[
\chi(K(p,q,1/k) = \chi(K(p,q,1/(k+1)) = 0.
\]
Since $\Delta(1/k,1/(k+1)) = 1$, Lemma \ref{triangle induction} supplies us with a chain of triangles $S_0$, $S_1,\ldots, S_n$ such that $k, k+1\in S_0$, $r \in S_n$, and $S_i$ is adjacent to $S_{i+1}$ for all $i= 0,\ldots,n$. We claim that for any $m = 0,\ldots,n$ and $s \in S_{m}$ such that $K(p,q,s)$ is ramifiable, 
\[
\chi(K(p,q,s))=0.
\]
We will proceed by induction on $m$. First, suppose that $m = 0$. If $s = k$ or $k+1$, the claim follows from Proposition \ref{one coefficient is integer}; otherwise, it follows from Theorem \ref{skein with error term}\,(1). Next, suppose that the claim holds for $m$ and prove it for $m+1$. Write $S_{m+1} = \{s,u,v\}$ and suppose that $K(p,q,s)$ is ramifiable. If $s \in S_m$ then the claim follows from the induction hypothesis. Otherwise, write $S_m = \{u,v,w\}$ and consider two possibilities. One possibility is that both $K(p,q,u)$ and $K(p,q,v)$ are ramifiable. Then $\chi(K(p,q,u)) = \chi(K(p,q,v)) = 0$ by the induction hypothesis and the vanishing of $\chi(K(p,q,s))$ follows from Theorem \ref{skein with error term}\,(1). The other possibility is that one of $K(p,q,u)$ and $K(p,q,v)$ is not ramifiable. Then by Theorem \ref{skein with error term}\,(2), the link $K(p,q,w)$ is ramifiable and $\chi(K(p,q,s)) = \chi(K(p,q,w)) = 0$ by the induction hypothesis.
\end{proof}

The following lemma will be helpful in computing $t(Y(n,a/b,\bullet))$. It uses the notation $
\div_2(m) = \max\,\{c\in\mathbb{N}\mid 2^{c}\;\text{divides}\; m\}$.

\begin{lem}\label{divisibility computation}
For any integer $n$ and co-prime integers $a$ and $b$ with $b_1(Y(n,a/b,\bullet)) = 1$, the integer $t(Y(n,a/b,\bullet))$ is a divisor of $n$. In particular, $t(Y(n,a/b,\bullet))\leq |n|$, with $t(Y(n,a/b,\bullet)) = |n|$ if and only there exists an integer $k$ such that $a=kn$ and $n$ divides $k+b$. Furthermore, if $\div_2(n)\neq \div_2(a)$ then the integer $t(Y(n,a/b,\bullet))$ is odd.
\end{lem}

\begin{proof}
We will use the notation $Y = Y(n,a/b,\bullet)$. The homology group $H_1 (Y;\Z)$ is generated by the homology classes $[x_1]$, $[x_2]$, $[x_{3}]$, and $[h]$ subject to the relations
\[
n\cdot [x_1] + [h] = 0,\quad a\cdot [x_2] + b\cdot [h] = 0,\quad [x_1] + [x_2] + [x_3] = 0.
\]
One can easily see that the kernel of the map $H_1(\partial Y; \Z) \to H_1 (Y; \Z)$ is generated by the homology class 
\[
\frac{a+bn}{\gcd(n,a)} \cdot [h] - \frac{na}{\gcd(n,a)} \cdot [x_3],
\]

\medskip\noindent
and therefore
\[
t(Y) = \gcd\left(\frac{a+bn}{\gcd(n,a)}\, ,\;\frac{na}{\gcd(n,a)}\right).
\]

\medskip\noindent
To prove the first statement of the lemma, write $n = \gcd(n,a)\cdot n'$ and $a = \gcd(n,a)\cdot a'$ with the relatively prime $n'$ and $a'$. Then 
\[
t(Y) = \gcd\, (a'+bn' ,\; \gcd(n,a)\cdot n' a').
\]
Note that any prime $p$ that divides the product $n' a'$ must divide either $n'$ or $a'$ but not both. In either case, $p$ cannot divide $a' + bn'$ because $a'$ and $b$ are relatively prime. Therefore, all the common divisors of $a'+bn'$ and $\gcd(n,a)\cdot n' a'$ must also be divisors of $\gcd(n,a)$, which implies that $t(Y)$ is a divisor of $\gcd(n,a)$ and hence of $n$.

If $t(Y) = |n|$, the integer $n$ must divide $\gcd(n,a)$ implying that $n = \gcd(n,a)$ and $a = kn$ for some integer $k$. Since $n' = 1$ and $a' = k$, the fact that $n$ divides $a' + bn'$ is equivalent to saying that $n$ divides $k + b$.

Finally, suppose $\div_2(n)\neq \div_2(a)$, then $n'a'$ must be even. If $n'$ is even then $a'$ is odd hence $t(Y) = a' + bn'$ must be odd. On the other hand, if $a'$ is even then both $n'$ and $b$ are odd hence $t(Y) = a' + bn'$ must be odd. 
\end{proof}

\begin{pro}\label{the case t odd}
Theorem \ref{vanishing of constant} holds for all admissible six-tuples with $t$ odd.
\end{pro}

\begin{proof}
Observe that since $t $ is odd, $s$ must be zero by Definition \ref{admissible tuple}. We will  proceed by induction on $t$. The case $t=1$ was proved in Proposition \ref{t=1 case}. Suppose the statement holds for all odd $t = 1,\ldots, 2n-1$ and consider links of the form $K(2n+1,-(2n+1),a/b)$. An easy calculation with formula \eqref{E:order} shows that such a link is ramifiable if and only if $a/b \neq \infty$.

It follows from Lemma \ref{divisibility computation} that $t(Y(2n+1,\bullet,a/b))$ and $t(Y(\bullet,-(2n+1),a/b))$ are divisors of $2n+1$. We claim that they can not be both $2n+1$: otherwise, by Lemma \ref{divisibility computation} again, there would exist an integer $k$ such that $a = (2n+1)k$ and $2n+1$ divides both $b + k$ and $b - k$, which would contradict the assumption that $a$ and $b$ are co-prime. Together with Proposition \ref{switching singular fiber} and the induction hypothesis, this claim implies that 
\[
\chi(K(2n+1,-(2n+1),a/b)) = 0
\]
for any $a/b \neq \infty$. Since $t(Y(2n+1,-(2n+1),\bullet)) = 2n+1$ by Lemma \ref{divisibility computation}, all the constants $C(2n+1,0,(p_0,q_0),(p_1,q_1))$ and $C^{\pm}_j (2n+1,0,(p_0,q_0),(p_1,q_1))$ must vanish by Theorem \ref{skein with error term}. This completes the inductive step and hence the proof of the proposition.
\end{proof}

\begin{pro}\label{1/n singular fiber}
Suppose that $K(n,a_2/b_2,a_3/b_3)$ is ramifiable and $\div_2 (n)\neq \div_2(a_2)$. Then $\chi(K(n,a_2/b_2,a_3/b_3))=0$.
\end{pro}

\begin{proof}
When $\div_2(n)\neq \div_2(b_2)$, the integer $t(K(n,a_2/b_2,\bullet))$ must be odd by Lemma \ref{divisibility computation}. The result now follows  from Proposition \ref{the case t odd} and Proposition \ref{switching singular fiber}.
\end{proof}

The following simple lemma will be instrumental in completing the proof of Theorem \ref{vanishing of constant}.

\begin{lem}\label{framing changing}
For any admissible six-tuple $(t,s,(p_0,q_0),(p_1,q_1))$, 
\begin{gather*}
C(t,s,(p_0,q_0),(p_1,q_1))=C(t,s,(p_0,q_0+kp_0),(p_1,q_1+kp_1))\;\text{and} \\
C^{\pm}_j(t,s,(p_0,q_0),(p_1,q_1))\,=\,C^{\pm}_j(t,s,(p_0,q_0+kp_0),(p_1,q_1+kp_1)),
\end{gather*}
assuming the constants are defined. Here, $k$ can be any integer when $s=0$, and $k$ can be any even integer when $s=1$.
\end{lem}

\begin{proof}
The right hand sides of the equalities in Theorem \ref{skein with error term} do not depend on the choice of framing. Therefore, we can replace the framing $(m,l)$ by $(m-kl,l)$ without changing the corresponding constants (for this to be true, $k$ needs to be even when $s = 1$ so that $m-kl$ is still an $\F_2$ longitude). This proves the statement of the lemma. 
\end{proof}

\begin{pro}\label{the case t even, all ramifiable}
Theorem \ref{vanishing of constant}\,(1) holds for any $t$.
\end{pro}
\begin{proof}
The case of odd $t$ was dealt with in Proposition \ref{the case t odd} hence we will focus on the case of $t = 2n$ with positive $n$. We will consider two separate cases, those of $s = 0$ and $s=1$.

Let us first suppose that $s=1$. It follows from the homological calculation in the proof of Lemma \ref{divisibility computation} that the $\Z$--longitude $l$ of the manifold $Y(2n,-2n,\bullet)$ is $x_3$ with divisibility $2n$, while its $\F_2$ longitude $m$ can be chosen to be $h$. We wish to show that $C(2n,1,(p_0,q_0),(p_1,q_1)) = 0$ for any admissible six-tuple $(2n,1,(p_0,q_0),(p_1,q_1))$ with non-zero $p_0$, $p_1$, and $p_2$ (recall that $q_2 = -q_0 - q_1$ and $p_2 = -p_0 - p_1$). By Lemma \ref{framing changing}, it suffices to show that $C(2n,1,(p_0,q_0+kp_0),(p_1,q_1+kp_1))=0$ for some even integer $k$.

Since $m = h$ and $l = x_3$, the constant $C(2n,1,(p_0,q_0+kp_0),(p_1,q_1+kq_1))$ arises in the admissible skein triangle comprising the links
\[
L_j = K(2n,-2n,(q_j+kp_j)/p_j),\quad j = 0, 1, 2.
\]
Since $q_2$ is even by Definition \ref{admissible tuple}\,(3), the integers $q_0$, $q_1$, and $p_2$ must be odd. Therefore,
\[
\div_2(q_j+kp_j) = 0\,\neq\, \div_2(2n)
\]
for $j = 0, 1$ and any even $k$. Using Proposition \ref{1/n singular fiber}, we conclude that $\chi(L_0) = \chi(L_1) = 0$. To show that $\chi(L_2) = 0$, we just need to find an even integer $k$ such that
\[
\div_2(q_2+kp_2)\,\neq\, \div_2(2n).
\]
This can be done as follows: since $p_2$ and $2^{\div_2(2n)+1}$ are co-prime, there exists an (obviously even) $k$ such that $2^{\div_2(2n)+1}$ divides $q_2 + kp_2$, which implies that $\div_2(q_2+kp_2) > \div_2(2n)$. Now that we know that $\chi(L_j) = 0$ for $j = 0,1,2$, we use Theorem \ref{skein with error term} and Lemma \ref{framing changing} to conclude that $C(2n,1,(p_0,q_0),(p_1,q_1)) = 0$.

Let us now suppose that $s = 0$. Our argument will be similar to that in the $s = 1$ case but with the manifold $Y(4n,4n/(2n-1),\bullet)$. The $\Z$--longitude $l$ of this manifold is $- 2 x_3 + h$ with divisibility $2n$, and it also happens to be its $\F_2$ longitude. We set $m = x_3$. As before, for any admissible six-tuple $(2n,0,(p_0,q_0),(p_1,q_1))$ with non-zero$p_0$, $p_1$, and $p_2$, we want to show that 
\[
C(2n,0,(p_0,q_0+kp_0),(p_1,q_1+kp_1))=0
\]
for some integer $k$. This constant arises in the admissible skein exact triangle with the links
\[
L_j = K(4n,4n/(2n-1),(p_j - 2q_j - 2kp_j)(q_j + kp_j)),\quad j = 0, 1, 2.
\]
Since $(2n,0,(p_0,q_0),(p_1,q_1))$ is an admissible six-tuple, $p_2$ is even and $q_2$, $p_0$, and $p_1$ are odd. Therefore, 
\[
\div_2(p_j-2q_j-2kp_j)=0\neq \div_2(4n)
\]
for $j = 0, 1$ and any $k$. Using Proposition \ref{1/n singular fiber}, we conclude that $\chi(L_0)=\chi(L_1) = 0$. Now, we wish to find an integer $k$ such that
\[
\div_2(p_2-2q_2-2kp_2)\,\neq\, \div_2(4n).
\]
If $p_2 = 0\pmod 4$ this is true for any $k$ because $\div_2(p_2 - 2q_2 - 2kp_2) = 1$. Let us now assume that $p_2 = 4\ell + 2$. Then $p_2 - 2q_2 - 2kp_2 = 4(\ell + (1-q_2)/2 - k(2\ell +1))$. Since $2\ell+1$ is odd, we can choose $k$ so that $2^{\div_2(4n)}$ divides $(\ell + (1-q_2)/2 -k(2\ell + 1))$ and therefore
\[
\div_2(p_2-2q_2-2kp_2) \, >\, \div_2(4n).
\]
In either case, Proposition \ref{1/n singular fiber} implies that $\chi(L_2) = 0$ for a properly chosen $k$. Theorem \ref{skein with error term} now completes the proof.
\end{proof}

\begin{lem}\label{1/n singular fiber general}
Suppose that $n$ is even and the link $K(n,a_2/b_2,a_3/b_3)$ is ramifiable. Then 
\[
\chi(K(n,a_2/b_2,a_3/b_3)) = 0.
\]
\end{lem}

\begin{proof}
If either $a_2$ or $a_3$ is odd, this follows from Proposition \ref{1/n singular fiber}, hence we will focus on the case of even $a_2$ and $a_3$. Since $a_3$ and $b_3$ are co-prime, there exist integers $c_3$ and $d_3$ such that $a_3 d_3 - c_3 b_3 = 1$. By replacing $(c_3, d_3)$ by $(c_3 + ka_3, d_3 + kb_3)$ if necessary and using \eqref{E:order}, we may assume that the links
\[
L = K(n,a_2/b_2,c_3/d_3)\quad\text{and}\quad L' = K(n,a_2/b_2,(a_3+c_3)/(b_3+d_3))
\]
are both ramifiable. Since $c_3$ is odd and $a_3$ is even, we have $\chi(L) = \chi(L') = 0$ by Proposition \ref{1/n singular fiber}.  After a cyclic permutation if necessary, the triple  $(K(n,a_2/b_2,b_3/a_3),L,L')$ forms an admissible skein triangle consisting of three ramifiable links. By Proposition \ref{the case t even, all ramifiable}, 
\[\chi(K(n,a_2/b_2,a_3/b_3)) = 0.\qedhere
\]
\end{proof}

\begin{pro}\label{the case t even, not all ramifiable}
Theorem \ref{vanishing of constant}\,(2) holds for any $t$.
\end{pro}

\begin{proof}
The case $t$ is odd has been dealt with in Proposition \ref{the case t odd} so we will assume that $t = 2m$. By Lemma \ref{1/n singular fiber general}, all ramifiable links of the form $K(2m,-2m,p/q)$ and $K(4m,4m/(2m-1),p/q)$ have vanishing $\chi$. These links cover all possible admissible skein triangles with $t = 2m$. Therefore, we can use Theorem \ref{skein with error term} to conclude that $C^{\pm}_{n}(2m,s,(p_0,q_0),(p_1,q_1))$ equals zero, as long as it is defined.
\end{proof}

Proposition \ref{the case t even, all ramifiable} together with Proposition \ref{the case t even, not all ramifiable} finishes the proof of Theorem \ref{vanishing of constant}.


\section{The Seiberg--Witten and Furuta--Ohta invariants of mapping tori}\label{lfo} 
Let $Y$ be the double branched cover of a knot $K$ in an integral homology sphere $Y'$. The manifold $Y$ is a rational homology sphere, which comes equipped with the covering translation $\tau: Y \to Y$. The mapping torus of $\tau$ is the smooth 4-manifold $X = ([0,1]\times Y)\,/\,(0,x)\sim (1,\tau(x))$ with the product orientation. We will show in Section \ref{S:prelim} that $X$ has the integral homology of $S^1 \times S^3$ and that it has a well defined invariant $\lfo (X)$ of the type introduced by Furuta--Ohta \cite{FO}. The following theorem is the main result of this section.

\begin{thm}\label{T:Lfo for mapping torus}
Let $\lambda (Y')$ be the Casson invariant of $Y'$, and $\sign\,(K)$ the signature of the knot $K$. Then
\[
\lfo (X)\; =\; 2\cdot\lambda(Y')\; + \; \frac 1 8\, \sign\,(K).
\]
\end{thm}

\begin{proof}[Proof of Theorem \ref{T:lfo-lsw}] 
Applying Theorem \ref{T:Lfo for mapping torus} to the homology sphere $Y' = S^3$, we obtain 
\[
\lfo(X) = \frac1{8}\,\sgn\,(K).
\]
On the other hand, using the splitting theorem \cite[Theorem A]{LRS} together with Theorem \ref{T:main} of this paper, we have 
\[
\lsw(X) = -\Lef(\tau_{*})-h(Y,\s)=-\frac 1 8\, \sgn\,(K)
\]
for the unique spin structure on $Y$. This completes the proof.
\end{proof}

Theorem \ref{T:Lfo for mapping torus} was proved in \cite{CS} and \cite{RS2} under the assumption that $Y$ is an integral homology sphere. Our proof here will rely on the extension of those techniques to the general case at hand.


\subsection{Preliminaries}\label{S:prelim}
We begin in this section with some topological preliminaries, including an extension of the Furuta--Ohta invariant $\lfo (X)$ to a wider class of manifolds than that in the original paper \cite{FO}.

The Furuta--Ohta invariant was originally defined in \cite{FO} for smooth 4-manifolds $X$ satisfying two conditions, $H_* (X; \Z) = H_* (S^1 \times S^3; \Z)$ and $H_*(\tilde X; \Z) = H_*(S^3; \Z)$, where $\tilde X$ is the universal abelian cover of $X$.  To fix the signs, one needs to fix an orientation on $X$ as well as a homology orientation, i.e.~a choice of generator of $H^1(X;\Z)$.  The mapping tori we consider in this section provide examples of manifolds $X$ which satisfy the first condition but not the second (which can only be guaranteed if we use rational coefficients). Therefore, we need an extension of the Furuta--Ohta work to define $\lfo (X)$ in this case.

Let $X$ be an arbitrary smooth oriented 4-manifold such that $H_* (X; \Z) = H_* (S^1 \times S^3; \Z)$ and $H_*(\tilde X; \Q) = H_* (S^3; \Q)$. Let $\M^*(X)$ be the moduli space of irreducible ASD connections in a trivial $SU(2)$--bundle $E \to X$. All such connections are necessarily flat hence we can identify $\M^*(X)$ with the irreducible part of the $SU(2)$--character variety of $\pi_1 (X)$.

\begin{lem}
The moduli space $\M^*(X)$ is compact.
\end{lem}

\begin{proof}
The spectral sequence argument of Furuta--Ohta \cite[Section 4.1]{FO} works in our situation with little change to show that the $SU(2)$ character variety of $\pi_1 X$ has the right Zariski dimension at the reducible representations, and hence the set of reducibles is a single isolated component of the character variety, which is obviously compact.
\end{proof}

Given the compactness of $\M^* (X)$, the definition of the Furuta--Ohta invariant proceeds exactly as in \cite{FO} and \cite{RS2} giving a well-defined invariant
\begin{equation}\label{E:lfo}
\lfo (X)\; = \; \frac 1 4\, \#\M^* (X)\, \in\, \Q,
\end{equation}
where $\#\M^* (X)$ stands for the count of points in the (possibly perturbed) moduli space $\M^* (X)$ with the signs determined by a choice of orientation and homology orientation on $X$.
\begin{rmk}\label{R:rational}
The original definition of $\lfo$ in \cite{FO} had a denominator of $1/2$, which was replaced by the $1/4$ in equation~\eqref{E:lfo} in \cite{RS2} to match the conjectured mod $2$ equality with the Rohlin invariant \cite[Conjecture 4.5]{FO}. It is not obvious from the definition that the original $\lfo$ should even be an integer, although this turns out to be true \cite[Section 5]{ruberman-saveliev:survey}. On the other hand, Theorem~\ref{T:Lfo for mapping torus} makes it clear that the generalized $\lfo$ invariant defined herein is not an integer, since the signature of a knot can be an arbitrary even integer. We conjecture that with the normalization used in this paper, $\lfo(X)$ reduces mod $2$ to the Rohlin invariant of $X$, defined as an element of $\Q/2\Z$. This conjecture was confirmed in \cite{RS2} for the mapping tori of finite order diffeomorphisms of integral homology spheres, and now the formula of Theorem \ref{T:Lfo for mapping torus} reduced mod 2 implies that the conjecture is also true for all of the mapping tori $X$ in Theorem \ref{T:Lfo for mapping torus}.
\end{rmk}

\begin{rmk}\label{R:inoue}
A closer examination of the argument in \cite[Section 4.1]{FO} shows that the following hypotheses would allow for a well-defined $\lfo$ invariant: $X$ has the integral homology of $S^1 \times S^3$ and, for every non-trivial $U(1)$ representation $\alpha$, the cohomology $H^1 (X; \mathbb C_\alpha)$ vanishes. Examples of such manifolds $X$ may be obtained by surgery on a knot in $S^4$ whose Alexander polynomial has no roots on the unit circle. For instance, the spin of the figure-eight knot in the $3$-sphere has this property, as do the Cappell--Shaneson knots~\cite{cappell-shaneson:knots}. The latter knots are fibered with fiber $T^3$, and hence it is not difficult to count the irreducible $SU(2)$ representations of $\pi_1 (X)$. For example, one of the Cappell--Shaneson knots gives rise to a 3--torus fibration $X$ over the circle with the monodromy 
\[
\begin{pmatrix}
  0 & 1 & 0 \\
  0 & 1 & 1 \\
  1 & 0 & 0
\end{pmatrix}
\]
and hence has fundamental group with presentation
\[
\pi_1 (X)\; =\; \langle t, x, y, z\;|\; [x,y] = 1, [y,z] = 1, [x,z] = 1, txt^{-1} = y, tyt^{-1} = yz, tzt^{-1} = x
\rangle.
\]
A direct calculation shows that, up to conjugation,  $\pi_1 (X)$ admits a unique irreducible $SU(2)$ representation given by
\[
t = j,\quad x = e^{2\pi i/3},\quad y = z = e^{-2\pi i/3},
\]
and that this representation gives a non-degenerate point in the instanton moduli space on $X$. Therefore, the generalized $\lfo$ invariant of $X$ equals $\pm\, 1/4$. On the other hand, the spin structure on the torus fiber induced from its embedding in $X$ is the group-invariant one~\cite{cappell-shaneson:rp4}. Since the Rohlin invariant of this spin structure equals $1$, the generalized $\lfo(X)$ does not reduce mod $2$ to the Rohlin invariant of $X$.
\end{rmk}

For the rest of Section \ref{lfo}, we will assume that $X$ is the mapping torus of $\tau: Y \to Y$, an involution which exhibits $Y$ as the double branched cover of an integral homology sphere $Y'$ with branch set a knot $K$.

\begin{lem}\label{L:zhss}
The manifold $X$ has the integral homology of $S^1 \times S^3$.
\end{lem}

\begin{proof}
Let $\Delta_K (t)$ be the Alexander polynomial of the knot $K$ normalized so that $\Delta_K (1) = 1$ and $\Delta_K (t^{-1}) = \Delta_K (t)$. Then $H_1 (Y)$ is a finite group of order $|\Delta_K (-1)|$ on which $\tau_*$ acts as minus identity, see Lemma~\ref{conjugation action} or \cite[Theorem 5.5.1]{K}. Since $|\Delta_K (-1)|$ is odd, the fixed point set of $\tau_*: H_1 (Y) \to H_1 (Y)$ must be zero. Now, the natural projection $X \to S^1$ gives rise to a locally trivial bundle with fiber $Y$. The $E^2$ page of its Leray--Serre spectral sequence is
$$
E^2_{pq} = H_p (S^1,\H_q(Y)),
$$
where $\H_q(Y)$ is the local coefficient system associated with the fiber bundle. The groups $E^2_{pq}$ vanish for all $p\ge 2$ hence the spectral sequence collapses at its $E_2$ page. This implies that
$$
H_1(X) = H_1(S^1,\H_0(Y))\,\oplus\,H_0(S^1,\H_1(Y)) = \mathbb Z\,\oplus\,H_0(S^1,\H_1(Y)).
$$
The generator of $\pi_1(S^1)$ acts on $H_1(Y)$ as $\tau_*: H_1 (Y) \to H_1 (Y)$, therefore, $H_0 (S^1, \H_1 (Y)) = \Fix (\tau_*) = 0$ and hence $H_1 (X) = \mathbb Z$. Similarly,
$$
H_2 (X) = H_1(S^1,\H_1(Y))\,\oplus\,H_0(S^1,\H_2(Y)) = 0
$$
because $\Fix (\tau_*) = 0$ and $H_2 (Y) = H^1 (Y) = 0$. This completes the proof.
\end{proof}

Since the $\tilde X = \R \times Y$, where $Y$ is a rational homology sphere, both conditions $H_* (X; \Z) = H_* (S^1 \times S^3; \Z)$ and $H_*(\tilde X; \Z) = H_*(S^3; \Z)$ are satisfied, and the invariant $\lfo (X)$ is well defined by the formula \eqref{E:lfo}. To prove that $\lfo(X)$ is  given by the formula of Theorem \ref{T:Lfo for mapping torus}, we need to analyze the moduli spaces $\M^*(X)$ that go into its definition.


\subsection{Equivariant theory}\label{S:eq}
We will first describe $\M^*(X)$ in terms of $\RR(Y)$, the $SU(2)$ character variety of $\pi_1(Y)$. To this end, consider the splitting
$$
\RR(Y)\; =\; \{\theta\}\, \sqcup\, \RR_{\ab} (Y)\, \sqcup\, \RR_{\irr}(Y),
$$
whose three components consist of the trivial representation and the conjugacy classes of abelian (that is, non-trivial reducible) and irreducible representations, respectively. Note that $\theta$ is the only central representation $\pi_1 (Y) \to SU(2)$ because $Y$ is a $\Z/2$ homology sphere ; see the proof of Lemma \ref{L:zhss}. This decomposition is preserved by the map $\tau^*: \RR(Y) \to \RR (Y)$.

\begin{lem}\label{L:ab}
The involution $\tau^*$ acts as the identity on $\RR_{\ab} (Y)$.
\end{lem}

\begin{proof}
Up to conjugation, any abelian representation $\pi_1 (Y) \to SU(2)$ can be factored through a representation $\alpha: H_1 (Y) \to U(1)$, where $U(1)$ stands for the group of unit complex numbers in $SU(2)$. Since the involution $\tau_*$ acts as minus identity on $H_1 (Y)$, we have $\tau^*\alpha = \alpha^{-1}$, which is obviously a conjugate of $\alpha$. Moreover, any unit quaternion $u$ which conjugates $\alpha^{-1}$ to $\alpha$ must belong to $j\cdot U(1)$ because $\alpha$ is not a central representation.
\end{proof}

Let $\RR^{\tau} (Y)$ be the fixed point set of the involution $\tau^*$ acting on $\RR(Y)\setminus \{\theta\} = \RR_{\ab}(Y)\, \sqcup\, \RR_{\irr} (Y)$. It follows from the above lemma that
\[
\RR^{\tau} (Y)\; =\; \RR_{\ab} (Y)\, \sqcup\, \RR_{\irr}^{\tau} (Y).
\]



\begin{pro}\label{P:one}
Let $i: Y \to X$ be the inclusion map given by the formula $i(x) = [0,x]$. Then the induced map
\begin{equation}\label{E:i*}
i^*: \M^*(X) \to \RR^{\tau}(Y)
\end{equation}
is well defined, and is a one-to-one correspondence over $\RR_{\ab} (Y)$ and a two-to-one correspondence over $\RR^{\tau}_{\irr} (Y)$.
\end{pro}

\begin{proof}
The natural projection $X\to S^1$ is a locally trivial bundle whose
homotopy exact sequence
\[
\begin{CD}
0 @>>> \pi_1(Y) @>>> \pi_1(X) @>>> \mathbb Z @>>> 0
\end{CD}
\]
splits, making $\pi_1(X)$ into a semi-direct product of $\pi_1(Y)$ and $\mathbb Z$. Let $t$ be a generator of $\mathbb Z$ then every representation $A: \pi_1(X)\to SU(2)$ determines and is uniquely determined by the pair $(\alpha,u)$ where $u = A(t)$ and $\alpha= i^* A: \pi_1(Y)
\to SU(2)$ is a representation such that $\tau^*\alpha = u\alpha u^{-1}$. In particular, the conjugacy class of $\alpha$ is fixed by $\tau^*$.

If $\alpha = \theta$ then $A$ must be reducible, hence $\alpha$ is not in the image of $i^*$. If $\alpha$ is non-trivial abelian, we can conjugate it to a representation whose image is in the group of unit complex numbers in $SU(2)$. Then $\alpha$ is of the form $\alpha = i^*A$ with $u = A(t)$ in the circle $j\cdot U(1)$, as in the proof of Lemma \ref{L:ab}. In particular, $A$ is irreducible and $u^2 = -1$. Since any two quaternions in $j\cdot U(1)$ are conjugate to each other by a unit complex number, the map $i^*$ is a one-to-one correspondence over $\RR^a (Y)$. Finally, let $\alpha$ be an irreducible representation with the character in $\RR^{\tau} (Y)$. Then there is a unit quaternion $u$ such that $\tau^*\alpha = u\alpha u^{-1}$, and therefore $\alpha$ is in the image of $i^*$. Moreover, there are exactly two different choices of $u$ such that $\tau^*\alpha= u\alpha u^{-1}$ because if $u_1\alpha u_1^{-1} = u_2\alpha u_2^{-1}$ then $u_1 = \pm\, u_2$ since $\alpha$ is irreducible. The irreducibility of $\alpha$ also implies that $u^2 = \pm 1$. In this case, the map $i^*$ is a two-to-one correspondence.
\end{proof}

\begin{rmk}
It follows from the above proof that the characters in $\M^*(X)$ that are mapped by $i^*$ to $\RR_{\ab} (Y)$ are binary dihedral, while those mapped to $\RR^{\tau}_{\irr}(Y)$ are not.
\end{rmk}

The Zariski tangent space to $\RR^{\tau}(Y)$ at a point $[\alpha] \in \RR^{\tau}(Y)$ is the fixed point set of the map $\tau^*: T_{[\alpha]} \RR (Y) \to T_{[\alpha]} \RR (Y)$. Using an identification $T_{[\alpha]} \RR (Y) = H^1 (Y, \ad\alpha)$ and the fact that $\tau^*\alpha = u\alpha u^{-1}$, this set can be described in cohomological terms as the fixed point set of the map
\[
\Ad u\circ \tau^*: H^1 (Y, \ad\alpha) \to H^1 (Y, \ad\alpha).
\]
We will call $\RR^{\tau}(Y)$ {\it non-degenerate} if the equivariant cohomology groups 
\[
H^1_{\tau}(Y, \ad\alpha) = \Fix\,(\Ad u\circ \tau^*: H^1 (Y, \ad\alpha) \to H^1 (Y, \ad\alpha))
\]
vanish for all $[\alpha]\in\RR^{\tau}(Y)$. The moduli space $\M^*(X)$ is called {\it non-degenerate} if $\coker(d_A^*\,\oplus\, d_A^+) = 0$ for all $[A]\in \M^*(X)$. Since $\ind(d^*\,\oplus\, d_A^+) = \dim\M^*(X) = 0$, this is equivalent to $\ker(d_A^*\,\oplus\, d_A^+) = 0$ and, since $A$ is flat and irreducible, to simply $H^1(X;\ad A) = 0$.

\begin{pro}\label{P2}
The moduli space $\M^*(X)$ is non-degenerate if and only if $\RR^{\tau}(Y)$ is non-degenerate.
\end{pro}

\begin{proof}
The group $H^1 (X, \ad A)$ can be computed with the help of the Leray--Serre spectral sequence of the fibration $X\to S^1$ with fiber $Y$. The $E_2$--page of this spectral sequence is
\[
E_2^{pq} = H^p(S^1, \H^q(Y, \ad\alpha)),
\]
where $\alpha = i^* A$ and $\H^q(Y,\ad\alpha)$ is the local coefficient system associated with the fibration. The groups $E_2^{pq}$ vanish for all $p\ge 2$, so the spectral collapses at the $E_2$--page, and
\begin{equation}\label{E:E2}
H^1(X,\ad A) = H^1(S^1,\H^0(Y,\ad\alpha))\,\oplus\,H^0(S^1,\H^1(Y,
\ad\alpha)).
\end{equation}
The generator of $\pi_1(S^1)$ acts on the cohomology groups $H^* (Y,\ad\alpha)$ as
\[
\Ad u\circ \tau^*: H^* (Y,\ad\alpha) \to H^* (Y,\ad\alpha),
\]
where $u$ is such that $\tau^*\alpha = u \alpha u^{-1}$. If $\alpha$ is irreducible, $H^0 (Y,\ad\alpha) = 0$ and the first summand in \eqref{E:E2} vanishes. If $\alpha$ is non-trivial abelian, we may assume without loss of generality that it takes values in the group $U(1)$ of unit complex numbers. Then $\tau^*\alpha = u\alpha u^{-1}$ for some $u \in j\cdot U(1)$ and $H^0 (Y,\ad\alpha) = i\cdot \mathbb R$ as a subspace of $\su(2)$, with $\tau^* = \id$. One can easily check that $\Ad u$ acts as minus identity on $i\cdot\mathbb R$ hence the first summand in \eqref{E:E2} again vanishes. The second summand in \eqref{E:E2} is the fixed point set of $\tau^*$ acting on $H^1(Y,\ad\alpha)$, which is the equivariant cohomology $H^1_{\tau}(Y,\ad\alpha)$. Thus we conclude that $H^1(X,\ad A) = H^1_{\tau}(Y,\ad\alpha)$, which completes the proof.
\end{proof}

Let us assume that $\RR^{\tau}(Y)$ is non-degenerate. For any $[\alpha] \in \RR^{\tau}(Y)$, its orientation will be given by
\[
(-1)^{\sf^{\tau}(\theta,\alpha)}
\]
where $\sf^{\tau} (\theta,\alpha)$ is the mod 2 equivariant spectral flow defined in \cite[Section 3.4]{RS2} for irreducible $\alpha$. That definition extends word for word to abelian $\alpha$ after one resolves the technical issue of the existence of a constant lift, which we will do next.

Let $P$ be an $SU(2)$ bundle over $Y$ with a fixed trivialization and $\alpha$ an abelian flat connection in $P$; we are abusing notations by using the same symbol for the connection and its holonomy.  It follows from Lemma \ref{L:ab} that $\tau$ admits a lift $\tilde \tau: P \to P$ such that $\tilde\tau^*\alpha = \alpha$. Since $\alpha$ is abelian, this lift is defined uniquely up to the stabilizer of $\alpha$, which is a copy of $U(1)$ in $SU(2)$. The lift $\tilde\tau$ can be written in the base-fiber coordinates as $\tilde \tau(x,y) = (\tau(x), \rho(x)\cdot y)$ for some function $\rho: Y \to SU(2)$. We call it \emph{constant} if there exists $u\in SU(2)$ such that $\rho(x) = u$ for all $x\in SU(2)$.

\begin{lem}\label{L:lift}
By changing $\alpha$ within its gauge equivalence class, one may assume that $\tilde\tau$ is a constant lift with $u^2 = -1$.
\end{lem}

\begin{proof}
The equation $\tilde\tau^*\alpha = \alpha$ implies that $(\tilde\tau^2)^*\alpha = \alpha$ hence the gauge transformation $\tilde\tau^2$ belongs to the stabilizer of the connection $\alpha$. If $x \in \Fix (\tau)$ then $\tilde\tau^2 (x,y) = (x,\rho(x)^2\cdot y)$ hence $\rho(x)^2$ is a unit complex number independent of $x$. This implies that $\rho(x)$ itself is a unit complex number unless $\rho(x)^2 = -1$. It is this last case that must be realized because, at the level of holonomy representations, $\tau^*\alpha = \alpha^{-1}$ is conjugate to $\alpha$ by an element $u \in SU(2)$ with $u^2 = -1$; see the proof of Lemma \ref{L:ab}. Since $\rho(x)^2 = -1$ describes a single conjugacy class $\tr \rho(x) = 0$ in $SU(2)$, we may assume that $\rho(x) = u$ for all $x \in \Fix(\tau)$.

To finish the proof, we will follow the argument of \cite[Section 2.2]{RS2}. Let $u: P \to P$ be the constant lift $u(x,y) = (\tau(x),u\cdot y)$ and consider the $SO(3)$ orbifold bundles $P/\tilde\tau$ and $P/u$ over the integral homology sphere $Y'$. All such bundles are classified by the holonomy around the singular set in $Y'$. Since this holonomy equals $\ad(u)$ in both cases, the bundles $P/\tilde\tau$ and $P/u$ must be isomorphic, with any isomorphism pulling back to a gauge transformation $g: P \to P$ that relates the lifts $\tilde\tau$ and $u$.
\end{proof}

\begin{pro}
Assuming that the moduli space $\RR^{\tau}(Y)$ is non-degenerate, the map \eqref{E:i*} is orientation preserving.
\end{pro}

\begin{proof}
The proof from \cite[Section 3]{RS2} extends to the current situation with no change.
\end{proof}

\subsection{Orbifold theory}\label{S:orb}
Under the continued non-degeneracy assumption, we will now describe $\RR^{\tau}(Y)$ in terms of orbifold representations. Let us consider the orbifold fundamental group $\pi_1^V (Y',K) = \pi_1 (Y' - N(K))/ \langle \mu^2 \rangle$, where $\mu$ is a meridian of $K$. This group can be included into the split orbifold exact sequence
\[
\begin{tikzpicture}
\draw (1,1) node (a) {$1$};
\draw (3,1) node (b) {$\pi_1 Y$};
\draw (6,1) node (c) {$\pi_1^V(Y', K)$};
\draw (9,1) node (d) {$\Z/2$};
\draw (11,1) node (e) {$1.$};
\draw[->](a)--(b);
\draw[->](b)--(c) node [midway,above](TextNode){$\pi_{*}$};
\draw[->](c)--(d) node [midway,above](TextNode){$j$};
\draw[->](d)--(e);
\end{tikzpicture}
\]
Denote by $\RR^V (Y',K; SO(3))$ the character variety of irreducible $SO(3)$ representations of the group $\pi_1^V (Y', K)$, and also introduce the character variety $\RR^{\tau} (Y; SO(3))$ of irreducible representations $\pi_1 Y \to SO(3)$.

\begin{pro}\label{P:eq-knots}
The pull back of representations via the map $\pi_*$ in the orbifold exact sequence gives rise to a one-to-one correspondence
\[
\pi^*: \RR^V (Y',K; SO(3)) \longrightarrow \RR^{\tau} (Y;SO(3)).
\]
\end{pro}

\begin{proof}
One can easily see that a representation $\alpha': \pi_1^V(Y', K) \to SO(3)$ pulls back to a trivial representation $\theta: \pi_1 Y \to SO(3)$ if and only if $\alpha'$ is reducible. The same argument as in \cite[Proposition 3.3]{CS} shows that all pull-back representations belong to $\RR^{\tau} (Y,SO(3))$. The inverse map for $\pi^*$ is constructed as follows: given $[\alpha] \in \RR^{\tau} (Y,SO(3))$ choose $v \in SO(3)$ such that $\tau^*\alpha = v \alpha v^{-1}$, and define a representation $\alpha'$ of $\pi_1^V(Y',K) = \pi_1 Y \rtimes \Z/2$ by the formula
\begin{equation}\label{E:inverse}
\alpha'(g \cdot \mu^k)\, =\, \alpha(g)\cdot v^k.
\end{equation}
If $\alpha$ is irreducible, the element $v$ is unique hence formula \eqref{E:inverse} gives an inverse map. If $\alpha$ is non-trivial abelian, lift it to a $U(1)$ representation using the fact that $Y$ is a $\Z/2$ homology sphere. The proof of Lemma \ref{L:ab} then tells us that $v = \Ad u$ for some $u \in j\cdot U(1)$. Since any two elements of $j\cdot U(1)$ are conjugate to each other by a unit complex number, formula \eqref{E:inverse} again gives an inverse map.
\end{proof}

The representations $\pi_1^V (Y',K) \to SO(3)$ need not lift to $SU(2)$ representations. However, they lift to projective representations $\pi_1^V (Y',K) \to SU(2)$ sending $\mu^2$ to $\pm 1$. The character variety of such projective representations will be denoted by $\RR^V (Y',K)$, and it will be oriented using the orbifold spectral flow.

\begin{pro}
The correspondence of Proposition \ref{P:eq-knots} gives rise to an orientation preserving correspondence $\RR^V (Y',K) \to \RR^{\tau} (Y)$ which is one-to-one over $\RR_{\ab}(Y)$ and two-to-one over $\RR^{\tau}_{\irr} (Y)$.
\end{pro}

\begin{proof}
Let us consider the adjoint representation $\Ad: SU(2) \to SO(3)$ and the induced maps
\[
 \RR^{\tau}(Y) \to \RR^{\tau}(Y; SO(3))\quad\text{and}\quad\RR^V (Y',K) \to \RR^V (Y',K; SO(3)).
\]
The first map is a one-to-one correspondence because $Y$ is a $\Z/2$ homology sphere. The second map is the quotient map by the action of $\Z/2$ sending the image of the meridian $\mu$ to its negative. The fixed points of this action are precisely the binary dihedral projective representations $\alpha': \pi_1^V (Y',K) \to SU(2)$. Now, the proof will be finished as soon as we show that an irreducible projective representation $\alpha': \pi_1^V (Y',K) \to SU(2)$ is binary dihedral if and only if its pull back representation $\pi^*\alpha': \pi_1 (Y) \to SU(2)$ is abelian.

If $\pi^*\alpha'$ is abelian, its image belongs to $U(1) \subset SU(2)$ and the image of $\alpha'$ to its $\Z/2$ extension. This extension is the binary dihedral group $U(1)\,\cup\,j\cdot U(1)$. Conversely, it follows from the orbifold exact sequence that $\pi_1 Y$ is the commutator subgroup of $\pi_1^V (Y',K)$ therefore, if $\alpha'$ is binary dihedral, the image of $\pi^*\alpha'$ must belong to the commutator subgroup of $U(1)\,\cup\,j\cdot U(1)$, which is of course the group $U(1)$.

Since the orbifold spectral flow matches the equivariant spectral flow used to orient $\RR^{\tau}(Y)$, the above correspondence is orientation preserving.
\end{proof}


\subsection{Perturbations}\label{S:pert}
In this section, we will remove the assumption that $\RR^{\tau}(Y)$ is non-degene\-rate which we used until now. To accomplish that, we will switch from the language of representations to the language of connections. Let $P$ a trivialized $SU(2)$ bundle over $Y$. Any endomorphism $\tilde\tau: P \to P$ which lifts the involution $\tau$ induces an action on the space of connections $\A(Y)$ by pull back. Since any two such lifts are related by a gauge transformation, this action defines a well defined action on the configuration space $\B(Y) = \A(Y)/\G(Y)$. The fixed point set of this action will be denoted by $\B^{\tau}(Y)$.

The irreducible part of $\B^{\tau}(Y)$ was studied in \cite{RS2} hence we will only deal with reducible connections. In fact, we will further restrict ourselves to constant lifts $u$ with $u^2 =-1$ because any flat abelian connection $\alpha$ admits such a lift; see Lemma \ref{L:lift}.

Let $\A^u (Y) \subset \A(Y)$ consist of all non-trivial connections $A$ such that $u^*A = A$, and $\G^u(Y) \subset \G(Y)$ of all gauge transformations $g$ such that $gu = ug$. The quotient space $\A^u(Y)/\G^u(Y)$ will be denoted by $\B^u(Y)$. The following lemma is a key to making the arguments of \cite{RS2} work in the case of abelian connections.

\begin{lem}
The group $\G^u(Y)$ acts on $\A^u(Y)$ with the stabilizer $\{\pm 1\}$. Moreover, the natural map $\B^u(Y) \to \B^{\tau}(Y)$ is a two-to-one correspondence to its image on the irreducible part of $\B^u(Y)$, and a one-to-one correspondence on the reducible part.
\end{lem}

\begin{proof}
For the sake of simplicity, we will assume that reducible connections have their holonomy in the subgroup $U(1)$ of unit complex numbers in $SU(2)$, and that $u \in j\cdot U(1)$. Let us suppose that $g^*A = A$ for a connection $A \in \A^u (Y)$ and a gauge transformation $g \in \G^u (Y)$. If $A$ is irreducible, we automatically have $g = \pm 1$. If $A$ is non-trivial abelian, then $g$ is a complex number, and the condition $ug = gu$ implies that $g = \pm 1$.

To prove the second statement, consider a connection $A$ such that $u^*A = A$ and consider its gauge equivalence class in $\B^{\tau}(Y)$. It consists of all connections $g^*A$ such that $u^* g^* A = g^* A$. Since $A = u^* A$, we immediately conclude that $u^* g^* A = g^* u^* A$ so that $ug$ and $gu$ differ by an element in the stabilizer of $A$. If $A$ is irreducible, its stabilizer consists of $\pm 1$ hence $ug = \pm gu$. The group of gauge transformations satisfying this condition contains $\G^u(Y)$ as a subgroup of index two, which leads to the desired two-to-one correspondence. If $A$ is non-trivial abelian, its stabilizer consists of unit complex numbers. Therefore, we can write $ug = c^2 gu$ with $c \in U(1)$ or, equivalently, $ucg = cgu$. This provides us with a gauge transformation $cg \in \G^u (Y)$ such that $(cg)^* A = A$, yielding the one-to-one correspondence on the reducible part.
\end{proof}

With this lemma in place, the proof of Proposition \ref{P:one} can be re-stated in gauge-theoretic terms as in \cite[Proposition 3.1]{RS2}. The treatment of perturbations in our case is then essentially identical to that in \cite{CS} and \cite{RS2}, one important observation being that the orbifold representations $\alpha'$ that pull back to abelian representations of $\pi_1 (Y)$ are in fact irreducible. This fact is used in the proof of \cite[Lemma 3.8]{CS}, which supplies us with sufficiently many admissible perturbations.


\subsection{Proof of Theorem \ref{T:Lfo for mapping torus}}
The outcome of Section \ref{S:eq} and Section \ref{S:orb} is that, perhaps after perturbing as in Section \ref{S:pert}, we have two orientation preserving correspondences,
\[
\M^*(X)\; \longrightarrow\; \RR^{\tau}(Y)\; \longleftarrow\; \RR^V (Y',K),
\]
both of which are one-to-one over $\RR_{\ab} (Y)$ and two-to-one over $\RR^{\tau}_{\irr} (Y)$ (we omit perturbations in our notations). These correspondences give rise to an orientation preserving one-to-one correspondence between $\M^*(X)$ and $\RR^V (Y',K)$. The proof of Theorem \ref{T:main} will be complete after we express the signed count of points in $\RR^V (Y',K)$ in terms of the Casson invariant of $Y'$ and the knot signature of $K$.

The character variety $\RR^V (Y',K)$ of projective representations $\alpha'$ splits into two components corresponding to whether the square of $\alpha'(\mu)$ equals $+1$ or $-1$. Let $E$ be the exterior of the knot $K$ then this splitting corresponds to the splitting
\[
\RR^V (Y',K)\; = \; \SS_0 (E,SU(2))\,\cup\,\SS_{1/2} (E,SU(2))\,\cup\,\SS_1 (E,SU(2))
\]
 of \cite[Proposition 3.4]{CS}, where $\SS_a (E,SU(2))$ comprises the conjugacy classes of representations $\gamma: \pi_1 X \to SU(2)$ such that $\tr \gamma (\mu) = 2\cos (2\pi a)$. According to Herald \cite{H}, the signed count of points in $\SS_0 (E,SU(2))\,\cup\,\SS_1 (E,SU(2))$ equals $4 \cdot \lambda(Y')$, while the signed count of points in $\SS_{1/2} (E,SU(2))$ equals $4 \cdot \lambda (Y') + 1/2\,\sgn\,(K)$. Adding up the two counts and dividing by four we obtain the desired formula
\[
\lfo(X)\; =\; 2\cdot\lambda (Y')\; +\; \frac 1 8\, \sgn\,(K).
\]


\section{Strongly non-extendable involutions and Akbulut corks}\label{S:cork}
A cork is a pair $(W,\tau)$ which consists of a smooth compact contractible $4$--manifold $W$ and an involution $\tau$ on its boundary that  does not extend to a self-diffeomorphism of $W$. Sometimes the definition of a cork includes the hypothesis that $W$ have a Stein structure (see for instance~\cite[Definition 10.3]{akbulut:book}) but we do not require this. 


\subsection{Strongly non-extendable involutions}\label{S:inv}
Figure~\ref{F:cork-symm}\,(a) shows the cork constructed by Akbulut \cite{Akbulut:cork}, and Figure \ref{F:cork-symm}\,(b) shows the involution $\tau$ on its boundary. This cork will be called $W_1$, and its boundary $Y_1$. 

\begin{figure}[!ht]
    \labellist
    \normalsize\hair 0mm
    \pinlabel {(a)} at  58 -13
    \pinlabel {(b)} at  275 -13
    \pinlabel {$0$} at 12 125
    \pinlabel {$0$} at 230 140
    \pinlabel {$0$} at 315 140
    \pinlabel {$\tau$} at 285 155
    \pinlabel {{\LARGE $\cong_{\footnotesize\,\partial}$}} at 175 72
    \endlabellist
\includegraphics[scale=1.1]{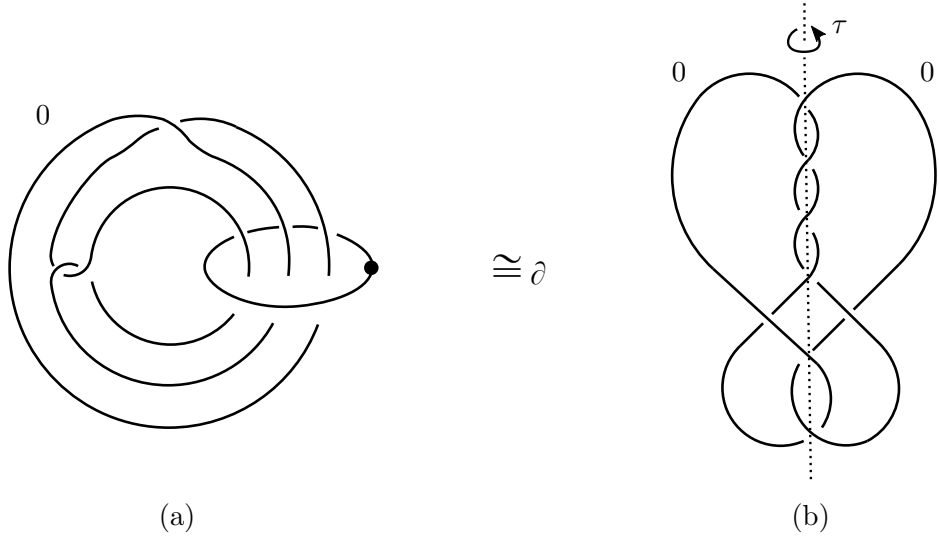}
\bigskip
\caption{Akbulut cork $W_1$ and the involution on  $Y_1 = \partial W_1$}
\label{F:cork-symm}
\end{figure}
\begin{thm}\label{T:strong}  
The involution $\tau: Y_1 \to Y_1$ does not extend to a diffeomorphism of any smooth $\mathbb{Z}/2$ homology 4-ball bounded by $Y_1$.
\end{thm}

The proof of Theorem~\ref{T:strong} makes use of a gluing theorem for Seiberg--Witten invariants, which we briefly summarize. Let $(X,\s)$ be a smooth closed oriented 4-manifold with a $\spinc$ structure $\s$ and $b^+_2 (X) > 1$. Suppose that $X$ is decomposed as $X = X_1\, \cup\, X_2$ with $b^+_2 (X_2) > 1$. Let $Y$ be the oriented boundary of $X_1$ and consider two cobordisms, $M_1$ from $S^3$ to $Y$ and $M_2$ from $Y$ to $S^3$, obtained by removing open 4-balls from $X_1$ and $X_2$, respectively. Let $\s_i$ be the induced $\spinc$ structures on $M_i$, $i = 1, 2$, and $\s_0$ the induced $\spinc$ structure on $Y$. Then we have two maps in monopole homology,
\begin{align*}
& \widehat{HM}_{*}(M_1,\s_1): \widehat{HM}_{*}(S^{3})\rightarrow \widehat{HM}_{*}(Y,\s_0)
\quad\text{and} \\
& \overrightarrow{HM}^{*}(M_2,\s_2): \widecheck{HM}^{*}(S^{3})\rightarrow \widehat{HM}^{*}(Y,\s_0).
\end{align*}
Denote by $\check1$ and $\hat1$ the canonical generators of $\widecheck{HM}^{*}(S^{3})$ and  $\widehat{HM}_{*}(S^{3})$. The gluing theorem expresses the Seiberg--Witten invariant of $(X,\s)$ as follows.

\begin{pro}\label{gluing theorem}
Suppose that $Y$ is a rational homology sphere. Then
\begin{equation}\label{E:swgluing}
SW(X,\s) = \langle\, \widehat{HM}_{*}(M_1,\s_1)(\hat1), \overrightarrow{HM}^{*}(M_2,\s_2)(\check1)\,\rangle.
\end{equation}
\end{pro}
Formula \eqref{E:swgluing} is a slight strengthening of the formula that appears just before \cite[Definition 3.6.3]{Kronheimer-Mrowka}, in that \eqref{E:swgluing} holds for each $\spinc$ structure separately, rather than for the sum over the $\spinc$ structures on $X$, as would be the case for $b_1(Y) > 0$. Our strengthened formula follows from the remark on~\cite[page 569]{Kronheimer-Mrowka} following the proof of Proposition 27.4.1. (Separating the $\spinc$ structures can also be achieved using local coefficients as in~\cite[Section 3.7--3.8]{Kronheimer-Mrowka}, but we do not need this in our situation.)  

The following simple algebraic lemma is presumably well-known.

\begin{lem}\label{L:trace}
Let $A$ be a $2 \times 2$ matrix with $A^2 = I$ and $\tr(A) = -2$. Then $A = -I$, where $I$ stands for the identity matrix.

\end{lem}
\begin{proof}
By the Cayley--Hamilton theorem, we have that $A^2 - \tr(A)\cdot A + \det(A)\cdot I = 0$, where $\det (A) = \pm 1$. If $\det(A) = -1$, we obtain $\tr(A)\cdot A = 0$, which contradicts the invertibility of $A$. Hence $\det(A)  = 1$, which implies that $A = -I$.  
\end{proof}

\begin{proof}[Proof of Theorem~\ref{T:strong}]
We will omit the $\spinc$ structure $\s_0$ from our notations. We claim first that the action of $\tau_{*}$ on $\HM^{\red}(Y_1)$ is minus the identity. To prove this, we will combine our Theorem~\ref{T:main} with a Heegaard Floer homology calculation by Akbulut and Durusoy~\cite{akbulut-durusoy:involution}. They work with a picture that is the mirror image of Figure \ref{F:cork-symm}\,(a) and show that $\HF^+(-Y_1) \cong \T_{(0)}\, \oplus\, \Z_{(0)}\, \oplus\, \Z_{(0)}$, where the first summand is a tower $\Z[U,U^{-1}]/U \cdot \Z[U]$ with the lowest degree in grading $0$. It follows that $\HF^{\red}(-Y_1) \cong \Z_{(0)}\,\oplus\,\Z_{(0)}$ and $\HF^{\red}(Y_1) \cong  \Z\,\oplus\,\Z$, with both summands of odd grading (with respect to the absolute $\Z/2$ grading). The parity can be checked using the formula
\[
\lambda(Y) = \chi(\HF^{\red}(Y)) - 1/2\cdot d(Y)
\]

\smallskip\noindent
of \cite[Theorem 1.3]{oz:boundary}, where $\lambda(Y)$ is the Casson invariant of $Y$. Since $\lambda(Y_1) = -2$, see for instance \cite{saveliev:cork}, and $d(Y_1) = 0$, both summands in  $\HF^{\red}(Y_1)$ must have odd grading.  We translate this computation  into the monopole homology, keeping in mind the isomorphisms
\begin{equation}\label{E:dual}
\widehat{\HM}_{a}(Y)\; \cong\; (\widecheck{\HM}_{-1-a}(-Y))^{*}\; \cong\; (\HF^+_{-1-a}(-Y))^{*}.
\end{equation}
The grading shift  for the first `duality' isomorphism is \cite[Proposition 28.3.4]{Kronheimer-Mrowka}, while the second equality of the absolute $\Q$-gradings is deduced from \cite{Gripp,huang-ramos:grading,Gardiner}.

Now, the involution $\tau$ makes $Y_1$ into a double branched cover of the 3-sphere with branch set a knot $K_1\subset S^{3}$. As described in~\cite{ruberman-saveliev:survey} and drawn in Figure \ref{F:twist}, the knot $K_1$ is obtained from the left-handed $(5,6)$-torus knot on six strings by adding one full left-handed twist on two adjacent strings. In particular, the signature of $K_1$ is $16$.  Using Theorem~\ref{T:main}, we compute
\[
\tr(\tau_*) = - \Lef(\tau_*) = -\frac18\sign(K_1) = -2
\]
and, using Lemma~\ref{L:trace}, conclude that 
\begin{equation}\label{E:tau}
\tau_{*} = -I: \HF_{\red}(Y_1) \to \HF_{\red}(Y_1).
\end{equation}

\begin{figure}[!ht]
\centering
\includegraphics[scale=.8]{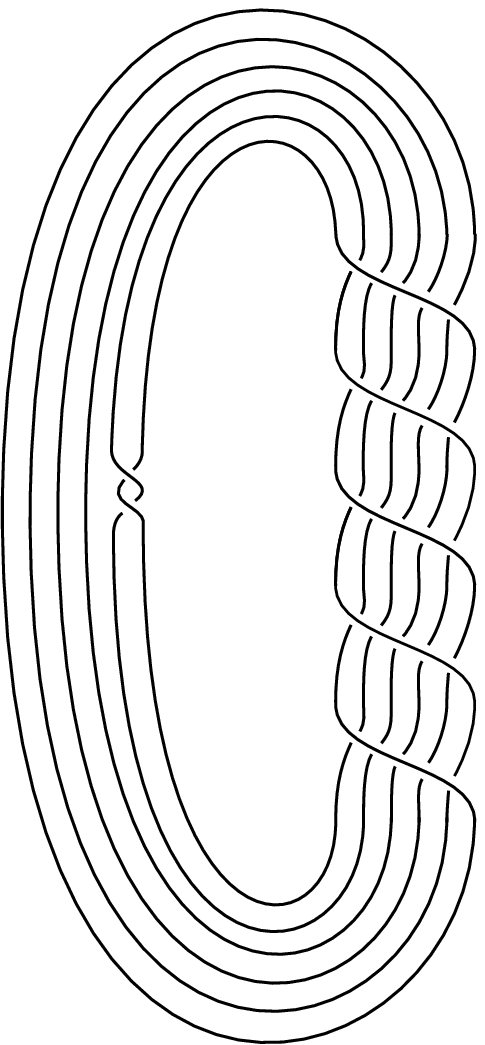}
\caption{}\label{F:twist}
\end{figure}

In order to compute the action of $\tau_{*}$ on $\widehat{\HM}(Y_1)$, consider the short exact sequence in monopole homology
\[
\begin{CD}
 0 @>>> \HM^{\red}(Y_1) @>>> \widehat{\HM}_{-1}(Y_1) @> f >> \overline{\HM}_{-2}(Y_1) @>>> 0.
\end{CD}
\]

\medskip\noindent
Since $\HM^{\red}(Y_1)\cong \Z\,\oplus\,\Z$ and $\overline{\HM}_{-2}(Y_1)\cong \Z$, the group $\widehat{\HM}_{-1}(Y_1)$ must be free of rank 3. We define a splitting $\overline{\HM}_{-2}(Y_1) \rightarrow \widehat{\HM}_{-1}(Y_1)$ of this short exact sequence by sending the canonical generator $1\in \overline{\HM}_{-2}(Y_1)$ to the element $e_W = \widehat{HM}_{*}(M_1,\s_1)(\hat1)\in \widehat{\HM}_{-1}(Y_1)$ as above, where $M_1$ is obtained from $W_1$ by removing an open 4-ball. Using the fact that $\overline{HM}_{*}(M_1,\s_1)$ maps the canonical generator of $\overline{\HM}_{-2}(S^{3})$ to that of $\overline{\HM}_{-2}(Y_1)$, we see that $f(e_{W})=1$.  

For any choice of free generators $\{e_0,e_1\}$ of $\HM^{\red}(Y_1)$ we have a set of free generators $\{e_0,e_1,e_{W}\}$ of $\widehat{\HM}_{-1}(Y_1)$. The action of $\tau_{*}$ on $\widehat{\HM}_{-1}(Y_1)$ is then given by a matrix of the form
\medskip
\begin{equation}\label{E:matrix}
\begin{pmatrix*}[r] -1 & 0 & p \\ 0 & -1 & q \\ 0 & 0 & 1 \end{pmatrix*}
\end{equation}

\medskip\noindent
with some unknown integers $p$ and $q$. In what follows, we will extract some information about $p$ and $q $ from the fact that a cork twist on $W_1$ changes the Seiberg--Witten invariant of a certain closed 4-manifold.

There is an embedding~\cite[Figure 9.5]{gompf-stipsicz:book} (see also~\cite{Akbulut:cork}) of $W_1$ into the blown up $K3$--surface, $X = K3\, \#\, \cptwobar$, such that the the cork twist results in the manifold
\[
X^\tau = W_1\,\cup_{\tau}\,(X - \inte(W_1))
\]
with the trivial Seiberg--Witten invariant. On the other hand, the blowup formula for Seiberg--Witten invariants~\cite{fs:sw-blowup} implies that the Seiberg--Witten invariant of $X$ equals 1 for the $\spinc$ structure $\s$ whose first Chern class is the generator of $H^2 (\cptwobar)$. Since $Y_1$ is an integral homology sphere, there is an obvious correspondence, $\s \leftrightarrow \s^\tau$, between $\spinc$ structures on $X$ and $X^\tau$. Using the gluing formula \eqref{E:swgluing} with $X_1 = W_1$ and $X_2 = X - \inte(W_1)$, we obtain
\begin{align*}
& \operatorname{SW}(X,\s) = \langle \widehat{HM}_{*}(M_1,\s_1)(\hat1), \overrightarrow{HM}^{*}(M_2,\s_2)(\check1)\rangle = 1\quad\;\text{and} \\
& \operatorname{SW}(X^{\tau},\s^{\tau}) = \langle \tau_*(\widehat{HM}_{*}(M_1,\s_1)(\hat1)), \overrightarrow{HM}^{*}(M_2,\s_2)(\check1)\rangle = 0.
\end{align*}
If we write $\overrightarrow{HM}^{*}(M_2,\s_2)(\check1) = ae_0^* + b e_1^* + c e_W^*$ with respect to the dual basis of $\widehat{HM}^{-1}(Y_1)$, the above formulas reduce to
\[
\operatorname{SW}(X,\s) = c = 1\quad\text{and}\quad \operatorname{SW}(X^\tau,\s^\tau) = ap + bq + c = 0,
\]
implying that $ap + bq + 1 = 0$ and, in particular, that the integers $p$ and $q$ are co-prime. Therefore, by a change of basis $\{e_0,e_1\}$, we can turn the matrix \eqref{E:matrix} of the involution $\tau_*$ into
\medskip
\[
A=\begin{pmatrix*}[r] -1 & 0 & 0 \\ 0 & -1 & 1 \\ 0 & 0 & 1 \end{pmatrix*}.
\]

\medskip
Now, suppose that $Y_1$ bounds another smooth $\Z/2$-homology 4-ball $W'$. $W'$ has a unique spin structure $\s_{W'}$, which must be preserved by any diffeomorphism. By studying the the spin manifold $(W',\mathfrak{s}_{W'})$, one  defines the element $e_{W'} \in \widehat{\HM}_{-1}(Y_1)$ by the same procedure as $e_{W}$. As before, $f(e_{W'})=1$. Suppose that $\tau$ extends to a diffeomorphism on $W'$. Then, by naturality of monopole Floer homology, one must have
\[
\tau^{*}(e_{W'})=e_{W'}.
\]
But since the kernel of $A-I$ is generated by the vector $(0,1,2)$, we have $e_{W'} = (0,c,2c)$ for some integer $c$. In particular, $f(e_{W'}) = 2c$ is an even integer, which contradicts $f(e_{W'}) = 1$.
\end{proof}

\begin{rmk}\label{R:237}
We can prove the same non-extension result for other involutions on homology spheres, even those that are not the boundaries of contractible manifolds. For example, an extension~\cite{ohta:boundary} of Taubes' result \cite{taubes:periodic} (plus the fact~\cite{fs:237} that $\Sigma(2,3,7)$ bounds a spin manifold with intersection form $E_8 \oplus H$) implies that the homology sphere $\Sigma(2,3,7)\,\# -\Sigma(2,3,7)$ does not bound a smooth contractible manifold.  On the other hand, we can construct an involution on this manifold as follows. 
View $\Sigma(2,3,7)$ as the link of a singularity,
\[
\Sigma (2,3,7)\,=\,\{(x,y,z)\in \mathbb{C}^{3}\mid x^{2}+y^{3}+z^{7}=0,\ |x|^{2}+|y|^{2}+|z|^{2}=1\},
\]
and consider the involutions $\tau_0$ and $\tau_1$ acting on $\Sigma(2,3,7)$ by the formula 
\begin{equation}\label{two involutions}
\tau_{0}(x,y,z)=(-x,y,z)\quad\text{and}\quad\tau_{1}(x,y,z)=(\bar{x},\bar{y},\bar{z}).
\end{equation}
Let $\tau^{*}_{i}$ denote the map on  $HM^{\red}(\Sigma(2,3,7);\mathbb{Q}) = \mathbb{Q}$ induced by $\tau_{i}$, $i = 0, 1$. The involution $\tau_0$ is isotopic to the identity hence $\tau^*_0$ is the identity; the action of $\tau^*_1$ is computed in Section \ref{S:contact} below as negative one. Suppose $\tau = \tau_0 \,\# \tau_1$ extends as a diffeomorphism on some $\Z/2$ homology ball with boundary $\Sigma(2,3,7)\,\# -\Sigma(2,3,7)$. Adding a $3$-handle results in a $\Z/2$ homology cobordism $W$ from $\Sigma(2,3,7)$ to itself that admits a self-diffeomorphism restricting to $\tau_0$ and $\tau_1$ on its two boundary components. By functoriality of monopole Floer homology, $W$ induces trivial map on $\hmred(\Sigma(2,3,7))$. This contradicts the splitting formula for $\lambda_{SW}$ \cite[Theorem A]{LRS} and the fact that it reduces mod 2 to the Rohlin invariant \cite[Theorem A]{MRS1}.
\end{rmk}


\subsection{Constructing corks}
Starting with the cork $W_1$, one can construct a number of other corks by the method we describe in this subsection. Recall that the involution $\tau: \partial W_1 \to \partial W_1$ makes $\partial W_1$ into a double branched cover $\Sigma(K_1)$ of the 3--sphere with branch set the knot $K_1 \subset S^3$ shown in Figure \ref{F:twist}. Let $K$ be an arbitrary knot in $S^3$ smoothly concordant to $K_1$. The double branched cover of $I\times S^3$ with branch set the concordance is a $\Z/2$ homology cobordism $U_K$ from $\partial W_1 = \Sigma(K_1)$ to $\Sigma(K)$. The manifold 
\[
W_{K} =W_1\cup_{\partial W_1} U_{K}
\]
is a smooth $\Z/2$ homology 4-ball with the natural involution $\tau_K: \partial W_K \rightarrow \partial W_{K}$ on its boundary given by the covering translation. 

\begin{cor}\label{C: new cork} 
The involution $\tau_{K}$ can be extended to $W_{K}$ as a homeomorphism but not as a diffeomorphism. Moreover, if $\pi_1(U_{K})$ is normally generated by the image of $\pi_1(\partial W_1)$ then the manifold $W_K$ is contractible and therefore $(W_K,\tau_K)$ is a cork.
\end{cor}

\begin{proof} 
The involution $\tau_{K}$ extends as a homeomorphism because $\tau$ does. To prove that $\tau_{K}$ does not extend as a diffeomorphism, consider the $\Z/2$ homology ball 
\[
W\, =\, W_{K}\cup_{\Sigma(K)}\,(- U_K)
\]
with boundary $Y_1$, where $- U_K$ denotes $U_K$ with reversed orientation. Suppose $\tau_{K}$ extends as a diffeomorphism on $W_{K}$. By gluing this diffeomorphism with the covering translation on $- U_K$, we obtain a diffeomorphism on $W$ that extends the involution $\tau$ on its boundary. This contradicts Theorem \ref{T:strong}.  \end{proof}

Examples of knots $K$ which are concordant to $K_1$ and, at the same time, satisfy the condition of Corollary \ref{C: new cork} can be constructed using the technique of infection~\cite{cochran-orr-teichner:l2}. Choose a knot $\eta$ in the complement of $K_1$ that is unknotted in $S^3$ and has even linking number with $K_1$. Let $J$ be any slice knot in $S^3$. Denote by $\nu(\eta)$ and $\nu(J)$ open tubular neighborhoods of the two knots. Then 
\[
(S^3 - \nu(\eta)) \cup (S^3 - \nu(J))
\]
is diffeomorphic to $S^3$, provided we glue the meridian of $J$ to the longitude of $\eta$, and vice versa. Under this diffeomorphism, the knot $K_1$ becomes a new knot, $K(J,\eta)$.  

One can similarly `infect' the product concordance from $K_1$ to itself by removing $I \times \nu(\eta)$ from $I \times S^3$ and gluing in the exterior of a concordance $C \subset I \times S^3$ from the unknot to $J$; see Gordon~\cite{gordon:contractible}. This gives a concordance $C(J,\eta)$ from $K_1$ to $K(J,\eta)$. Writing $U_{K(J,\eta)}$ for the double branched cover of $I\times S^3$ with branch set $C(J,\eta)$, we claim that $U_{K(J,\eta)}$ is a  $\Z/2$ homology cobordism from $\partial W_1 = \Sigma(K_1)$ to $\Sigma(K(J,\eta))$ whose fundamental group is normally generated by $\pi_1(\Sigma(K_1))$.

To see this, note that by the assumption on the linking number, the preimage of the cylinder $I \times \eta$ in $I \times \Sigma(K_1)$ consists of two cylinders, $I \times \eta_1$ and $I \times \eta_2$. Therefore,
\[
U_{K(J,\eta)} = ((I \times \Sigma(K_1)) - (I \times \nu(\eta_1)) - (I \times \nu(\eta_2))) \cup \left((I \times S^3) - \nu(C)\right) \cup \left((I \times S^3) - \nu(C)\right).
\]
In this identification, the longitude for each copy of $C$ is glued to the corresponding meridian of $\eta_1$ or $\eta_2$. Since the bottom of $C$ is an unknot, this means that the group $\pi_1 (U_{K(J,\eta)})$, computed via van Kampen's theorem, is normally generated by $\pi_1 (\Sigma(K_1))$ and two copies of $\pi_1((I \times S^3) - \nu(C))$.  But the meridians of the two copies of $C$, which normally generate $\pi_1((I \times S^3) - \nu(C))$, are the longitudes of $I \times \eta_1$ and $I \times \eta_2$. Since these are in $\pi_1(\Sigma(K_1))$, it follows that $\pi_1(\Sigma(K_1))$ normally generates $\pi_1(U_{K(J,\eta)})$.

One can also construct concordances to which Corollary \ref{C: new cork} would apply by replacing a tangle in $K_1$ with one that is concordant to it;  see Kirby--Lickorish \cite{kirby-lickorish:prime} and Bleiler \cite{bleiler:prime}. As we mentioned in the introduction, the corks are usually detected with the help of an effective embedding. A good example illustrating this point would be the corks constructed in~\cite{akbulut-ruberman:absolute} using a similar trick with invertible homology cobordisms. However, this is not how the corks in Corollary \ref{C: new cork} are detected: there does not seem to exist an obvious effective embedding that would help detect them.


\subsection{A re-gluing formula}
The above calculation of the induced action of $\tau$ on monopole Floer homology allows us to determine the effect of cutting and gluing along the homology sphere $Y_1$ via $\tau$ in a more general situation. 

\begin{thm}\label{invariance}
Let $Y_1$ be the manifold with involution $\tau$ shown in Figure \ref{F:cork-symm}\,(b), and $X$ a smooth closed oriented 4-manifold with $b^{+}_2(X) > 1$ decomposed as $X = X_1\cup X_2$ with $b^+_2 (X_2) > 1$ and $\partial X_2 =Y_1$. Let $X^{\tau}$ be the manifold obtained by cutting $X$ open along $Y_1$ and regluing using $\tau$. Then
$$
SW(X,s)=(-1)^{b_1(X_1)+b^{+}_2(X_1)}SW(X^{\tau},s^{\tau}).
$$
\end{thm}

\begin{proof}[Proof of Theorem \ref{invariance}] 
We wish to apply the gluing formula of Section \ref{S:inv} to $Y = -Y_1$ (note that the orientation convention for $Y_1$ in the above theorem is opposite of that in Section \ref{S:inv}). The key to doing that are the following two observations:
\begin{itemize}
\item[(1)] Write $M_{i}=X_{i}- \inte(B^{4})$ then the absolute $\Z/2$ grading of $\widehat{HM}^{*}(M_1,s_1)(\hat1)$ is equal to $b_1(X_1)+b^{+}_2(X_1)+1\pmod 2$. 
\item[(2)] The isomorphisms \eqref{E:dual} and the formula \eqref{E:tau} imply that $\tau^{*}$ acts as identity on $\widehat{\HM}_{\text{odd}} \allowbreak (-Y_1)$ and minus identity on $\widehat{\HM}_{\text{even}}(-Y_1)=\HM_{\red}(-Y_1)$.
\end{itemize}
Now, if $b^{+}_2(X_2) > 1$, the result follows from Proposition \ref{gluing theorem}. If $b^+_2 (X_2) = 1$ then both manifolds $X_1$ and $X_2$ in the splitting $X = X_1 \cup X_2$ have positive $b^+_2$ and the result follows from the pairing formula \cite[Equation 3.22]{Kronheimer-Mrowka}.  
\end{proof}

\begin{cor}
In the situation of Theorem \ref{invariance}, twisting the manifold $X$ along $Y_1$ via the involution $\tau$ can only kill the Seiberg--Witten invariant of $X$ when the piece bounded by $Y_1$ is negative definite.
\end{cor}
In particular, if $X_1 = -W_1$, the cork twist cannot change the Seiberg--Witten invariant of $X$. This is perhaps more readily seen via the blow-up formula for the Seiberg--Witten invariants, using the fact (which is implicit in~\cite{akbulut-kirby:mazur}) that the cork twist extends over $W_1\,\#\,\cptwo$.



\section{Knot concordance and Khovanov homology thin knots}\label{S:thin} In this section, we prove the results in Section \ref{S:concordance} from the introduction. We start with the following lemma, which is presumably well known.

\begin{lem}\label{L: double branched cover of thin knot}
Let $L$ be a ramifiable link in the 3-sphere that is Khovanov homology thin over $\F_2$. Then $\Sigma(L)$ is a monopole $L$-space over the rationals, that is, 
\[
\hmred(\Sigma(L);\mathbb{Q})\,=\,0.
\]
\end{lem}

\begin{proof}
Let us fix an orientation on the link $L$. According to Bloom \cite{bloom: spectral sequence}, there is a spectral sequence whose $E_2$ page is $\widetilde{\KH}(L;\F_2)$ and which converges to $\widetilde{\HM}(-\Sigma(L);\F_2)$ (we refer to \cite[Section 8]{bloom: spectral sequence} for the definition of this  tilde-version of monopole Floer homology). In particular, this implies that
\begin{equation}\label{rank inequality}
\dim_{\F_2}(\widetilde{\KH}(L;\F_2))\;\geq\; \dim_{\F_2}(\widetilde{\HM}(-\Sigma(L);\F_2)).
\end{equation}
Recall from \cite{khovanov I} that the reduced Khovanov cohomology categorifies the Jones polynomial $J_{L}$. Together with the Khovanov homology thin condition, this implies that
\begin{equation}\label{rank equality}
\dim_{\F_2}(\widetilde{\KH}(L;\F_2)) = |J_L (-1)| = |H_1(\Sigma(L);\Z)|.
\end{equation}
Combining \eqref{rank inequality} and \eqref{rank equality} with the universal coefficient theorem, we obtain 
\begin{equation}\label{bounds on HM-tilde}
|H_1(\Sigma(L);\Z)|\; \geq\; \dim_{\F_2}(\widetilde{\HM}(-\Sigma(L);\F_2))\; \geq\; \dim_{\mathbb{Q}}(\widetilde{\HM}(-\Sigma(L);\mathbb{Q})).
\end{equation} 
By the definition of $\widetilde{\HM}$, one has $$\dim_{\mathbb{Q}}(\widetilde{\HM}(-\Sigma(L),\mathfrak{s};\mathbb{Q}))\geq 1$$
for any spin$^c$ structure $\mathfrak{s}$, with equality holding if and only if $\hmred(-\Sigma(L),\mathfrak{s};\mathbb{Q})=0.$
Therefore, \eqref{bounds on HM-tilde} implies that $-\Sigma(L)$ is a monopole $L$--space over the rationals. By the duality in the reduced monopole homology, $\Sigma(L)$ is a monopole $L$--space over the rationals as well.
\end{proof}

\begin{proof}[Proof of Corollary \ref{C: froyshov for thin knots}]This is a direct consequence of Lemma \ref{L: double branched cover of thin knot} and Theorem \ref{T:main}.
\end{proof}

\begin{proof}[Proof of Corollary \ref{C: Lefschetz concordance invariant}]
In the case of a knot $K$, the Murasugi signature of $K$ equals its usual signature, and the double branch cover $\Sigma(K)$ has a unique spin structure. Therefore, Theorem \ref{T:main} reduces to the statement that
\[
L(K)\; =\; \frac 1 8\, \sigma(K)\, -\, h(\Sigma(K)),
\]
where $\sigma(K)$ and $h(\Sigma(K))$ are additive concordance invariants. This makes $L(K)$ into an additive concordance invariant. This invariant is non-trivial: for example, if  $K$ is the right handed $(3,7)$-torus knot, $\sigma(K) = -8$ and $h(\Sigma(K)) = h(\Sigma (2,3,7)) = 0$, hence $L(K) = -1$.
\end{proof}

\begin{proof}[Proof of Corollary \ref{C: direct summand}]
This is immediate from Corollary \ref{C: Lefschetz concordance invariant}.
\end{proof}


\section{Monopole Contact invariant} \label{S:contact}
In this section we will prove Theorem \ref{T: contact sign}. Consider the Brieskorn homology sphere $Y = \Sigma(2,3,7)$, along with the involution $\tau_1(x,y,z)=(\bar{x},\bar{y},\bar{z})$ described in Remark \ref{R:237}.  We give $Y$ the canonical orientation as a link of singularity.
By combining the calculation of Heegaard Floer homology in \cite{oz: plumbed} with the identification between Heegaard Floer and monopole Floer homology, we obtain
\[
\widecheck{\HM}(Y;\Z)=\HF^{+}(Y;\Z)=\mathcal{T}_{(0)}\oplus \Z_{(-1)} 
\quad \text{and} \quad
\hmred(Y)=\widecheck{\HM}_{(-1)}(Y;\Z)=\Z_{(-1)}.
\]
To determine the induced action of $\tau$ on $\hmred(Y)$ we will use the fact that $\tau$ makes $Y$ into a double branched cover of the 3-sphere with branch set the $(2,-3,-7)$ pretzel knot $K$, pictured as the Montesinos knot $K(2,-3,-7)$ in Figure \ref{F:fig11}. Either by a direct calculation, or by using the formula of \cite[Section 7]{Saveliev} for the knot signature in terms of the $\bar\mu$--invariant, we obtain 
\[
\sigma(K) = 8\,\bar{\mu}(\Sigma(2,3,7)) = 8.
\]
Theorem \ref{T:main} then tells us that
\[
\Lef(\tau_{*})\, =\, \frac 1 8\,\sigma(K) - h(Y)\, =\, 1
\]

\medskip\noindent
and therefore $\tau_*: \hmred(Y)\to \hmred(Y)$ is negative identity. 

According to \cite{Tosun} (see also \cite[Theorem 1.6]{Mark-Tosun}), the manifold $-Y$ admits a unique (up to isotopy) tight contact structure $\xi$, which is Stein fillable and has Gompf invariant $\theta=2$. By applying the involution $\tau$, we obtain another contact structure $\tau^{*}(\xi)$. Since $\tau^{*}(\xi)$ is also tight, it must be isotopic to $\xi$. 

Now, suppose there is a canonical choice of the contact element $\tilde{\psi}(-Y,\xi) \in \widecheck{\HM} (Y;\Z)$. This element is non-zero and it is supported in degree $-(\theta+2)/4 = -1$. Since $\tau_*$ acts as negative identity on $\hmred(Y)=\widecheck{\HM}_{(-1)}(Y;\Z)$, we have 
\[
\tilde{\psi}(-Y,\xi)\neq -\tilde{\psi}(-Y,\xi)=\tau_{*}(\tilde{\psi}(-Y,\xi))=\tilde{\psi}(-Y,\tau_{*}(\xi)).
 \]
However, this contradicts the naturality of the contact invariant because $\tau_*(\xi)$ and $\xi$  are isotopic.  


\end{document}